\documentclass[pdflatex,sn-mathphys-num]{sn-jnl}

\usepackage{graphicx}%
\usepackage{multirow}%
\usepackage{amsmath,amssymb,amsfonts}%
\usepackage{amsthm}%
\usepackage{mathrsfs}%
\usepackage[title]{appendix}%
\usepackage{xcolor}%
\usepackage{textcomp}%
\usepackage{manyfoot}%
\usepackage{booktabs}%
\usepackage{algorithm}%
\usepackage{algorithmicx}%
\usepackage{algpseudocode}%
\usepackage{listings}%
\usepackage{xcolor} 
\usepackage{mathtools}

\usepackage[normalem]{ulem}
\numberwithin{equation}{section}
\numberwithin{table}{section}
\numberwithin{figure}{section}
\usepackage{enumitem}

\newcommand{\be}{\begin{equation}}
\newcommand{\ee}{\end{equation}}

\renewcommand{\epsilon}{\varepsilon}
\renewcommand{\phi}{\varphi}
\DeclareMathOperator{\sign}{sign}

\newcommand{\cE}{\mathcal{E}}
\newcommand{\cO}{\mathcal{O}}
\newcommand{\cR}{\mathcal{R}}
\newcommand{\cS}{\mathcal{S}}
\newcommand{\N}{\mathbb{N}}  
\newcommand{\Z}{\mathbb{Z}}  
\newcommand{\R}{\mathbb{R}}  

\theoremstyle{thmstyleone}%
\theoremstyle{thmstyletwo}%

\theoremstyle{thmstylethree}%
\newtheorem{thm}{Theorem}[section]
\newtheorem{prop}[thm]{Proposition}
\newtheorem{cor}[thm]{Corollary}

\newtheorem{defn}[thm]{Definition}

\begin{document}

\title[Seasonal Forcing Dominated Dynamics of a
piecewise smooth
Ghil-Zaliapin-Thompson ENSO model]{\large\bf Seasonal Forcing Dominated Dynamics of a piecewise smooth Ghil-Zaliapin-Thompson ENSO model}


\author[1]{\fnm{Samuel} \sur{Bolduc-St-Aubin}}

\author[2]{\fnm{Antony R.} \sur{Humphries}}

\affil*[1]{\orgdiv{Department of Mathematics}, \orgname{University of Auckland}, 
\orgaddress{Private Bag 92019, Auckland 1142, New Zealand}}

\affil[2]{\orgdiv{Department of Mathematics and Statistics, and, Department of Physiology}, 
\orgname{McGill University}, 
\orgaddress{Montreal, QC H3A 0B9, Canada}}





\abstract{The Ghil–Zaliapin–Thompson (GZT) model, a scalar delay differential equation with periodic forcing and time-delayed feedback, captures key features of the El Ni\~no–Southern Oscillation (ENSO) phenomenon. Numerical studies of the GZT model have revealed stable period-one orbits under strong forcing and locked, quasiperiodic, or even chaotic regimes under weaker forcing, but its analytical treatment remains challenging. To bridge this gap, we propose a piecewise smooth version of the GZT model 
with piecewise constant delayed feedback and continuous periodic forcing.
For this piecewise smooth GZT model we {explicitly} construct 
solutions of initial value problems, and study the existence and properties of 
periodic orbits of period one. By studying the symmetries and possible phases of periodic solutions we are able to construct period-one solutions and the regions of parameter space in which they exist. We show that the stability of these orbits is governed by a linear mapping from which we find the Floquet multipliers for the periodic orbit and also the bifurcation curve along which these orbits lose stability. We show that for most values of the delay this occurs at a torus bifurcation, but that for small delays a fold  bifurcation of period one orbits occurs. 
We then compare these analytical results with numerical continuation of the GZT model, 
showing that they align very closely. }




\maketitle

\section{Introduction}\label{sec1}
In this work we study the dynamics of the conceptual climate model
\be \label{eq:iGZT}
h'(t)=-\sign(h(t-\tau))+c\cos(2\pi t),
\ee
which we refer to as the \textit{piecewise smooth GZT} (psGZT) model. 
This model
is derived from the Ghil-Zaliapin-Thompson (GZT) model
\be \label{eq:GZT}
h'(t)=-\tanh({\kappa h(t-\tau)})+c\cos({2\pi t})
\ee
introduced in \cite{GZT2,GZT}. 

The GZT model \eqref{eq:GZT} has been extensively studied numerically \cite{GZT2,GZT,KKP15,KK18}.
Complicated dynamics arise, but over a relatively small range of parameters values. In \cite{KKP15}, features observed via numerical simulation in \cite{GZT2,GZT} are explained by using the continuation software \texttt{DDE-BifTool} \cite{sieber2014dde}.
The dynamics of the GZT model are strongly affected by the relative sizes of the seasonal forcing strength $c$ and the delay $\tau$. For large $c$ all stable invariant objects are period-one orbits, but for small $c>0$
quasi-periodic and chaotic solutions may arise.

Both the GZT model \eqref{eq:GZT} and the psGZT model \eqref{eq:iGZT} are constant delay, periodically forced delay differential equations (DDEs). 
In the GZT model \eqref{eq:GZT} the parameter $\kappa$ is usually taken to be large, indeed most authors consider $\kappa\geq11$
\cite{GZT,GZT2,KK18,KKP15,KKP16}. In the pointwise limit of large $\kappa$, the delayed feedback function becomes
\begin{equation}  
\label{eq:limit}
\lim_{\kappa\to\infty}\tanh{(\kappa h(\rho))}=\sign(h(\rho))\quad \forall \rho\in\mathbb{R},
\end{equation}
where we define $\sign(0) =0$.
Since with $\kappa=11$, the $\tanh$ function already resembles a step function,
this leads us to consider the limit \eqref{eq:limit} in the GZT model \eqref{eq:GZT}
and, hence, we study the psGZT model \eqref{eq:iGZT}.

In \cite{RKA20}, the dynamics of the signum forced psGZT model
\begin{equation} \label{eq:RKA}
h'(t)=-\sign(h(t-\tau))+c\sign(\cos(2\pi t))
\end{equation}
were studied.
In both \eqref{eq:iGZT} and \eqref{eq:RKA}, changing the delayed feedback term from $\tanh$ to $\sign$ is justifiable because the coupling strength $\kappa$ is usually considered to be large. However, the additional assumption in \eqref{eq:RKA} that the smooth seasonal periodic forcing can be replaced by a piecewise constant function does not follow from any parameter limit in the GZT model \eqref{eq:GZT} and is therefore questionable. Nonetheless, the beauty of this simplification is that it allowed solutions and bifurcations of \eqref{eq:RKA} to be constructed analytically in \cite{RKA20}.

The model \eqref{eq:GZT} was originally proposed in \cite{GZT} as a conceptual model of the the El Niño-Southern Oscillation (ENSO) climate phenomenon. The ENSO phenomenon is characterized by periodic fluctuations in sea surface temperatures (SSTs) and atmospheric pressure across the equatorial Pacific Ocean. ENSO primarily manifests in two phases: El Niño and La Niña. El Niño events are associated with warmer-than-average SSTs in the central and eastern Pacific, while La Niña events are characterized by cooler-than-average SSTs. There is also an atmospheric component to the phenomenon, where unusually high atmospheric pressure develops in the western Pacific during El Niño and low pressure during La Niña.

\begin{figure}
    \centering
    \includegraphics[width=\textwidth]{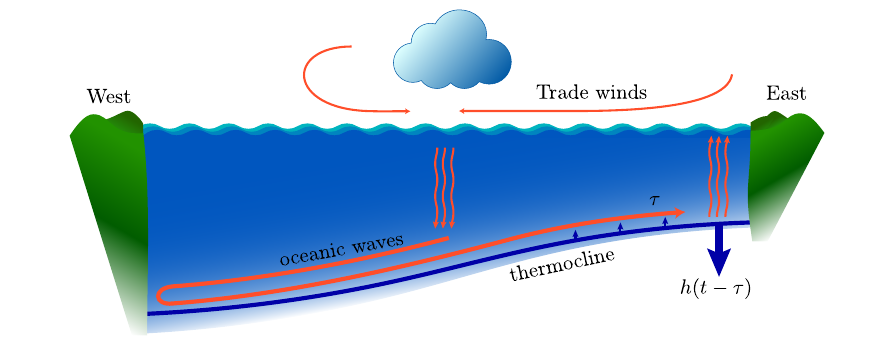}
    \caption{Schematic representation of the time-delayed negative feedback loop of the DAO paradigm.
    A positive perturbation $h$ of the thermocline depth near the eastern boundary at time $t-\tau$ leads to a weakening of the easterly trade winds, resulting in westerly anomalies. These anomalies induce a mass adjustment in the central Pacific, where it is argued that the atmosphere-ocean coupling is the strongest. Such a process generates westward-propagating Rossby waves. 
    Upon interacting with the western boundary, these waves become Kelvin waves, which cool the SSTs and terminate the warm phase at time $t$. }
    \label{fig:SS88}
\end{figure}

An important physical feature in the context of ENSO is the thermocline, a thin layer separating warm surface waters from deep cold waters. This layer supports large-scale oceanic waves, such as Kelvin and Rossby waves, which have been used to explain the oscillatory nature of ENSO. The delayed action oscillator (DAO) paradigm
was first proposed and modelled in \cite{SuarezSchopf1988} and
states that the propagation of such oceanic waves forms a negative feedback process, which is delayed due to the finite velocity of the waves. The GZT model \eqref{eq:GZT} follows from this DAO philosophy, incorporating both the time-delayed negative feedback and an idealised representation of the seasonal forcing. The state variable $h(t)$, a scalar, represents the deviation of the thermocline depth in the eastern tropical Pacific Ocean from its mean value. 
A positive value of  $h$ indicates a deeper thermocline in the east, which reduces upwelling and thus warms the SSTs, and vice versa.
The delay $\tau$ represents the time taken for these waves to traverse the Pacific Ocean. This time-delayed negative feedback process is depicted in Figure~\ref{fig:SS88}.

Conceptual models play a fundamental role in climate science by isolating the key mechanisms that govern complex climate dynamics, while remaining amenable to mathematical analysis (see Section 2 of \cite{DIJKSTRA2024133984} for a discussion of the hierarchy of climate models). 
 DDEs form an especially powerful class of conceptual models, as they naturally represent time-delayed feedbacks intrinsic to many climate subsystems. The review \cite{keane2017climate} highlights how DDE-based climate conceptual models can be used to study a range of phenomena.

In the context of ENSO, several other DDE conceptual models have been proposed to capture different aspects of the delayed ocean–atmosphere coupling explained by the DAO.
For works related to the GZT model \eqref{eq:GZT}, see \cite{GZT2,GZT,KK18,KKP16}.
The Tziperman, Stone, Cane and Jarosh (TSCJ) model 
is an extension of the GZT model that incorporates an asymmetric coupling function and a positive delayed feedback \cite{KKP16, TSCJ2, TSCJ94}.
A state-dependent delay version of the GZT model has also been proposed in \cite{KKD19}.
There are other DAO models for ENSO, such as the Suarez–Schopf model \cite{SuarezSchopf1988, Courtney2019, boutle2007nino}, which contain different feedback mechanisms.
Moreover, DDE-based conceptual frameworks have been applied beyond ENSO, for example, to describe the feedback loops in the Atlantic Meridional Overturning Circulation (AMOC) \cite{wei2022simple}.

In the current work, we study \eqref{eq:iGZT} as an intermediate model between the smooth GZT model \eqref{eq:GZT} and the model \eqref{eq:RKA}, and show how to analytically construct periodic solutions of \eqref{eq:iGZT} with smooth periodic forcing. This allows us to study the bifurcation structures of \eqref{eq:iGZT} and to compare and contrast them with the analytical results obtained in \cite{RKA20} and with the numerical computations of solutions of the GZT model \eqref{eq:GZT} provided in \cite{KK18,KKP15}.

In Section~\ref{sec:ivp}, we establish existence, uniqueness, and boundedness of solutions to \eqref{eq:iGZT} when posed as an initial value problem (IVP) with a suitable initial function. We also show that all solutions oscillate about zero, and provide a closed-form expression for the solution of the IVP.
In Section~\ref{sec:po}, we discuss the symmetries of 
periodic solutions of \eqref{eq:iGZT} and show that all period one solutions must be symmetric (in the sense of \eqref{eq:sym}).
We construct period-one orbits with a single upward crossing of zero per period, and (in Theorems~\ref{thm:po} and~\ref{thm:po2}) determine the region of parameter space in which these orbits exist.
In Section~\ref{sec:stab11}, we analyze the stability of these periodic orbits. Stability is determined by how the zeros of a perturbation of the periodic orbit evolve over time, and in Proposition~\ref{prop:cEformulation} we derive a linearized map that describes this and enables us in Theorem~\ref{thm:Poly} to derive the characteristic polynomial that determines the Floquet multipiers of the periodic orbit. 
We use this (in Theorem~\ref{thm:stabc}) to identify curves of torus and fold bifurcations where the period-one orbits lose stability. 
Finally, in Section~\ref{sec:compare} we compare these curves with the bifurcation curves of the smooth GZT model \eqref{eq:GZT}, and the signum-forced psGZT model \eqref{eq:RKA}.

\section{Initial Value Problems}
\label{sec:ivp}

Consider the psGZT model \eqref{eq:iGZT} posed as
an initial value problem (IVP):
\begin{equation}
\label{eq:iGZTIVP}
    \begin{cases}
      h'(t)=-\sign(h(t-\tau))+c\cos(2\pi t), \quad t>t_0,\\
      h(t)=\phi(t-t_0),\quad t\in[t_0-\tau,t_0],
    \end{cases}
\end{equation}
where $\varphi:[-\tau,0]\to\mathbb{R}$ is continuous with isolated zeros.

There is a comprehensive theory of smooth constant delay DDEs \cite{BellmanCooke63,Hale88,HaleLunel93,Smith11},
but piecewise smooth periodically forced problems are less well studied, and here we will provide a direct proof of the existence, uniqueness and boundedness of solutions of \eqref{eq:iGZTIVP}.

We will consider a function $h(t)$ to be a solution of the IVP \eqref{eq:iGZTIVP} if it is continuous for all $t\geq t_0$, satisfies the initial condition for all $t\in[t_0-\tau,t_0]$, and is differentiable and satisfies the differential equation for almost all $t>t_0$.  
We will construct solutions of \eqref{eq:iGZTIVP} with isolated zeros, accordingly
we restrict our attention to continuous initial functions 
$\varphi$ whose zero set is finite.
Regarding $h(t)$ as a solution
in this sense, the following theorem shows that solutions are bounded and oscillate.

\begin{thm} \label{thm:ivp}
The IVP \eqref{eq:iGZTIVP} has a unique solution which is continuous, remains bounded for all $t\geq t_0$ and crosses the zero-line infinitely many times.
\end{thm}

\begin{proof}
We solve the IVP iteratively on successive intervals of length $\tau$. For $n \geq 1$, define the interval:  
\[
I_n = [t_0 + \tau(n-1), t_0 + \tau n].
\]
On each interval $I_n$, the delayed term $h(t - \tau)$ is determined by the solution on the previous interval $I_{n-1}$. This reduces the DDE to a sequence of non-autonomous ODEs.

For $t\in I_1$, the term $h(t - \tau)$ depends on values of $h$ before $t_0$: $h(t - \tau)$ is determined by the initial function $\varphi$. It follows that $h(t-\tau)=\varphi(t-t_0-\tau)$. Then,
\[
h(t) = \varphi(0) + \int_{0}^{t-{t_0}} -\text{sign}(\varphi(\rho - \tau)) + c \cos(2\pi (\rho+t_0)) d\rho.
\]
Because $\varphi$ is bounded and measurable, the integrand is also bounded and measurable. Therefore, the integral exists in the Lebesgue sense, and the solution $h(t)$ exists on $I_1$.

Existence  on $[t_0,\infty)$ follows by applying the method of steps on subsequent intervals (see \cite{Smith11}). Assuming the solution exists on $I_{n-1}$, we use induction to show it exists on $I_n$.

Uniqueness follows from the following observation: Let $h(t)$ and $g(t)$ be solutions of the IVP \eqref{eq:iGZTIVP} with a bounded initial function $\varphi$. For $t \in I_1$, we compute:
\begin{align*}
h(t)-g(t) &= \int_{t_0}^t h'(\rho) d\rho - \int_{t_0}^t g'(\rho) d\rho \\
&= \int_{t_0}^t -\text{sign}(h(\rho-\tau)) + c\cos(2\pi \rho) + \text{sign}(g(\rho-\tau)) - c\cos(2\pi \rho) d\rho \\
&= \int_{t_0}^t \text{sign}(g(\rho-\tau)) - \text{sign}(h(\rho-\tau)) d\rho \\
&= \int_{t_0-\tau}^{t-\tau} \text{sign}(\varphi(\rho-t_0)) - \text{sign}(\varphi(\rho-t_0)) d\rho = 0.
\end{align*}
It then follows by induction that
$h(t) = g(t)$ for all $t \in I_n$, ensuring uniqueness.

Because $|{h'(t)}|\leq1+c$, solutions cannot become unbounded in finite time.

It remains only to show that every solution crosses the zero-line infinitely many times, and to do this we first show that the solution crosses the zero-line at least once.
For contradiction, assume that $h(t)>0$ for almost all $t\geq t_0$. Then
\begin{align} \notag
h(t) & =h(t_0)+\int_{t_0}^{t}\hspace*{-0.5em}-\sign{(h(\rho-\tau))}+c \cos{(2\pi \rho)}d\rho \\ 
&=\varphi(0)+\int_{t_0-\tau}^{t-\tau}\hspace*{-1em}-\sign{(h(\rho))}d\rho+\frac{c}{2\pi}\left( \sin{(2\pi t)}-\sin{(2\pi t_0)} \right) \notag\\ 
&\mbox{}\hspace*{-1em}=\varphi(0)+\!\int_{t_0-\tau}^{t_0}\hspace*{-1.2em}-\sign{(\varphi(\rho-t_0))}d\rho+\!\int_{t_0}^{t-\tau}\hspace*{-1.4em}-\sign{(h(\rho))}d\rho+\frac{c}{2\pi}\left( \sin{(2\pi t)}-\sin{(2\pi t_0)} \right)\!. 
\label{eq:inftyzero_B}
\end{align}
Observing that $\varphi(t-t_0)$ is bounded over the compact set $t\in[t_0-\tau,t_0]$, we conclude that $-\sign{(\varphi(t-t_0))}\leq 1$. Thus, $\int_{t_0-\tau}^{t_0}-\sign{(\varphi(\rho-t_0))}d\rho\leq\tau$. Thus from \eqref{eq:inftyzero_B},
\begin{equation}
\label{eq:inftyzero_C}
h(t)  \leq  \varphi(0)+\tau+\int_{t_0}^{t-\tau}\hspace*{-1em}-\sign{(h(\rho))}d\rho+\frac{c}{\pi}.
\end{equation}
Because $h(t)>0$ for almost all $t\geq t_0$,
\begin{equation}
\label{eq:inftyzero_D}
0<h(t)  \leq \phi(0)+\tau+\frac{c}{\pi}+t_0-(t-\tau).
\end{equation}
For $t$ sufficiently large, \eqref{eq:inftyzero_D} contradicts that $h(t)$ is greater than zero for almost all $t_0\geq t$. Because the nonlinearity is piecewise constant and the forcing is sinusoidal, any zero where a solution does not change sign will be isolated. Hence, $h(t)$ must cross the zero line at least once. The case $h(t)<0$ for almost all $t\geq t_0$ is similar. 

Let $t=z_1$ be the first point for $t>t_0$ 
where the solution $h(t)$ changes sign, then apply the same argument for $t>z_1$ with 
initial function $\{h(t):t\leq z_1\}$, and repeat sequentially to find $z_l$ for all $l=1,2,\ldots$ where the solution changes sign.

 Applying \eqref{eq:inftyzero_D} with $\phi(0)=0$ shows that after the first zero subsequent zeros are separated by at most $2\tau+\tfrac{c}{2\pi}$ time units, and this together with the bound $|h'(t)|\leq 1+c$ ensures that $h(t)$ is uniformly bounded.
\end{proof}

We use the term \emph{nondegenerate zero} 
to denote
an isolated zero of $h(t)$ where the solution changes sign. We define the set of nondegenerate zeros of $h(t)$: 
\begin{equation}
\label{eq:Z}
\mbox{}\hspace*{-0.2em}
Z = \left\{ z \in \mathbb{R} \,\Bigg|\, h(z) = 0 \ \text{and} \ \exists\, \delta > 0 \ \text{such that} \ \begin{aligned} 
&h(z_-)h(z_+) < 0 \ \text{for all} \\ 
&z_- \in (z - \delta, z), \ z_+ \in (z, z + \delta) 
\end{aligned} \right\}\!.
\end{equation}

If $h(z)=0$ and
$h$ is differentiable at $t=z$, then $h'(z)\ne0$ implies that $z\in Z$, and thus is a sufficient (but not necessary) condition for the zero to be nondegenerate. However, because of discontinuities on the right-hand side of \eqref{eq:iGZT}, in general $h(t)$ will not be smooth. We will mostly construct periodic orbits with nondegenerate zeros but will be interested in periodic orbits with degenerate zeros that arise as boundary cases.

We define two subsets of $ Z $, which will be relevant later. Let $ Z^+ $ denote the subset of $ Z $ corresponding to upward crossings, that is, where $ z \in Z $ such that $ h(z_-) < 0 $.
Similarly, $ Z^- =Z\backslash Z^+$ denotes the subset corresponding to downward crossings. Thus, 
$ Z $ is the disjoint union of $ Z^- $ and $ Z^+ $. Moreover, for any two consecutive elements in $ Z $, they cannot both belong to the same subset.

Theorem~\ref{thm:ivp} states that the ordered sequence of nondegenerate zeroes $\{z_l\}_{l\in\mathbb{N}}\subset Z$ of a solution of the IVP problem \eqref{eq:iGZTIVP} is infinite. The piecewise nonlinearity in \eqref{eq:iGZTIVP} allows for the analytical construction of solutions of \eqref{eq:iGZTIVP} as the nonlinear negative delayed feedback term $\sign(h(t-\tau))$ is now either $-1$ or $+1$ over intervals of the form $t\in[z_l+\tau,z_{l+1}+\tau)$ where $z_l<z_{l+1}$ are successive nondegenerate zeroes.

In the absence of accumulation points within
the set of nondegenerate zeros, 
the ordered sequence $\{z_l\}_{l\in\mathbb{N}}$ diverges. Otherwise $z_\infty := \lim_{l\to\infty} z_l<\infty$ is the first accumulation point. 
The so-called \emph{oscillation frequency number} was first  introduced in
\cite{JMP88}, and later named and extended to non-smooth systems in 
\cite{Shustin1995}.
It is defined as the number of nondegenerate zeros occurring within a time interval of length equal to the delay and preceding a given nondegenerate zero of the solution. 
When accumulation points exist, this number is infinite.
The study of oscillation frequency has been carried out in the context of discontinuous delayed systems related to \eqref{eq:iGZT}, where the periodic forcing is replaced by a bounded term of magnitude less than one.
Shustin \cite{Shustin1995} showed that, when the forcing amplitude remains below one, the oscillation frequency is non-increasing and becomes constant after a finite time. Subsequent works \cite{akian:1998, fridman2002, nussbaum2001nonexpansive} extended these results to more general classes of delayed systems.
These results imply that any solution with an initial function containing a finite number of zeros has a finite oscillation frequency. Such results can be applied to the psGZT model \eqref{eq:iGZT} when the forcing strength $c$ is smaller than one. Specifically, if the initial function $\varphi$ has a finite oscillation frequency and $c < 1$, then the solution will not contain any accumulation points in its zero set. In the current work we will exclusively study orbits with isolated zeros.

In Theorem~\ref{thm:ivpsol} we find a closed-form expression for the solutions of the IVP \eqref{eq:iGZTIVP} with an initial function that does not change sign over $[-\tau,0)$. In the case where $h(t)>0$ for $t\in(t_0,z_1) $, $\sign{(h(t-\tau))}=1$ for $t\in(t_0+\tau,z_{1}+\tau)$. Thus,
\begin{equation}
\label{eqn:zerodistriutionhh}
    \sign{(h(t-\tau))}=(-1)^l\quad\textnormal{for almost all }t\in(z_l+\tau,z_{l+1}+\tau).
\end{equation}

\begin{thm} \label{thm:ivpsol}
Let $h(t)$ be a solution of the IVP \eqref{eq:iGZTIVP} with initial function $\phi:[-\tau,0]\to\mathbb{R}$ satisfying
\begin{enumerate}
    \item $\phi(0)=0$;
    \item $\phi(t)<0$ \textnormal{ and continuous for all } $ t\in[-\tau,0)$.
\end{enumerate}
Then, the unique solution to the IVP is
\begin{equation}
\label{eq:closedformexpression}  
h(t)=p(t)+\frac{c}{2\pi} (\sin{(2\pi t)} -\sin{(2\pi t_0)}),
\end{equation}
where $p(t)$ is a continuous piecewise linear function. 

Assuming furthermore that
\begin{equation}
    \label{eq:1+cos}
    1+c\cos{(2\pi t_0)}>0\quad\text{or}\quad1+c\cos{(2\pi t_0)}=0\ \text{and}\ \sin{(2\pi t_0)}<0.
\end{equation}
Then,
the term $p(t)$ is given by
\begin{equation}
\label{eq:p(t)}
p(t)= \begin{dcases*}
t-t_0, & {\normalfont if} $ t\in[t_0,t_0+\tau]$\\
2\tau-t+t_0, & {\normalfont if} $ t\in[t_0+\tau,z_{1}+\tau]$\\
-2z_1+t+t_0, & {\normalfont if} $ t\in[z_{1}+\tau,z_{2}+\tau]$\\
a_l\tau+2\sum_{m=1}^{l-1}(-1)^{m}z_m+(-1)^{l}t+t_0, & {\normalfont if} $ t\in[z_{l-1}+\tau,z_l+\tau]$\textnormal{ for }$l>2$\end{dcases*}
\end{equation}
where $\{z_l\}_{l\in\mathbb{N}}$ are the nondegenerate zeroes of $h(t)$ and  $a_l=2$ when $l$ is even and $a_l = 0$ when $l$ is odd.
\end{thm}

\begin{proof}
From the proof of Theorem~\ref{thm:ivp}, we find that for any $t>t_0$,
\begin{equation*}
h(t) =  p(t)+\int_{t_0}^t c \cos(2\pi\rho) d\rho=p(t)+\frac{c}{2\pi} (\sin{(2\pi t)} -\sin{(2\pi t_0)}),
\end{equation*}
where $p(t)$ is a piecewise linear term given by
\[
p(t) = 
\begin{dcases}
 \int_{t_0-\tau}^{t - \tau}\hspace*{-1em}-\sign(\varphi(\rho-t_0)) \, d\rho, & \text{for } t \in [t_0, t_0 + \tau] \\
\int_{-\tau}^{0}\hspace*{-0.5em}-\sign(\varphi(\rho))\, d\rho
+ \int_{t_0}^{t - \tau}\hspace*{-1em}-\sign{(h(\rho))} \, d\rho, & \text{for } t > t_0+\tau.
\end{dcases}
\]
In particular, because $\varphi({t})<0$ for all $t\in[-\tau,0)$, it follows that
\[p(t)=\int_{t_0-\tau}^{t - \tau} -\sign(\varphi(\rho-t_0)) \, d\rho=\int_{t_0-\tau}^{t - \tau} +1 \, d\rho=t-t_0, \quad
\forall t \in [t_0, t_0 + \tau].
\]

Next, we show that condition \eqref{eq:1+cos} ensures that \( t_0 \) is a nondegenerate zero of \( h(t) \). 
To see this, first note that \( h(t) = \phi(t-t_0) < 0 \) for all \( t \in [t_0-\tau,t_0) \). Then 
the right derivative of \( h(t) \) at \( t_0 \) is given by 
\[ h'(t_0^+)=-\sign(\phi(-\tau))+ c \cos(2\pi t_0)=1 + c \cos(2\pi t_0). \]
Hence, the first condition in \eqref{eq:1+cos} ensures that
$h'(t_0^+)>0$, while the second condition in \eqref{eq:1+cos}
is equivalent to $h'(t_0^+)=0$ and $h''(t_0^+)>0$. 
Since, from above, \( h(t) \) is smooth for \( t\in(t_0,t_0+\tau)\),
in both cases, there exists \( \delta > 0 \) such that \( h(t) > 0 \) for all \( t \in (t_0, t_0 + \delta) \). 
It follows that \( t_0 \) is indeed a nondegenerate zero of \( h(t) \).

We have shown that the solution satisfies \eqref{eq:closedformexpression}
and \eqref{eq:p(t)} for $t \leq t_0+\tau$. 
We complete the result by induction.
Suppose the solution satisfies \eqref{eq:closedformexpression}
and \eqref{eq:p(t)} for $t \leq z_l+\tau$. 
Then for $t\in[z_{l}+\tau,z_{l+1}+\tau]$,
\begin{equation}
\label{eqn:h_0_to_tau_2}
h(t)=h(z_{l}+\tau)+\int_{z_{l}+\tau}^{t}-\sign{(h(\rho-\tau))}+c\cos{(2\pi \rho)}d\rho.
\end{equation}

It is important to realise that $t\in(z_{l}+\tau,z_{l+1}+\tau)$ implies
$t-\tau\in(z_{l},z_{l+1})$, and since the endpoints of this interval are defined by successive nondegenerate zeros, $\sign(h(t-\tau))$ is constant
for $t\in(z_{l}+\tau,z_{l+1}+\tau)$. Moreover, since $h(t)>0$ for $t\in(t_0,z_1)$, the solution satisfies \eqref{eqn:zerodistriutionhh}.
Thus we can evaluate the integral in \eqref{eqn:h_0_to_tau_2} yielding
\begin{equation*}
h(t)=h(z_{l}+\tau)+(-1)^{l+1} (t-\tau-z_l)+\frac{c}{2\pi}(\sin{(2\pi t)}-\sin{(2\pi (z_{l}+\tau))}).
\end{equation*}
Using the expression for $h(z_{l}+\tau)$ from \eqref{eq:closedformexpression} gives
\begin{equation}
\label{eqn:h_0_to_tau_2_a1}
h(t)=p(z_{l}+\tau)+(-1)^{l+1}(t-\tau-z_l)+\frac{c}{2\pi}(\sin{(2\pi t)}-\sin{(2\pi t_0)}).
\end{equation}
Let 
\begin{equation}
\label{eqn:indq}
q(t)=p(z_{l}+\tau)+(-1)^{l+1}(t-\tau-z_l).
\end{equation}
The induction step is complete if we show that $p(t)=q(t)$ for $t\in[z_{l}+\tau,z_{l+1}+\tau]$.
Using the expression for $p(z_{l}+\tau)$ from \eqref{eq:p(t)} gives
\begin{equation}
\label{eqn:indq2}
p(z_{l}+\tau)=(1+(-1)^{l-1})\tau+2\sum_{m=1}^{l-1}(-1)^{m}z_m
{+(-1)^l(z_l+\tau)+t_0}.
\end{equation}
Substituting this into \eqref{eqn:indq} and simplifying gives
\[q(t)=(1+(-1)^{l})\tau+2\sum_{j=1}^{l}(-1)^{m}z_m+(-1)^{l+1} t+t_0,\]
where we note that
\[(1+(-1)^{l-1})\tau+(-1)^l\tau-(-1)^{l+1}\tau=(1+(-1)^l)\tau=a_{l+1}\tau,\]
since $(-1)^{l-1}=(-1)^{l+1}$.
Then $q(t)=p(t)$ so the solution 
satisfies \eqref{eq:closedformexpression}
and \eqref{eq:p(t)} for $t \leq z_{l+1}+\tau$. 
This completes the induction.

The value of $p(z_l+\tau)$ is defined twice for each $l$ in \eqref{eq:p(t)}; it is easy to verify that these values agree so $p(t)$ is well-defined and continuous.
\end{proof}

If there is no accumulation point of the nondegenerate zeros then 
\eqref{eq:closedformexpression} and \eqref{eq:p(t)} define the solution $h(t)$ for all $t\geq t_0$. On the other hand, if
there is an accumulation point $z_\infty$ of the set of nondegenerate zeros $\{z_l\}_{l\in\mathbb{N}}$, the expressions
\eqref{eq:closedformexpression} and \eqref{eq:p(t)} remain valid for all $t$ in the interval $ [t_0, z_\infty + \tau) $, and cannot be extended beyond.

\begin{figure}
    \centering
    \includegraphics[width=\textwidth]{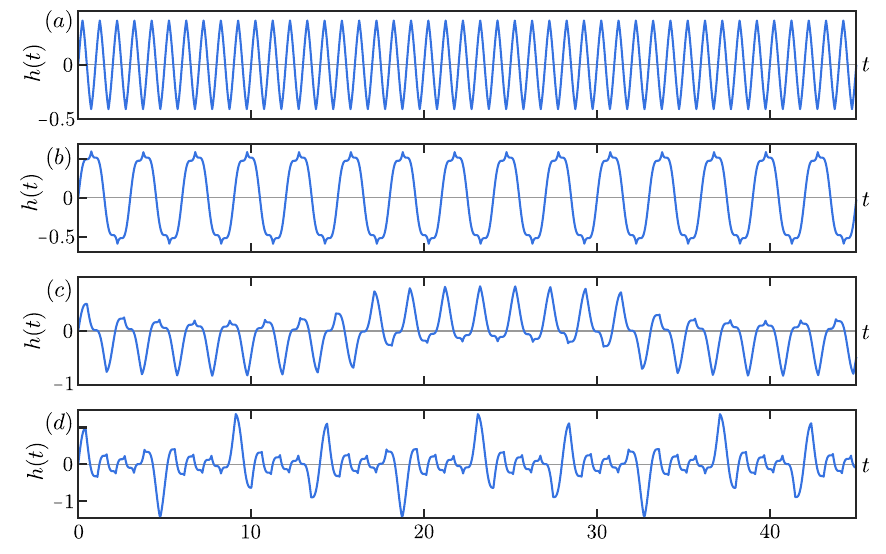}
    \caption{Solutions defined by \eqref{eq:closedformexpression} and \eqref{eq:p(t)} of the {psGZT} IVP \eqref{eq:iGZTIVP} with $c=1$ and (a) $\tau=0.25$, (b) $\tau=0.76$, (c) $\tau=0.508$ and (d) $\tau=0.42$. }
    \label{fig:ivpsol}
\end{figure}

We can evaluate \eqref{eq:closedformexpression} and \eqref{eq:p(t)}  in \texttt{MATLAB} \cite{MATLAB2024b} for an arbitrarily $t_0$ satisfying \eqref{eq:1+cos},
with the zeros calculated using \texttt{fzero}. Figure~\ref{fig:ivpsol} shows four trajectories {of the psGZT model \eqref{eq:iGZT}} 
with $c=1$ and different values of $\tau$. 
Figure~\ref{fig:ivpsol}(a) shows a periodic solution with period 1 dominated by seasonal forcing, while
(b)  depicts a trajectory converging toward a period-three orbit. 
In (c), we observe interdecadal variability in the mean amplitude of minima and maxima, while (d) shows a solution pattern that loses regularity, displaying large minima and maxima at irregular intervals ranging from three to seven years.

The solutions shown in 
Figure~\ref{fig:ivpsol}(c-d) 
use the same parameter values as those used in Figures~2(e) and~2(c) of \cite{GZT} 
for the smooth GZT model \eqref{eq:GZT} with $\kappa=50$.
The trajectories in Figure~\ref{fig:ivpsol}(c-d) for the psGZT model
closely match those reported in \cite{GZT} for the smooth GZT model, demonstrating a similarity between the solutions of 
these two models. After studying the properties of periodic orbits of the psGZT model \eqref{eq:iGZT} in Section~\ref{sec:po}, in Section~\ref{sec:compare} we will compare the bifurcations of these models.

\section{Periodic Orbits}
\label{sec:po}

Because the forcing in the psGZT model \eqref{eq:iGZT} has period one, any periodic orbit must have integer period $n\in\N$. In this work, the period $n$ is always the prime period. We are interested in both the case where the seasonal forcing dominates, resulting in stable period-one orbits, and in the loss of stability of these orbits leading to higher-period orbits and more complicated dynamics. In this section, we will study periodic orbits with odd period, and finitely many zeros. 
In Section~\ref{section:sym}, we show that certain solutions respect a symmetry property: $h(t)\mapsto -h(t+n/2)$, where $h$ satisfies \eqref{eq:iGZT} and $n$ is odd. An even stronger symmetry property relating solutions across the parameter space is also derived in Theorem~\ref{thm:sym2} and Corollary~\ref{cor:sym2}. Finally, Theorem~\ref{thm:po_closedform}
provides a closed-form expression for candidate 
period one orbits with one upward crossing of zero per period, which we will refer to as 1:1-period one orbits.

In Section~\ref{section:sym11}, we derive the
region of existence in parameter space
for these orbits. First in
Section~\ref{section:parameterconstraintsforperiodicorbits}  we derive parameter constraints for these orbits, 
before determining the existence region in
$(\theta,c)$ parameter space of 1:1 periond-one orbits
of \eqref{eq:iGZTth}
with phase $\alpha\in[-\tfrac14,\tfrac14]$
in 
Section~\ref{sec:existsym11pos14} 
and with phase $\alpha\in[\tfrac14,\tfrac34]$
in 
Section~\ref{section:existencesymmetric1434}.

\subsection{Properties of Periodic Orbits}
\label{section:sym}
To build intuition for the symmetries of \eqref{eq:iGZT}, 
we begin with the autonomous case with $c=0$, and consider the well-known autonomous system
\begin{equation}
 \label{eq:fautonomous}
    h'(t) = -f(h(t - \tau)),
\end{equation}
for some odd and continuous or piecewise continuous function $f:\mathbb{R}\to\mathbb{R}$. 

Setting the delay to $\tau = 1$, equation 
\eqref{eq:fautonomous} admits a period-4 solution that satisfies the symmetry $h(t) =-h(t + 2)$. This solution is sometimes referred to as the Kaplan and Yorke solution of \eqref{eq:fautonomous}, named after the pioneering work by Kaplan and Yorke \cite{kaplan1974ordinary}. 
Such a periodic orbit defines a slowly oscillating periodic solution, characterized by oscillations for which the time between zeros is greater than the delay interval. 
See \cite{cao1996uniqueness,chow1988characteristic,nussbaum1979uniqueness,kennedy2015multiple}
for other results concerning periodic solutions in \eqref{eq:fautonomous}.

When $f$ is given by the $\sign$ function, so that 
\eqref{eq:fautonomous} corresponds to \eqref{eq:iGZT} with $c=0$, the system admits a solution $h_0$ of period $4$:
\[h_0(t)=\begin{cases}
      t,\quad\quad\,-1\leq t\leq 1,\\
      2-t,\quad\, 1 \leq t\leq 3,
    \end{cases} \qquad h_0(t)=h_0(t+4k),\, k\in\mathbb{Z}.\]
Such an orbit has jump discontinuities in $h'_0$ whenever $t=\pm1+4k$, corresponding to nondegenerate zeros of $h_0(t-1)$. Aside from the orbit $h_0$, there is a countable set of periodic orbits $h_1,h_2,\dots$ with period $4/(4n+1)$, $n\geq 1$, given by
\[h_n(t)=\frac{1}{4n+1}h_0((4n+1)t),\quad t\in\mathbb{R}.\]
The solution $h_0(t)$ is stable, while all of the solutions $h_n(t)$ for $n\geq1$ are unstable \cite{fridman2002}. Any solution originating from an initial function with a finite number of zeros, is known to be 
equal, up to a time shift, after finite time
to a solution of the form $h_n$ \cite{fridman2002,Shustin1995}.

Next, we add the periodic forcing to \eqref{eq:fautonomous}, yielding the DDE
\be \label{eq:GZTgenf}
h'(t)=-f(h(t-\tau))+c\cos(2\pi t).
\ee
Notice that both the GZT model \eqref{eq:GZT} and the psGZT model \eqref{eq:iGZT} are of the form \eqref{eq:GZTgenf}.

The following theorem demonstrates that \eqref{eq:GZTgenf} possesses symmetrically related orbits.

\begin{thm} \label{thm:sym}
If $h(t)$ solves \eqref{eq:GZTgenf}, so does $g(t)=-h(t+n/2)$ for any odd integer $n$.
\end{thm}

\begin{proof}
Let $\xi=t+n/2$ and let $g(t)=-h(t+n/2)$ with $n$ odd.
Then,
\begin{align*} 
-f(g(t-\tau)) + c\cos(2\pi t)
&=-f(-h(\xi-\tau)) + c\cos(2\pi(\xi-n/2))\\
&=f(h(\xi-\tau)) - c\cos(2\pi\xi)=-h'(\xi)=g'(t). 
\end{align*} 

\vspace*{-3.5ex}
\end{proof}

As well as using the oddness of $f$, note that it is also essential for $n$ to be odd for the cosine term to have the correct sign in the proof of Theorem~\ref{thm:sym}. The theorem gives rise to two possibilities for $n$ odd. Each solution $h(t)$ of \eqref{eq:GZTgenf}
either exists as one half of a pair of solutions related by this symmetry, or the solution is itself symmetric with
\be
\label{eq:sym}
h(t)=-h(t+n/2).
\ee
However, any orbit that satisfies \eqref{eq:sym} for some $n$ odd is a periodic orbit of period $n$, since applying the symmetry twice gives
$$h(t)=-h(t+n/2)=(-1)^2h((t+n/2)+n/2))=h(t+n).$$

A result analogous to Theorem~\ref{thm:sym} can be established when the integer $n$ is even: if $h(t)$ is a solution of \eqref{eq:iGZT}, then $g(t) = -h(t + n/4)$ is also a solution. Despite this, for $n$ even it is not possible to have period-$n$ solutions with the symmetry $h(t)=-h(t+n/4)$ because by a similar argument it then follows that $h(t)=h(t+n/2)$, contradicting the assumption that $n$ is the prime period. Thus, any such periodic orbits occur in pairs. This symmetry explains the bistability observed in some resonance regions in \cite{KKP15} for the smooth GZT model \eqref{eq:GZT}.

We are interested in periodic orbits of \eqref{eq:iGZT}, with finitely many zeros. We will refer to such a periodic orbit as \emph{symmetric} if it has odd period $n$ and satisfies the symmetry \eqref{eq:sym}.
Following \cite{RKA20}, we further characterise period $n$ orbits by the number $s$ of (nondegenerate) upward crossings of zero which occur each period. In the current work, we will mainly be interested in periodic orbits dominated by seasonal forcing, which will have period $n=1$ with $s=1$ upward crossing per period, and, using the notation of \cite{RKA20}, we refer to such an orbit as a symmetric 1:1-periodic orbit. 
In the companion paper \cite{GZTn} we
will study symmetric $n$:$s$-periodic orbits for general odd $n$. 
It is essential not to confuse the notation used here with the notation typically used for a $p$:$q$-locked orbit. Phase locked periodic orbits will arise later in this paper, and will be denoted by $(p, q)$, to avoid confusion with the symmetric $n$:$s$-periodic orbits.

For an orbit of period $n$ of \eqref{eq:iGZT} we have
$$
0  =h(n)-h(0)=\!\int_0^n \! h'(t)dt
=\!\int_0^n \!\!\!-\sign(h(t-\tau))+c\cos(2\pi t)dt=-\!\int_0^n\! \!\sign(h(t-\tau))dt.
$$

Thus, since the zeros of the periodic orbit have measure $0$ on the real line, it follows that $\sign(h(t))$ is positive on exactly half of the orbit, and negative on the other half. For an $n$:1-periodic orbit we define the phase $\alpha\in\R$ by the location of the nondegenerate upward crossing of zero, so
\be\label{eq:halpha}
\alpha\in Z^+,
\ee
where $Z$ is defined by \eqref{eq:Z}. Then we must have
\be\label{eq:halphan}
\alpha+n\in Z^+\ \text{and}\ \alpha+\frac{n}{2}\in Z^-.\ee
This implies that, for such $n$,
\be \label{eq:hsigns}
\left\{
\begin{array}{ll}
h(t)>0, & \textit{a.e.~on } t\in(\alpha,\alpha+n/2),\\
h(t)<0, & \textit{a.e.~on } t\in(\alpha+n/2,\alpha+n),
\end{array} \right.
\ee
with
\be \label{eq:hzeros}
h(\alpha)=h(\alpha+n/2)=h(\alpha+n)=0.
\ee
Note that 
$h$ being strictly positive or negative on the two intervals in \eqref{eq:hsigns} only holds almost everywhere because of the possibility of isolated degenerate zeros, which we will encounter later in boundary cases.

For $n$ odd, $n$:1-periodic orbits of \eqref{eq:iGZT} are symmetric, as the following result shows.

\begin{prop}    \label{prop:n:1symmetric}
    Every $n$:1-periodic orbit of \eqref{eq:iGZT} with $n$ odd is symmetric. 
\end{prop}
\begin{proof}
    When $n$ is odd, $c\cos(2\pi(t+n/2))=-c\cos(2\pi(t))$, 
    and 
it follows immediately from \eqref{eq:hsigns} and \eqref{eq:hzeros}
that \eqref{eq:sym} is satisfied.
\end{proof}

When $h(t)$ is continuously differentiable at $t=\alpha$, the condition
\be \label{eq:halpha2}
h(\alpha)=0, \quad\text{and}\quad h'(\alpha)>0,
\ee
clearly implies that $ \alpha \in Z^+ $. 
However, because we are also interested in edge cases where $h$ is either not continuously differentiable at $t = \alpha$ or where $ h'(\alpha) = 0 $ while $ \alpha $  is still a nondegenerate zero of $h$, we use the more general condition \eqref{eq:halpha}.

Writing the delay $\tau$ as
$\tau=k+\theta$ for a non-negative integer $k$ and $\theta\in[0,1)$, for a periodic orbit of period 1 it follows that $h(t-\tau)=h(t-k-\theta)=h(t-\theta)$. Thus a 1:1-periodic orbit for \eqref{eq:iGZT} is also a
1:1-periodic orbit for
\be \label{eq:iGZTth}
h'(t)=-\sign(h(t-\theta))+c\cos(2\pi t).
\ee

We now relate the phase $\alpha$ of a symmetric 1:1-orbit to its forcing strength $c$ and the parameter $\theta$.
For convenience, we introduce the functions
\be \label{eq:uv}
{u}(\theta):=\pi\left(2\theta-\frac{1}{2}\right)\quad\text{and}\quad {v}(\theta):=\pi\left(\frac{3}{2}-2\theta\right),
\ee
where $u(\theta)=v(1-\theta)$.

\begin{prop}
\label{prop:alpha01}
The phase $\alpha\in\mathbb{R}$ of a symmetric 1:1-periodic solution satisfies
\be \label{eq:alpha0}
c\sin(2\pi\alpha)={u(\theta)}, 
\quad\text{if}\quad\theta\in[0,\tfrac{1}{2}),
\ee
and
\be \label{eq:alpha1}
c\sin(2\pi\alpha)={v(\theta)},
\quad\text{if}\quad\theta\in[\tfrac{1}{2},1).
\ee
\end{prop}

\begin{proof}
Such an orbit satisfies
\begin{align*}
0 &= h(\alpha+\tfrac{1}{2})   = h(\alpha)+\int_\alpha^{\alpha+\tfrac{1}{2}} h'(t)dt
= -\int_\alpha^{\alpha+\tfrac{1}{2}} \!\sign(h(t-\theta))dt
+c\int_\alpha^{\alpha+\tfrac{1}{2}}\cos(2\pi t)dt\\
&  = -\int_{\alpha-\theta}^{\alpha+\tfrac{1}{2}-\theta}\sign(h(t))dt  -\frac{c}{\pi}\sin(2\pi\alpha).
\end{align*}
Thus, the phase $\alpha$ of the periodic solution satisfies
\be \label{sin2pialph}
c\sin(2\pi\alpha)=-\pi\int_{\alpha-\theta}^{\alpha+\frac{1}{2}-\theta}\sign(h(t))dt.
\ee
Now, there are two cases to consider. If $\theta\in[0,\tfrac12)$ then
\be \nonumber
c\sin(2\pi\alpha)=-\pi\left[\int_{\alpha-\theta}^{\alpha}\hspace{-.5em}-1\,dt + \int_{\alpha}^{\alpha+\tfrac{1}{2}-\theta}\hspace{-1em}1\,dt\right]
=\pi\Bigl(\theta -(\tfrac{1}{2}-\theta)\Bigr)
=\pi(2\theta-\tfrac{1}{2}),
\ee
while if $\theta\in[\tfrac{1}{2},1)$
\[
c\sin(2\pi\alpha)=-\pi\left[
\int_{\alpha-\theta}^{\alpha-\tfrac{1}{2}}\hspace{-1em}1\,dt
+\int_{\alpha-\tfrac{1}{2}}^{\alpha+\tfrac{1}{2}-\theta}\hspace{-1.5em}-1\,dt\right]
=-\pi\Bigl((\theta -\tfrac{1}{2})-(1-\theta)\Bigr)
=\pi(\tfrac{3}{2}-2\theta).\]

\vspace*{-3ex}
\end{proof}

The following theorem establishes an additional symmetry of $n$:1-periodic orbits of \eqref{eq:iGZT},
which we will use through Corollary~\ref{cor:sym2} to relate the $n$:1-symmetric periodic orbits of \eqref{eq:iGZT}
for different delays.

\begin{thm} \label{thm:sym2}
Let $h(t)$ be a symmetric {$n$:1}-periodic orbit of \eqref{eq:iGZT} with $n$ odd, phase $\alpha\in\R$ and 
delay $\tau>0$.
Then $h^*(t)=-h(n-t)$ is a symmetric $n$:1-periodic orbit of \eqref{eq:iGZT} with 
phase $\alpha^*=-\alpha$ and delay $\tau^*=nk^*+n/2-\tau$ for any $k^*\in\N$.
\end{thm}

\begin{proof}
Let $h^*(t)=-h(n-t)$. Then using \eqref{eq:sym} we have
$$h^*(t)=-h(n-t)=h(3n/2-t).$$
Thus, since $h(t)$ is differentiable and satisfies the DDE \eqref{eq:iGZT} for almost every $t>0$, we have
\begin{align*}
(h^*)'(t)&=-h'(3n/2-t)=\sign(h(3n/2-t-\tau))-c\cos(2\pi(3n/2-t))\\
&=\sign(h(n-(t-nk^*-n/2+\tau)))-c\cos(3n\pi-2\pi t)\\
&=\sign(-h^*(t-nk^*-n/2+\tau))+c\cos(2\pi t)\\
&=-\sign(h^*(t-nk^*-n/2+\tau))+c\cos(2\pi t),
\end{align*}
and, hence, $h^*(t)$ satisfies the DDE \eqref{eq:iGZT}
for almost every $t>0$.

To complete the proof, it remains to show that $\alpha^*=-\alpha$  defines a phase of $h^*(t)$. Because $h(t)$ is a symmetric $n$:1-periodic orbit with phase $\alpha$, its zeros are at $t=\alpha+kn/2$ for all $k\in\Z$. Then
$$h^*(\alpha^*)=h^*(-\alpha)=-h(n+\alpha)=0,$$
as required.

To show that $\alpha^*\in Z_*^+$, where $Z_*^+$ denotes the upward nondegenerate zeros of $h^*(t)$. Choose $\delta>0$ sufficiently small and let
$z_-^*\in (\alpha^* - \delta, \alpha^*)$ 
and $z_+^*\in (\alpha^*,\alpha^* +\delta)$, and define
$z_+=-z_-^*$ and $z_-=-z_+^*$. 
Then
\begin{align*}
h^*(z_+^*) & = h^*(-z_-)=-h(n+z_-)=-h(z_-), \\
h^*(z_-^*) & = h^*(-z_+)=-h(n+z_+)=-h(z_+).
\end{align*}
It follows from $\alpha^*=-\alpha$, and $\alpha\in Z^+$ 
that for $\alpha>0$ 
\[h^*(z_+^*) >0 > h^*(z_-^*).\]

Thus, $\alpha^*\in Z_*^+$, 
which shows that $\alpha^*$ is the phase of $h^*(t)$.
\end{proof}

\begin{cor} \label{cor:sym2}
Let $h(t)$ be a symmetric 1:1-periodic orbit of \eqref{eq:iGZT} with delay 
$\tau=k+\theta$ where $\theta\in[0,1)$ and $k\in\mathbb{N}_0$. 
Then \eqref{eq:iGZT}
also has symmetric 1:1-periodic orbits with delay $\tau^*=k+\theta^*$ where $\theta^*\in[0,1)$ such that
\begin{enumerate}[label=(\alph*).]
    \item \(\theta^* = \frac{1}{2} - \theta \quad \text{for } \theta \in [0, \tfrac{1}{2}), \text{ ensuring } \theta^* \in (0, \tfrac{1}{2}],\)
    \item \(\theta^* = \frac{3}{2} - \theta \quad \text{for } \theta \in [\tfrac{1}{2}, 1), \text{ ensuring } \theta^* \in (\tfrac{1}{2}, 1].\)
\end{enumerate}
\end{cor}

\begin{proof}
Apply Theorem~\ref{thm:sym2} with $n=1$, $\tau=k+\theta$, and for 1. $k^*=2k$, and for 2. $k^*=2k+1$. 
\end{proof}

Corollary~\ref{cor:sym2}(a) implies that for each $\theta\in[0,\frac{1}{4})$ there is a corresponding
$\theta^*\in(\frac{1}{4},\frac{1}{2}]$ with $\theta^*=\frac{1}{2}-\theta$ such that \eqref{eq:iGZT} has a
1:1-periodic orbit with delay $\tau=k+\theta$, if and only if, it
has a 1:1-periodic orbit with delay $\tau^*=k+\theta^*$.
Corollary~\ref{cor:sym2}(b) implies a correspondence between the 
symmetric 1:1-periodic orbits for $\theta\in[\frac{1}{2},\frac{3}{4})$ and
$\theta^*\in(\frac{3}{4},1]$ with $\theta^*=\frac{3}{2}-\theta$. 
Moreover, Corollary~\ref{cor:sym2} implies that if $\theta=\tfrac14$ or
$\theta=\tfrac34$ then $\theta^*=\theta$. 

In the case where $\theta=\frac{1}{4}$ or $\theta=\frac{3}{4}$, Proposition~\ref{prop:alpha01} implies that the phase $\alpha$ satisfies $\alpha=0$ and $\alpha=\frac{1}{2}$ for both values of 
$\theta$. In this scenario, $\alpha = -\alpha^* \pmod 1$, which indicates that the associated orbits must satisfy the stronger symmetry property:
 \begin{equation}
    \label{eq:D4}
    h(t)=h(\tfrac{1}{2}-t).
\end{equation}

Figure~\ref{fig:sym2} illustrates symmetric 1:1 periodic solutions of \eqref{eq:iGZT}. Panel~(a) depicts two orbits with delays $\tau=k+\theta$, $\theta\in(\frac{1}{2},1)$ and $\tau^*=k+\frac{3}{2}-\theta$, respectively, whereas panel~(b) shows an orbit with $\theta=\tfrac14$. Such solutions trace continuous closed curves in the $(h(t),h(t-\tau))$-plane, which are invariant under a $180^\circ$ rotation. From Corollary~\ref{cor:sym2}, a $90^\circ$ rotation maps one closed curve to the other (red to blue and blue to red). The solution in panel~(b) satisfies the stronger symmetry property \eqref{eq:D4}: its closed curves over the $(h(t),h(t-\tau))$-plane are invariant under a rotation by $90^\circ$.

\begin{figure}
    \centering
    \includegraphics[width=\textwidth]{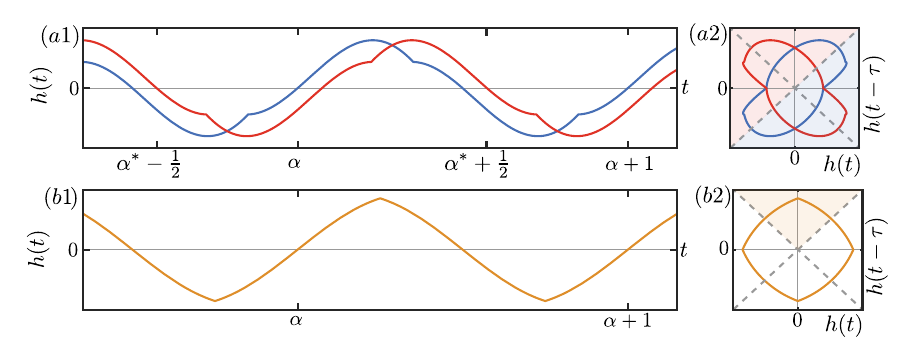}
    \caption{Symmetric 1:1-periodic solutions shown as time series in (a1)-(b1) and projected onto the $(h(t),h(t-\tau))$-plane in (a2)-(b2). 
    (a): 1:1-periodic solutions symmetrically related
    by Corollary~\ref{cor:sym2}(2),
    with $c=2.75$, and delay $\tau=k+\theta$, $\theta=0.8$, and phase 
    $\alpha$ in blue, and delay $\tau^*=k+\frac{3}{2}-\theta$, phase $\alpha^*=-\alpha$, in red. 
    (b): symmetric 1:1-periodic orbit with $\tau=k+\frac{1}{4}$ and $c=1.5$, satisfying the stronger symmetry property \eqref{eq:D4}.}
    \label{fig:sym2}
\end{figure}

The following theorem provides a closed-form expression for symmetric 1:1-periodic orbits, when such orbits exist. Showing that these candidate orbits actually satisfy the required conditions is non-trivial and will be tackled in Theorems~\ref{thm:po} and~\ref{thm:po2} in Section~\ref{section:sym11}.

\begin{thm} \label{thm:po_closedform}
Let $h(t)$ be a 1:1-periodic orbit 
of \eqref{eq:iGZT} with delay 
$\tau=k+\theta$ where $\theta\in[0,1)$, and $k\in\mathbb{N}_0$,
and 
phase \( \alpha \in[-\tfrac14,\tfrac34]\).
Then, such an orbit is symmetric, 
with the phase
defined by \eqref{eq:alpha0} for $\theta\in[0,\frac{1}{2})$,
and by \eqref{eq:alpha1} for $\theta\in[\frac{1}{2},1)$,
and satisfies
\be\label{eq:h11}
h(t)=p(t)+\frac{c}{2\pi}(\sin(2\pi t)-\sin(2\pi \alpha)),
\ee
where
\begin{align} \label{eq:p0}
p(t)&= \begin{dcases*}
-t+\alpha, & \textrm{for}\; $t\in[\alpha,\alpha+\tfrac{1}{2})$\\
-1+t-\alpha,\phantom{-2\theta} & \textrm{for}\;  $t\in[\alpha+\tfrac{1}{2},\alpha+1)$
\end{dcases*}\qquad\quad\textrm{if}\ \theta=0,\\   \label{eq:p0half}
p(t)&= \begin{dcases*}
t-\alpha, & \textrm{for}\; $ t\in[\alpha,\alpha+\theta)$\\
2\theta-t+\alpha, & \textrm{for}\; $ t\in[\alpha+\theta,\alpha+\theta+\tfrac{1}{2})$\\
-1+t-\alpha,\phantom{2f\theta} & \textrm{for}\; $ t\in[\alpha+\theta+\tfrac{1}{2},\alpha+1)$
\end{dcases*}\quad\;\textrm{if}\ \theta\in(0, \tfrac{1}{2}),\\  \label{eq:phalf}
p(t)& = \begin{dcases*}
t-\alpha, & \textrm{for}\; $ t\in[\alpha,\alpha+\tfrac{1}{2})$\\
1-t+\alpha,\phantom{-2f\theta} & \textrm{for}\; $ t\in[\alpha+\tfrac{1}{2},\alpha+1)$
\end{dcases*}\qquad\qquad\textrm{if}\ \theta=\tfrac{1}{2},\\  \label{eq:phalf1}
p(t)&= \begin{dcases*}
-t+\alpha, & \textrm{for}\; $ t\in[\alpha,\alpha+\theta-\tfrac{1}{2})$\\
1-2\theta+t-\alpha, & \textrm{for}\; $ t\in[\alpha+\theta-\tfrac{1}{2},\alpha+\theta)$\\
1-t+\alpha, & \textrm{for}\; $ t\in[\alpha+\theta,\alpha+1)$
\end{dcases*}\qquad\textrm{if}\ \theta\in(\tfrac{1}{2},1),
\end{align}
over a period $t\in[\alpha,\alpha+1)$.
\end{thm}

\begin{proof}
By Proposition~\ref{prop:n:1symmetric}, any such 1:1-periodic orbit is symmetric, from which it follows that \eqref{eq:hsigns} and \eqref{eq:hzeros} are satisfied. But for such an orbit, all the terms on the right-hand side of \eqref{eq:iGZTth} become functions of $t$ only, and the equation can then be integrated. It follows that the solution must be of the form
\eqref{eq:h11} with $p(t)$ defined by one of \eqref{eq:p0}-\eqref{eq:phalf1} depending on the value of $\theta\in[0,1)$.
\end{proof}

\subsection{Region of existence of symmetric periodic orbits}
\label{section:sym11}

In this section we analyze the parameter regions in the $(\theta,c)$-plane where symmetric 1:1-periodic orbits exist for  
\eqref{eq:iGZT}. These regions are derived from the conditions \eqref{eq:halpha}-\eqref{eq:hzeros}. We also present two theorems that characterize the solutions within such regions and discuss their behaviour at the regions' boundaries.

\subsubsection{Parameter constraints for periodic orbits}
\label{section:parameterconstraintsforperiodicorbits}
In general, \eqref{eq:alpha0} and \eqref{eq:alpha1} will each have two solutions for $\alpha \in [-\tfrac14,\tfrac34] \pmod 1$: one with $\alpha \in [-\frac{1}{4}, \frac{1}{4}] \pmod 1$ and another with $\alpha \in [\frac{1}{4}, \frac{3}{4}] \pmod 1$. However, we will show below that 
there can only be two valid solutions if $c$ is small and
$\theta\in(0,\tfrac12)$. When $\theta\in[\tfrac12,1]$ and/or when $c\gg0$ only the solution of \eqref{eq:alpha0} and \eqref{eq:alpha1} 
with $\alpha \in [-\frac{1}{4}, \frac{1}{4}] \pmod 1$ will yield  a valid periodic orbit.

Recall the functions $u(\theta)$ and $v(\theta)$ defined by
\eqref{eq:uv}.
Because $|\sin(2\pi\alpha)|\leq1$, Proposition~\ref{prop:alpha01} imposes necessary conditions for the 1:1-periodic orbits to exist:
\be \label{eq:necc}
\left\{
\begin{array}{ll}
c^2\geq {u}^2(\theta), & \theta\in[0,\frac{1}{2}),\\
c^2\geq {v}^2(\theta), & \theta\in[\frac{1}{2},1).
\end{array} \right.
\ee
But it also follows from \eqref{eq:iGZTth} that
\be \label{eq:hdashalph}
h'(\alpha)=-\sign(h(\alpha-\theta))+c\cos(2\pi\alpha)
=\left\{
\begin{array}{ll}
1+c\cos(2\pi\alpha), & \theta\in(0, \frac{1}{2}),\\
-1+c\cos(2\pi\alpha), & \theta\in(\frac{1}{2},1).
\end{array} \right.
\ee

We remark that, because of the $\sign(h(t-\theta))$ term in \eqref{eq:iGZTth}, the derivative of a period-one orbit, $h'(t)$, will have discontinuities at
$t=\alpha+k/2-\theta$ for any integer $k$. In particular, if $\theta=0$ or $\theta=\tfrac12$, then the zeros of $h(t)$ coincide with the discontinuities of $h'(t)$. As a result, \(h'(\alpha)\) and \(h'(\alpha + \frac{1}{2})\) will have different left and right derivatives.
We will set these cases aside for now, and return to them in the proof of Theorem~\ref{thm:po}. 

Proposition~\ref{prop:alpha01} determines the value of 
$c\sin(2\pi\alpha)$, which leads to two possible values of
$c\cos(2\pi\alpha)$ in \eqref{eq:hdashalph}.
We first consider the case where $c\cos(2\pi\alpha)$ is positive.

For $\theta\in(0,\tfrac12)$, 
\be
\label{eq:c0proof}
c\sin(2\pi\alpha)={u}(\theta), \qquad c\cos(2\pi\alpha)=+\sqrt{c^2-{u}^2(\theta)},\quad \theta\in(0,\tfrac12),
\ee
implies $\cos(2\pi\alpha) \geq 0$ and $\alpha \in [-\frac{1}{4}, \frac{1}{4}] \pmod 1$. It then follows from \eqref{eq:hdashalph}
when \eqref{eq:necc} is satisfied, 
that $h'(\alpha) \geq 1$ so $\alpha$ is indeed an upward nondegenerate zero of $h(t)$.

For $\theta\in(\frac{1}{2},1)$, we see from \eqref{eq:hdashalph} that, to satisfy $h'(\alpha)\geq0$, a necessary condition is that $c\cos(2\pi\alpha)\geq1$. Since $\alpha\in[\frac{1}{4},\frac{3}{4}](\bmod1)$
implies $c\cos(2\pi\alpha)\leq0$, it is not possible to obtain a periodic orbit with $\alpha\in[\frac{1}{4},\frac{3}{4}](\bmod1)$ and $\theta\in(\frac{1}{2},1)$. Then any 1:1-periodic orbit with 
$\theta\in(\frac{1}{2},1)$ satisfies
\be
\label{eq:c1proof}
c\sin(2\pi\alpha)={v}(\theta), \qquad c\cos(2\pi\alpha)=+\sqrt{c^2-{v}^2(\theta)},\quad \theta\in(\tfrac12,1),
\ee
and $\alpha\in[-\frac{1}{4},\frac{1}{4}](\bmod1)$.
For such a solution, from $c\cos(2\pi\alpha)\geq1$ and \eqref{eq:alpha1}, we obtain the constraint
\be \label{eq:necc1}
c^2=c^2\cos(2\pi\alpha)+c^2\sin(2\pi\alpha)\geq 1 + {v}^2(\theta),
\ee
when $\theta\in(\frac{1}{2},1)$, which is stronger than the corresponding constraint stated in \eqref{eq:necc}. 

In Theorem~\ref{thm:po} we will construct 1:1-periodic orbits with 
$\alpha\in[-\frac{1}{4},\frac{1}{4}](\bmod1)$ for all $\theta\in[0,1)$, where we will find that additional constraints are required to ensure that the periodic orbit does not have additional zeros.

Next, consider 1:1-periodic orbits with $\alpha \in [\frac{1}{4}, \frac{3}{4}] \pmod 1$. From above such an orbit is impossible when $\theta \in (\frac{1}{2},1)$. 
On the other hand,  
for $\theta \in (0, \frac{1}{2})$,  
we have $\cos(2\pi\alpha) \leq 0$. In that case, using  \eqref{eq:hdashalph}, to satisfy $h'(\alpha) \geq 0$, requires
$0 \leq -c \cos(2\pi\alpha) \leq 1$, otherwise $\alpha$ is not an upward nondegenerate zero. In particular,
\be
\label{eq:c0proof3}
c\sin(2\pi\alpha)={u}(\theta), \qquad c\cos(2\pi\alpha)=-\sqrt{c^2-{u}^2(\theta)},\quad \theta\in(0,\tfrac12).
\ee
Such a solution must satisfy the constraint
\be  
\label{eq:necc1.5}  
c^2 = c^2 \cos(2\pi\alpha) + c^2 \sin(2\pi\alpha) \leq 1 + {u}^2(\theta). 
\ee  
Combining this with \eqref{eq:necc}, we obtain  
\be \label{eq:necc2}  
{u}^2(\theta) \leq c^2 \leq 1 + {u}^2(\theta).  
\ee  
Even though 
this solution can only exist for a narrow range of $c$ values, it will be of interest later, and
we construct it in Theorem~\ref{thm:po2}.

\subsubsection{\texorpdfstring{Symmetric 1:1-periodic orbits with $\alpha\in[-\tfrac14,\tfrac14]$}{Symmetric 1:1-periodic orbits with alpha in [-1/4,1/4]}} \label{sec:existsym11pos14}
We now determine the region $\mathcal{R}$ of the $(\theta,c)$-plane where symmetric 1:1-periodic orbits with phase 
$\alpha \in [-\frac{1}{4}, \frac{1}{4}] \pmod 1$ exist.

Recall the functions $u(\theta)$ and $v(\theta)$ defined by \eqref{eq:uv}. Also let $\hat{\theta}\approx0.4695$ be defined as the solution for
$\hat{\theta}\in[\tfrac14,\tfrac12]$ of
\begin{equation}
    \label{eq:thetahat}    u(\hat{\theta})\sin{\left(\sqrt{u^2(\hat{\theta})-1}\right)}+\cos{(2\pi\hat{\theta})}\sqrt{u^2(\hat{\theta})-1}=\sin{(2\pi\hat{\theta})}.
\end{equation}

\begin{defn} \label{def:cR}
The region $\mathcal{R}$ 
is defined as the set of all $(\theta, c) \in [0, 1) \times [0, \infty)$ satisfying the conditions:  
\begin{enumerate}[label=\textbf{(\roman*)}]
    \item \label{enumerate:1} For \( \theta =0 \) and \(\theta=\tfrac12\):
        \be
        \label{eq:c2pi2}
        c^2 \geq \frac{\pi^2}{4}+1.
        \ee
    \item \label{enumerate14} For  \(\theta=\tfrac14\):
        \be
        \label{eq:c0}
        c^2 >0.
        \ee   
    \item \label{enumerate:2} For $\theta\in(\hat{\theta},\tfrac12)$,
        \be
        \label{eq:c2cth}
        c^2 \geq w^2_-(\theta)
        \ee
where $w^2_-(\theta)$ denotes the smallest solution to the equation
\begin{align}
\label{eq:boundaryIc(theta)}
    w^2_-\sin{(\sqrt{w^2_--1})}&={u}(\theta) \left( \sin(2\pi \theta) - \cos(2\pi \theta) \sqrt{w^2_- - 1} \right) \\&\hspace*{2.5em}-\sqrt{w^2_- - {u}^2(\theta)} \left( \sin(2\pi \theta) \sqrt{w^2_- - 1} + \cos(2\pi \theta) \right)\!.\nonumber
\end{align}
The function $w^2_-(\theta)$ is strictly increasing for $\theta\in[\hat{\theta},\tfrac12]$, with 
\be
\label{eq:w^-2boundary}
    w^2_-(\hat{\theta}) = \frac{\pi^2}{4}, \quad w^2_-\left(\frac{1}{2}\right) = \frac{\pi^2}{4} + 1. 
\ee
\item \label{enumerate:3} For $\theta\in[\tfrac14,\hat{\theta}]$: \be
        \label{eq:cP}
        c^2 \geq u^2(\theta),
\ee

\item For $\theta\in(0,\tfrac14)$,  replace \( \theta \) with \( \frac{1}{2} - \theta\in(\tfrac14,\tfrac12) \) in \ref{enumerate:1}, \ref{enumerate:2} and \ref{enumerate:3}.
    \item For \( \theta \in (\tfrac{1}{2}, 1) \):
                \be
                \label{eq:cQsin}
                c^{2} \geq {v}^2(\theta) + \left(\frac{\pi/2 + {v}(\theta)\cos(2\pi\theta)}{\sin(2\pi\theta)}\right)^2.
                \ee
\end{enumerate}
\end{defn}

\begin{figure}[t!]
    \centering
    \includegraphics[width=\textwidth]{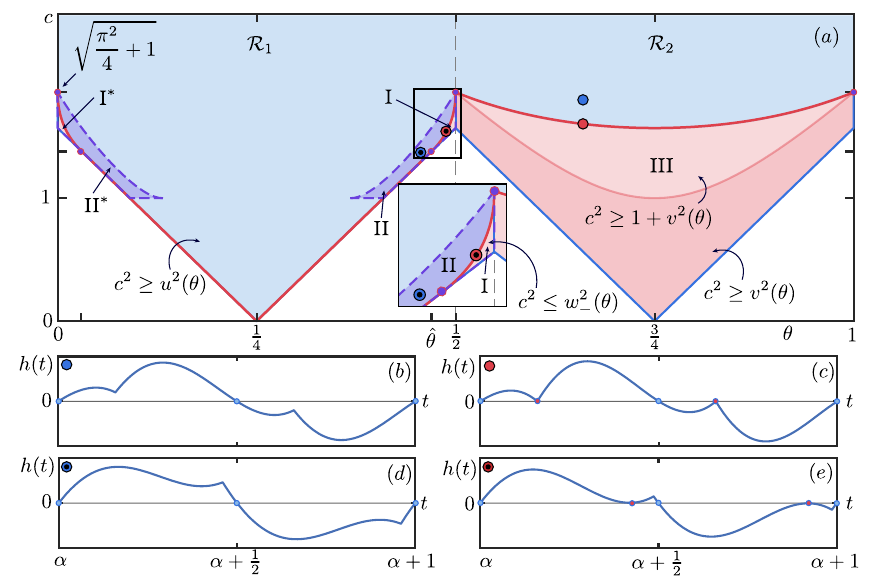} 
    \caption{
        Illustration of Definition~\ref{def:cR} and Theorem~\ref{thm:po}. 
        (a) The region $\mathcal{R}=\mathcal{R}_1\cup\mathcal{R}_2$ in light blue in the $(\theta, c)$-plane, with subregions: purple I and $\text{I}^*$ (smooth local extrema), II and $\text{II}^*$ (zero-line crossings of extrema), red III (non-smooth extrema at zero-line). 
        (b,d) Periodic orbits within $\mathcal{R}$; (c,e) periodic orbits at $\mathcal{R}$'s boundary.
    }
    \label{fig:RegEx}
\end{figure}

The following theorem states that 1:1-symmetric periodic
orbits exist within $\mathcal{R}$.

\begin{thm} \label{thm:po}
Let $\tau=k+\theta$ where $k\in\N_0$ and $\theta\in[0,1)$, and let $\cR$ be defined by
Definition~\ref{def:cR}.
If $(\theta,c)\in\mathcal{R}$, the psGZT model \eqref{eq:iGZT} admits a 1:1-symmetric periodic solution $h(t)$ 
given by \eqref{eq:h11},
with phase $\alpha\in[-\frac{1}{4},\frac{1}{4}]\pmod 1$ defined by
\eqref{eq:alpha0} for $\theta\in[0,\frac{1}{2})$,
and by \eqref{eq:alpha1} for $\theta\in[\frac{1}{2},1)$.
The function $p(t)$ satisfies  the relevant formula from
\eqref{eq:p0}-\eqref{eq:phalf1}, depending on the value of $\theta$.
\end{thm}

The proof of Theorem~\ref{thm:po} is quite technical and long, and involves considering many different cases. To avoid breaking up the flow of the presentation here, 
it is instead presented in Appendix~\ref{sec:pf:po}.

Figure~\ref{fig:RegEx}(a)
provides a visual representation of 
the existence region $\mathcal{R}$ in light blue, with the red curve as its boundary,
while Figure~\ref{fig:RegEx}(b)-(e) showcase four periodic orbits either within or at the boundary of $\mathcal{R}$, demonstrating how conditions \eqref{eq:halpha}-\eqref{eq:hzeros} are violated.  

The region $\cR$ can be split into a left-hand subregion 
\(\cR_1\) with \(\theta \in [0, \tfrac{1}{2})\) and a right-hand subregion
\(\cR_2\) with \(\theta \in [\tfrac{1}{2}, 1)\), with $\cR=\cR_1 \cup \cR_2$.
Because of Corollary~\ref{cor:sym2}, the closure of region $\cR_1$ is symmetric about
$\theta=\tfrac12$, and the closure of 
region $\cR_2$ is symmetric about
$\theta=\tfrac34$. The inequalities \eqref{eq:necc} define cones rooted at 
$(\theta,c) = (\tfrac{1}{4},0)$ and $(\theta,c) = (\tfrac{3}{4},0)$ which
contain $\cR_1$ and $\cR_2$. However, the necessary conditions \eqref{eq:necc} to obtain solutions are not sufficient everywhere. We see from 
Figure~\ref{fig:RegEx}(a) and \eqref{eq:c2pi2}--\eqref{eq:cQsin} that in $\cR_1$ the boundary of $\cR$ is given by \eqref{eq:cP} only when $\theta\in[\tfrac12-\hat{\theta},\hat{\theta}]$, while the boundary of $\cR_2$ does not touch the 
cone $c^2\geq v^2(\theta)$.  

Consider the right-hand subregion $\mathcal{R}_2$,
whose lower boundary is marked by the red curve
defined by \eqref{eq:cQsin}. Below this curve 
in region III, no valid symmetric 1:1 orbits exist: discontinuities at \(t = \alpha + \theta - \tfrac{1}{2}\) and \(t = \alpha + \theta\) intersect the zero-line, introducing four additional zeros per period. In Figure~\ref{fig:RegEx}(b), the discontinuities create sharp extrema which reach the zero-line in Figure~\ref{fig:RegEx}(c)
on the lower boundary of $\cR_2$. 

We now consider the more complicated left-hand subregion $\mathcal{R}_1$. The purple regions II and II* (defined by \eqref{eq:boundaryII} and via \(\theta \mapsto \tfrac{1}{2} - \theta\), respectively) mark where smooth segments of the orbit develop additional extrema. Within these lie smaller subregions I and I*. Subregion I is defined by $\theta\in[\hat\theta,\tfrac14]$
and $c^2\in[u^2(\theta),w^2_-(\theta)]$, and I* is its symmetric counterpart.
Within I and I* these extrema cross zero, invalidating \eqref{eq:hsigns},
and so I and I* are not contained in $\cR$.
Crucially, outside I and I* but within II and II*, symmetric 1:1 orbits persist.
Figure~\ref{fig:RegEx}(d) shows an example where extrema do not reach the zero-line, while (e) 
shows an example on the curve $c^2=w_-^2(\theta)$, the upper boundary of I, which has additional degenerate zeros. Below this curve, inside I, the condition \eqref{eq:hsigns} is violated and there is no 1:1-periodic orbit.
The curve $c^2=w_-^2(\theta)$ is
described in \ref{enumerate:2} 
and illustrated
in the enlargement box in Figure~\ref{fig:RegEx}(a), with its endpoints marked by purple dots at $(\theta,c)=(\hat{\theta},\tfrac\pi2)$ and $(\tfrac12,(\tfrac{\pi^2}{4}+1)^{\frac12})$.

\subsubsection{\texorpdfstring{Existence of symmetric 1:1-periodic orbits with $\alpha\in[\tfrac14,\tfrac34]$}{Existence of symmetric 1:1-periodic orbits with alpha in [1/4,3/4]}}
\label{section:existencesymmetric1434}

\begin{figure}
    \centering
    \includegraphics[width=\textwidth]{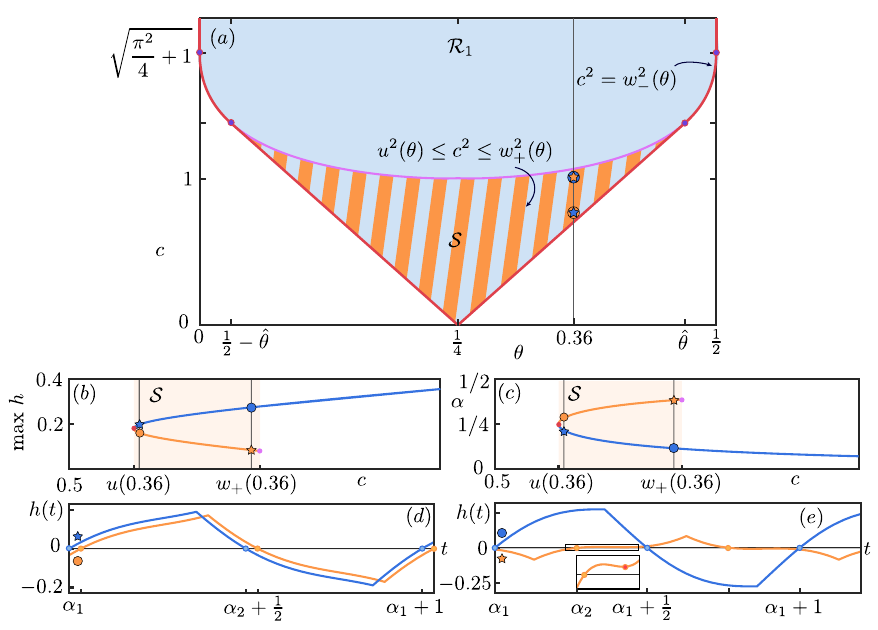}
    \caption{(a) The regions $\mathcal{R}_1$ (light blue) and $\mathcal{S}$ (hashed orange) in the $(\theta, c)$-plane for $\theta \in [0, \tfrac{1}{2}]$ where 1:1-periodic orbits exist with phase 
    $\alpha\in[-\tfrac14,\tfrac14](\bmod1)$ or $\alpha\in[\tfrac14,\tfrac34](\bmod1)$ respectively.
    (b) shows the maximum of the solution and (c) the phase, for
    continuation in $c$ for fixed $\theta = 0.36$ of the two 1:1-periodic orbits,
    showing a fold bifurcation between these periodic orbits when $u^2(\theta)=c^2$.
    (d)-(e) illustrate coexisting periodic orbits with phase $\alpha_1\in[-\tfrac14,\tfrac14]\pmod 1$ in blue, and phase $\alpha_2\in[\tfrac14,\tfrac34]\pmod 1$ in orange: (d) 
    at $c = 0.695$ near the fold point,  
    and (e) $c = 1.04$ near the point where the periodic orbit with phase $\alpha_2$ ceases to exist because 
    \eqref{eq:hsigns} is violated. 
    } 
    \label{fig:RegEx2}
\end{figure}

We now focus on symmetric 1:1-periodic orbits with phase $\alpha \in [\tfrac14,\tfrac34]$. As discussed in Section~\ref{section:parameterconstraintsforperiodicorbits}, the condition \( h'(\alpha) \geq 0 \) cannot be satisfied for \( \theta \in (\tfrac{1}{2}, 1) \); in other words, such orbits do not exist. Therefore, we restrict our attention to the case \( \theta \in [0, \tfrac{1}{2}] \).

\begin{defn} \label{def:cS}
    The region $\cS$ is
    defined 
    as the set of all \( (\theta, c) \in [\tfrac{1}{2}-\hat{\theta},\hat{\theta}] \times [0, \infty) \) satisfying:
\begin{enumerate}[label=\textbf{(\roman*')}]
\item \label{enumerate:1'} For $\theta\in[\tfrac14,\hat{\theta}]$, 
\be \label{eq:ucwp}
u^2(\theta)\leq c^2\leq w^2_+(\theta)
\ee
where $w^2_+(\theta)$ denote the smallest solution to the equation
\begin{align}
\label{eq:boundaryIc(theta)2}
    w^2_+\sin{(\sqrt{w^2_+-1})}&={u}(\theta) \left( \sin(2\pi \theta) - \cos(2\pi \theta) \sqrt{w^2_+ - 1} \right) \\&\hspace*{2em}+\sqrt{w^2_+ - {u}^2(\theta)} \left( \sin(2\pi \theta) \sqrt{w^2_+ - 1} + \cos(2\pi \theta) \right)\!.\nonumber
\end{align}
that is strictly increasing for $\theta\in[\tfrac14,\hat{\theta}]$, with 
\be
\label{eq:w+2boundary}
w^2_+(\tfrac14)=1,\quad w^2_+(\hat{\theta}) = \tfrac{\pi^2}{4}.
\ee  
\item For $\theta\in[\tfrac12-\hat{\theta},\tfrac14)$, replace $\theta$ with $\tfrac12-\theta$ in \ref{enumerate:1'}.
\end{enumerate}
\end{defn}

We now derive the region of existence for 1:1-periodic orbits with phase $\alpha\in[\tfrac 14,\tfrac 3 4]\pmod 1$.

\begin{thm} \label{thm:po2}
Let \( \tau = k + \theta \), where \( k \in \N_0 \) and \( \theta \in [0, \tfrac{1}{2}] \).  
Let the phase \( \alpha \in [\tfrac{1}{4}, \tfrac{3}{4}] \pmod 1\) be defined by \eqref{eq:alpha0}.  
Then, provided \( (\theta, c) \in \mathcal{S} \), where $\cS$ is defined in Definition~\ref{def:cS}, the psGZT model \eqref{eq:iGZT} admits a symmetric 1:1-periodic solution \( h(t) \), given by \eqref{eq:h11}, where $p(t)$ satisfies \eqref{eq:p0half}.
\end{thm}

The proof of Theorem~\ref{thm:po2} is similar to that of Theorem~\ref{thm:po}, and for the same reasons is 
provided in Appendix~\ref{sec:pf:po2}.

Figure~\ref{fig:RegEx2}(a) shows the region \( \mathcal{S} \) (hashed orange), overlaid on the region \( \mathcal{R} \) (light blue) for \( \theta \in [0, \tfrac{1}{2}] \). The upper boundary of \( \mathcal{S} \) (pink) is determined by the equation described in \ref{enumerate:1'}. Hence, the region \( \mathcal{S} \) is bounded, as expected from \eqref{eq:necc2}. This contrasts with orbits of phase \( \alpha \in [-\tfrac{1}{4}, \tfrac{1}{4}] \), which exist for all sufficiently large values of \( c \).

Within region \( \mathcal{S} \) for $c^2>u^2(\theta)$,
two distinct symmetric 1:1-periodic orbits coexist, one with phase \( \alpha_1 \in (-\tfrac{1}{4}, \tfrac{1}{4}) \), and the other with \( \alpha_2 \in (\tfrac{1}{4}, \tfrac{3}{4}) \).
Along the straight lines defined by $u^2(\theta) = c^2$, both orbits share the same phase: $\alpha = -\tfrac{1}{4}$ for $\theta \in (0, \tfrac{1}{4})$, and $\alpha = \tfrac{1}{4}$ for $\theta \in (\tfrac{1}{4}, \tfrac{1}{2})$; hence
by Theorems~\ref{thm:po} and~\ref{thm:po2} they are identical.
This behaviour is illustrated further in 
Figure~\ref{fig:RegEx2}(b-c), which display 
a one-parameter continuation in $c$ of these orbits for fixed $\theta=0.36$.
A fold
bifurcation of periodic orbits
is observed
at the curve defined by \( u^2(\theta) = c^2 \).
Following the upper branch in (c), we see that the 1:1-periodic orbit ceases to exist at the upper boundary of region \( \mathcal{S} \), defined by \eqref{eq:boundaryIc(theta)2}. At this boundary, two additional zeros of the orbit with phase \( \alpha \in [\tfrac{1}{4}, \tfrac{3}{4}] \) collide with the zero line. In contrast, the orbit with phase \( \alpha \in [-\tfrac{1}{4}, \tfrac{1}{4}] \) continues to exist for $c$ arbitrarily large
(see Theorem~\ref{thm:stabc}).

Figure~\ref{fig:RegEx2}(d) shows the two coexisting 1:1-periodic orbits 
which are nearly identical near the fold bifurcation. In contrast, Figure~\ref{fig:RegEx2}(e)
shows the two coexisting orbits near the upper boundary of \( \mathcal{S} \), 
which have very different profiles
because
two local extrema of the orbit with phase \( \alpha \in [-\tfrac{1}{4}, \tfrac{3}{4}] \) approach the zero line, as depicted in the enlargement box.

\section{Stability of 1:1-periodic orbits}
\label{sec:stab11}

In this section, we first study the stability of 1:1-periodic orbits of the psGZT model \eqref{eq:iGZT}, and then compare the stability domains in the $(\tau,c)$-plane for the 
GZT model  \eqref{eq:GZT}, the psGZT model \eqref{eq:iGZT} and for the signum forced psGZT model \eqref{eq:RKA}.

For smooth DDEs with constant delays, stability and local bifurcation theory is well established 
\cite{guo2013bifurcation,HaleLunel93,Smith11}, but that
theory is not directly applicable to the psGZT model \eqref{eq:iGZT} because of the discontinuity in the $\sign$ function. We will investigate the bifurcations that occur in the model indirectly by studying the stability of the 1:1-periodic orbits.

\subsection{Stability of 1:1-periodic orbits of the psGZT model}
\label{sec:stab11igzt}

We first examine the stability of 1:1-periodic orbits of the psGZT model \eqref{eq:iGZT}, focusing on identifying the precise points in the $(\tau,c)$-plane where these orbits lose stability.

The stability of a periodic orbit can be determined by analyzing the evolution of perturbations using the time-dependent linearization of the system along the orbit. The key to doing this is to realise that solutions of the psGZT model \eqref{eq:iGZT} are completely characterised by the location of their nondegenerate zeros. Thus we will study how the zeros of a perturbation 
of a 1:1-periodic orbit evolve over successive periods. This will result
in a linear map, whose solutions yield the Floquet multipliers, which indicate how perturbations grow or decay over time. The periodic orbit is stable if all the Floquet multipliers lie within the unit circle (aside from the trivial multiplier always located at 1), and it is unstable if there is at least one Floquet multiplier outside the unit circle.

\begin{figure}
    \centering
    \includegraphics[width=\textwidth]{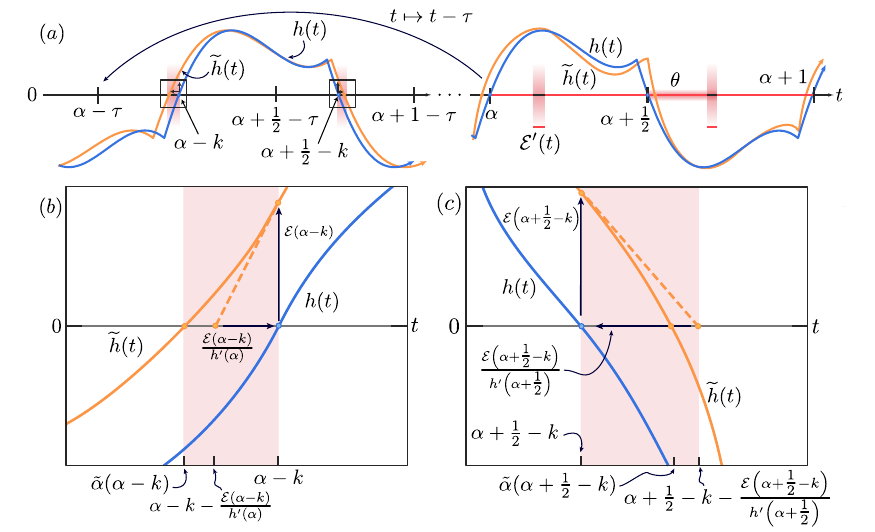}
    \caption{(a) Integration of $\mathcal{E}'(t)$ over $(\alpha, \alpha+1)$ with symmetric 1:1-periodic orbit $h(t)$ (blue), perturbed $\widetilde{h}(t)$ (orange), and $\mathcal{E}'(t) $ (red), for the case $\theta\in(0,\frac{1}{2})$. Differences of sign between $h(t-\tau)$ and $\widetilde{h}(t-\tau)$ imply that $|\mathcal{E}'(t)|=2$. (b) and (c) Linearized $h(t)$ and $\widetilde{h}(t)$ near $t = \alpha-k$ and $t = \alpha + \frac{1}{2}-k$.  
   }
    \label{fig:pert}
\end{figure}

Let $h(t)$ be a 1:1-periodic orbit which satisfies \eqref{eq:iGZT} and has no degenerate zeros, and let $\widetilde{h}(t)$ be a small perturbation of this orbit which satisfies
\be \label{eq:iGZTt}
\widetilde{h}'(t)=-\sign(\widetilde{h}(t-\tau))+c\cos(2\pi t).
\ee
Let
\be \label{eq:Et}
\cE(t)=\widetilde{h}(t)-h(t)
\ee
denote the difference between these solutions. To study the stability of the periodic orbit it is sufficient to study the
evolution of $\cE(t)$. But from \eqref{eq:iGZT} and \eqref{eq:iGZTt} we have
$$
\cE'(t)=\sign(h(t-\tau))-\sign(\widetilde{h}(t-\tau)),
$$ and thus $\cE'(t)=0$ unless
$\sign(h(t-\tau))\ne\sign(\widetilde{h}(t-\tau))$.
But for $\widetilde{h}(t)\approx h(t)$ the signs of the two functions will only differ near to the zeros of $h(t)$, which occur at $t=\alpha+k/2$ for $k\in\Z$. 

Assuming that $\tau\ne j/2$ for $j\in\N$ 
we obtain that
$h(\alpha)=0$ implies $h(\alpha-\tau)\ne0$. Then
$|h(t)-\widetilde{h}(t)|\ll1$
implies that $\sign(\widetilde{h}(\alpha-\tau))=\sign(h(\alpha-\tau))$.
Thus
$\widetilde{h}'(\alpha)=h'(\alpha)$ and, moreover, 
$\widetilde{h}'(t)=h'(t)$ in a neighbourhood of $t=\alpha$.

Let $\widetilde{\alpha}(\alpha+\tfrac{k}{2})$ denote the zero of $\widetilde{h}(t)$ close to $t=\alpha+\tfrac{k}{2}$. 
Applying one-step of Newton-Raphson to find $\widetilde{\alpha}(\alpha)$ starting from $x_0=\alpha$, we obtain
$$x_1=x_0-\frac{\widetilde{h}(x_0)}{\widetilde{h}'(x_0)}
=\alpha-\frac{\widetilde{h}(\alpha)}{{h}'(\alpha)}
$$
with (see for example Equation (23) in Section 3.1.2 of \cite{IK66}) error bound
$$|x_1-\widetilde{\alpha}(\alpha)|=
\cO\big((x_1-x_0)^2\big).$$
Thus
$$
\left|\widetilde{\alpha}(\alpha)
-\left(\alpha-\frac{\widetilde{h}(\alpha)}{h'(\alpha)} \right) \right| 
=  \cO\left(\Big(\alpha
-\big(\alpha-\frac{\widetilde{h}(\alpha)}{h'(\alpha)} \Big) \Big)^2\right)
=\cO\left(\Big|\frac{\widetilde{h}(\alpha)}{h'(\alpha)}\Big|^2\right).$$
But, since $\widetilde{h}(\alpha)=\mathcal{E}(\alpha)$, this implies that
\be \label{eq:thalpha=0}
\widetilde{\alpha}(\alpha)
= \alpha-\frac{\widetilde{h}(\alpha)}{h'(\alpha)} 
+\cO(\mathcal{E}(\alpha)^2).
\ee

Similarly, for $\tau\ne k/2$ for $k\in\N$  we obtain that
$h(\alpha+\tfrac12)=0$ implies $h(\alpha+\tfrac12-\tau)\ne0$. 
Then
$\sign(\widetilde{h}(t))=\sign(h(t))$
when $|h(t)-\widetilde{h}(t)|\ll1$ and $t$ is in a neighbourhood of $\alpha+\tfrac12-\tau$.
Then
\be \label{eq:thalphahalf=0}
\widetilde{\alpha}(\alpha+\tfrac12)
= \alpha+\frac12-\frac{\widetilde{h}(\alpha+\tfrac12)}{h'(\alpha+\tfrac12)}
+\cO(\cE(\alpha+\tfrac12)^2)
= \alpha+\frac12+\frac{\widetilde{h}(\alpha+\tfrac12)}{h'(\alpha)}
+\cO(\cE(\alpha+\tfrac12)^2)
\ee
by using the symmetry of the solution.
Thus, for a small perturbation $|\cE(t)|\ll1$, to leading order, the positions of the zeros of $\widetilde{h}(t)$ are directly related to the value of $\widetilde{h}(t)$ at the nearby zero of $h(t)$,
as shown in Figure~\ref{fig:pert}.
This allows us to determine the intervals on which the signs of $h(t)$ and $\widetilde{h}(t)$ differ, just from the values of $h'$ and $\widetilde{h}$ at the zero.
With this we can calculate the evolution of $\cE(t)$ over one period. Only the values of $\cE(t)$ at the two zeros per period are of interest, as these determine the intervals on which $h(t-\tau)$ and $\widetilde{h}(t-\tau)$ have different signs.

\begin{prop} \label{prop:cEformulation}
Let $h(t)$  be a 1:1-periodic orbit of \eqref{eq:iGZT} with no degenerate 
zeros, and let $\widetilde{h}(t)$ and $\cE(t)$ be defined by \eqref{eq:iGZTt} and \eqref{eq:Et}. Let $m\in{\Z}$ 
and let $\tau=k+\theta$ where $k\in\N_0$. 
For $\theta\in(0, \frac{1}{2})$, to leading order,
\be
\label{eq:cE_0_12_1}
\cE(m+1+\alpha)=\cE(m+\alpha)-\frac{2\cE(m-k+\alpha)}{h'(\alpha)}-\frac{2\cE(m-k+\alpha+\tfrac12)}{h'(\alpha)},
\ee
and
\be
\label{eq:cE_0_12_2}
\cE(m+\tfrac32+\alpha)=\cE(m+\tfrac12+\alpha)
-\frac{2\cE(m-k+\alpha+\tfrac12)}{h'(\alpha)} -\frac{2\cE(m-k+\alpha+1)}{h'(\alpha)}.
\ee
  For $\theta\in( \frac{1}{2},1)$, to leading order, 
    \be
\label{eq:cE_12_1_1}
\cE(m+1+\alpha)=\cE(m+\alpha)-\frac{2\cE(m-k-\tfrac12+\alpha)}{h'(\alpha)}
-\frac{2\cE(m-k+\alpha)}{h'(\alpha)},
\ee
and
\be
\label{eq:cE_12_1_2}
\cE(m+\tfrac32+\alpha)=\cE(m+\tfrac12+\alpha)
-\frac{2\cE(m-k+\alpha)}{h'(\alpha)} -\frac{2\cE(m-k+\alpha+\tfrac12)}{h'(\alpha)}.
\ee
\end{prop}

\begin{proof}
 Let $m\in{\Z}$ and let $\tau=k+\theta$ where $k\in\N_0$. Then
\begin{align*}
\cE(m+1+\alpha)&=\cE(m+\alpha)+\int_{m+\alpha}^{m+\alpha+1}\cE'(t)dt=\cE(m+\alpha)+\int_{m+\alpha}^{m+\alpha+1}\widetilde{h}'(t)-h'(t)dt\\
&=\cE(m+\alpha)+\int_{m-k+\alpha-\theta}^{m-k+\alpha-\theta+1}\sign(h(t))-\sign(\widetilde{h}(t))dt.
\end{align*} 

If $\theta\in(0, \frac{1}{2})$, then $h(t)$ changes sign on the integration interval at
$t=m-k+\alpha$ and $t=m-k+\alpha+\tfrac12$. Applying \eqref{eq:thalpha=0} and \eqref{eq:thalphahalf=0} for these points (and noting that $h(t)$ is periodic, but $\widetilde{h}(t)$ is not)
we obtain
\begin{align} \notag
\cE(m+1&+\alpha)\\ \notag
& = \cE(m+\alpha)
+\hspace*{-1.5em}\int\displaylimits_{\widetilde{\alpha}(m-k+\alpha)}^{m-k+\alpha}\hspace*{-1.3em}\sign(h(t))-\sign(\widetilde{h}(t))dt
+\hspace*{-2em}\int\displaylimits_{m-k+\alpha+\tfrac12}^{\widetilde{\alpha}(m-k+\alpha+\tfrac12)}\hspace*{-1.5em}\sign(h(t))-\sign(\widetilde{h}(t))dt\\   \notag
&=\cE(m+\alpha)
+\hspace*{-2em}\int\displaylimits_{m-k+\alpha-\frac{\cE(m-k+\alpha)}{h'(\alpha)}}^{m-k+\alpha}\hspace*{-3em}\sign(h(t))-\sign(\widetilde{h}(t))dt
+\hspace*{-3em}\int\displaylimits_{m-k+\alpha+\tfrac12}^{m-k+\alpha+\tfrac12-\frac{\cE(m-k+\alpha+1/2)}{h'(\alpha+1/2)}}\hspace*{-4.5em}\sign(h(t))-\sign(\widetilde{h}(t))dt\\ \notag
& \qquad +\cO({\cE(m-k+\alpha)^2})+\cO(\cE(m-k+\alpha+\tfrac12)^2)\\ \notag
& = \cE(m+\alpha)-\frac{2\cE(m-k+\alpha)}{h'(\alpha)}-\frac{2\cE(m-k+\alpha+\tfrac12)}{h'(\alpha)}\\ 
& \qquad +\cO(\cE(m-k+\alpha)^2)+\cO(\cE(m-k+\alpha+\tfrac12)^2),
\quad \theta\in(0, \tfrac{1}{2}).
\label{eq:recalphazero}
\end{align}

On the other hand, if $\theta\in(\frac{1}{2},1)$, then $h(t)$ changes sign at
$t=m-k-\tfrac12+\alpha$ and $t=m-k+\alpha$
on the integration interval.  Hence,
\begin{align} \notag
\cE(m+1+\alpha)& =\cE(m+\alpha)-\frac{2\cE(m-k-\tfrac12+\alpha)}{h'(\alpha)}
-\frac{2\cE(m-k+\alpha)}{h'(\alpha)}\\  \label{eq:recalphaone} 
& \qquad  +\cO(\cE(m-k-\tfrac12+\alpha)^2)+\cO(\cE(m-k+\alpha)^2),
\quad \theta\in(\tfrac{1}{2},1).
\end{align}

Similarly, for the other zero we obtain
\begin{align*}
\cE(m+\alpha+\tfrac32)&=\cE(m+\alpha+\tfrac12)+\int_{m+\alpha+\tfrac12}^{m+\alpha+\tfrac32}\cE'(t)dt\\
&=\cE(m+\alpha+\tfrac12)+\int_{m+\alpha+\tfrac12}^{m+\alpha+\tfrac32}\widetilde{h}'(t)-h'(t)dt\\
&=\cE(m+\alpha+\tfrac12)+\int_{m-k+\alpha+\tfrac12-\theta}^{m-k+\alpha-\theta+\tfrac32}\sign(h(t))-\sign(\widetilde{h}(t))dt.
\end{align*}
Then, for $\theta\in(0, \frac{1}{2})$ the solution $h(t)$ changes sign on the integration interval
at
$t=m-k+\alpha+\tfrac12$ and $t=m-k+\alpha+1$, while for $\theta\in(\frac{1}{2},1)$ the sign changes occur at $t=m-k+\alpha$ and $t=m-k+\alpha+\tfrac12$.

Hence, we find that
\begin{align} \notag
\cE(m+\tfrac32+\alpha)&=\cE(m+\tfrac12+\alpha)
-\frac{2\cE(m-k+\alpha+\tfrac12)}{h'(\alpha)} -\frac{2\cE(m-k+\alpha+1)}{h'(\alpha)}
\\ \label{eq:rechalfzero}
& \qquad +\cO(\cE(m-k+\alpha+\tfrac12)^2)+\cO(\cE(m-k+\alpha+1)^2),
\quad \theta\in(0, \tfrac{1}{2}).
\end{align}
and
\begin{align} \notag 
\cE(m+\tfrac32+\alpha)&=\cE(m+\tfrac12+\alpha)
-\frac{2\cE(m-k+\alpha)}{h'(\alpha)} -\frac{2\cE(m-k+\alpha+\tfrac12)}{h'(\alpha)}
\\  \label{eq:rechalfone}
& \qquad +\cO(\cE(m-k+\alpha)^2)+\cO(\cE(m-k+\alpha+\tfrac12)^2),
\quad \theta\in(\tfrac{1}{2},1).
\end{align}
The result follows upon dropping the higher-order nonlinear terms.
\end{proof}

In both cases we have a pair of coupled linear difference equations, which can be solved by making the ansatz
\be \label{eq:ansatz}
\left(\begin{array}{c}
\cE(m+\alpha)\\
\cE(m+\alpha+\tfrac{1}{2})
\end{array}\right)=\lambda^m\left(\begin{array}{c}
v_1\\
v_2
\end{array}\right).
\ee

We remark that since $\cE'(t)=0$ whenever $\sign(h(t-\tau))=\sign(\widetilde h(t-\tau))$, if $\lambda$ satisfies
\eqref{eq:ansatz} then $\cE(t+1)=\lambda \cE(t)$ for any $t>0$, provided $|\cE(t)|\ll1$ so that the linear approximation is valid.
Any $\lambda$ that satisfies \eqref{eq:ansatz} is then a Floquet multiplier of the periodic orbit, and these determine its stability.

\begin{thm}
    \label{thm:Poly}
    Let $\tau=k+\theta$, $k\in\mathbb{N}_0$. The characteristic polynomial $\Delta(\lambda)$ associated with \eqref{eq:ansatz} satisfies
    $\Delta(\lambda)=[\lambda-1]\Delta_{2k+1}(\lambda)$ for $\theta\in(0, \frac{1}{2})$
    and $\Delta(\lambda)=[\lambda-1]\Delta_{2k+2}(\lambda)$ for $\theta\in(\frac{1}{2},1)$,
    where
\begin{align} \label{eq:polynomiallambda2kp1}
\Delta_{2k+1}(\lambda)&=
(\lambda-1)\lambda^{2k}+\frac{4}{h'(\alpha)}\lambda^k-\frac{4}{[h'(\alpha)]^2}, \\ \label{eq:polynomiallambda2kp2}
\Delta_{2k+2}(\lambda)&=
(\lambda-1)\lambda^{2k+1}+\frac{4}{h'(\alpha)}\lambda^{k+1}+\frac{4}{[h'(\alpha)]^2},
\end{align}
given that $h'(\alpha)>0$. 
\end{thm}

\begin{proof}
In the case that $\tau=k+\theta$ with $\theta\in(0, \frac{1}{2})$, substituting
\eqref{eq:cE_0_12_1} and \eqref{eq:cE_0_12_2} into
\eqref{eq:ansatz} results in
$$A_1\left(\!\begin{array}{c}
v_1\\
v_2
\end{array}\!\right)=\left(\begin{array}{cc}
\lambda^{k+1}-\lambda^k+\tfrac{2}{h'(\alpha)}  & \frac{2}{h'(\alpha)} \\
\frac{2\lambda}{h'(\alpha)} & \lambda^{k+1}-\lambda^k+\tfrac{2}{h'(\alpha)}
\end{array}\right)\!\left(\!\begin{array}{c}
v_1\\
v_2
\end{array}\!\right)=\left(\begin{array}{c}
0\\
0
\end{array}\right).$$
For non-trivial solutions we require $\det(A_1)=0$. This leads to
$$0=\det(A_1)=[\lambda-1]\Bigl[(\lambda-1)\lambda^{2k}+\frac{4}{h'(\alpha)}\lambda^k-\frac{4}{[h'(\alpha)]^2}
\Bigr].$$

In the case that $\tau=k+\theta$ with $\theta\in(\frac{1}{2},1)$, substituting
\eqref{eq:cE_12_1_1} and \eqref{eq:cE_12_1_2} into
\eqref{eq:ansatz} results in
$$A_2\left(\!\begin{array}{c}
v_1\\
v_2
\end{array}\!\right)=\left(\begin{array}{cc}
\lambda^{k+2}-\lambda^{k+1}+\tfrac{2}{h'(\alpha)}\lambda  & \frac{2}{h'(\alpha)} \\
\frac{2}{h'(\alpha)} & \lambda^{k+1}-\lambda^k+\tfrac{2}{h'(\alpha)}
\end{array}\right)\!\left(\!\begin{array}{c}
v_1\\
v_2
\end{array}\!\right)=\left(\begin{array}{c}
0\\
0
\end{array}\right),$$
and for non-trivial solutions we require
\[0=\det(A_2)=[\lambda-1]\Bigl[(\lambda-1)\lambda^{2k+1}+\frac{4}{h'(\alpha)}\lambda^{k+1}+\frac{4}{[h'(\alpha)]^2}
\Bigr].\]

\vspace*{-3ex}
\qedhere
\end{proof}

Thus there is one root $\lambda=1$ of $\Delta$ (which corresponds to translation along the orbit for the linearised equation), with remaining
Floquet multipliers satisfying  $\Delta_{2k+1}(\lambda)=0$  for $\theta\in(0,\tfrac12)$
where $\Delta_{2k+1}$ defined by \eqref{eq:polynomiallambda2kp1}
is a degree $2k+1$ polynomial, 
and  $\Delta_{2k+2}(\lambda)=0$  for $\theta\in(\tfrac12,1)$
where $\Delta_{2k+2}$ defined by \eqref{eq:polynomiallambda2kp2}
is a degree $2k+2$ polynomial.

We will not attempt to find all the roots of the characteristic polynomial $\Delta$. Rather we are interested in changes of stability and bifurcations. These occur when Floquet multipliers cross the unit circle where $|\lambda|=1$. The classical Floquet theory
is well established for bifurcations of periodic orbits in finite-dimensional dynamical systems, including for periodically forced systems \cite{kuznetsov1998elements}, and the extension to DDEs is also well studied
\cite{guo2013bifurcation}. The Floquet theory framework also extends to discontinuous bifurcations of periodic orbits, where Floquet multipliers may jump abruptly across stability boundaries \cite{LEINE2002259}.

There are three types of codimension-1 bifurcations of periodic orbits that we briefly introduce. When a single real Floquet multiplier crosses at $+1$, any of the basic bifurcations (saddle-node, transcritical, pitchfork) are possible, depending on the underlying symmetries. For
the psGZT model 
we generally observe a saddle-node bifurcation of orbits, where two orbits of different stability collide and annihilate each other. 
When a real Floquet multiplier crosses $-1$, a period-doubling bifurcation occurs and a new periodic orbit is created with twice the period of the original one.
Lastly, when a complex pair of Floquet multipliers 
$\lambda = e^{\pm 2\pi i \rho}$ with $\rho\in(0,\tfrac12)$, crosses the unit circle
an invariant torus is created in a torus bifurcation. 
(If $\rho=\tfrac13$ or $\rho=\tfrac14$ a strong resonance occurs, and more complicated bifurcation structures arise, but we will not be concerned with these).

We begin by showing that the symmetric 1:1-periodic orbits, shown to exist 
in Theorem~\ref{thm:po2}, are always unstable,
while the
symmetric 1:1-periodic orbits, shown to exist 
in Theorem~\ref{thm:po}, are stable for all sufficiently large $c$.

\begin{thm} \label{thm:unstab11}
A symmetric 1:1-periodic orbit with phase $ \alpha \in \left(\tfrac{1}{4}, \tfrac{3}{4}\right) \pmod 1$, no degenerate zeros, and delay $ \tau > 0$, but $\tau$ not an integer or half integer (so $\tau\neq j/2$, $ j \in \mathbb{N} $)
is unstable.
\end{thm}

\begin{proof}
These orbits only exist for  $\tau=k+\theta$ where $k\in\N_0$ and $\theta\in(0,\tfrac12)$ and so their stability is governed by $\Delta(\lambda)=[\lambda-1]\Delta_{2k+1}(\lambda)$ where $\Delta_{2k+1}(\lambda)$ is defined by \eqref{eq:polynomiallambda2kp1}. Notice that
$\Delta_{2k+1}(\lambda)\to+\infty$ as $\lambda\to+\infty$, while
$$\Delta_{2k+1}(1)=\frac{4}{h'(\alpha)}\left(1-\frac{1}{h'(\alpha)}\right).$$
But $ \alpha \in (\tfrac{1}{4}, \tfrac{3}{4})$ implies that
$c\cos(2\pi\alpha)<0$ and, hence, $h'(\alpha)\in(0,1)$. Thus, 
$\Delta_{2k+1}(1)<0$ and it follows that $\Delta_{2k+1}(\lambda)$ has a real root $\lambda>1$ and, hence, the periodic orbit is unstable.
\end{proof}

\begin{thm} \label{thm:stabc}
Symmetric 1:1-periodic orbits with phase $\alpha \in[-\tfrac14,\tfrac14] \pmod 1$ exist for all $c$ sufficiently large for any delay $\tau>0$.
For $\tau$ not an integer or half integer (so $\tau\neq j/2$, $ j \in \mathbb{N} $)
they are also asymptotically stable for all $c$ sufficiently large. 
\end{thm}

\begin{proof}
As established in Theorem~\ref{thm:po} and illustrated in Figure~\ref{fig:RegEx}, the boundary of $\cR$ is contained in the region where $c^2\leq \tfrac{\pi^2}{4} + 1$. Thus, symmetric 1:1-periodic orbits with phase $\alpha\in[-\tfrac 14,\tfrac 14]\pmod 1$ exist for sufficiently large $c$. Moreover, it is shown in the proof of Theorem~\ref{thm:po}
that the constructed 1:1-periodic orbits do not
have any degenerate zeros for $c^2>\tfrac{\pi^2}{4} + 1$. Thus Proposition~\ref{prop:cEformulation} and
Theorem~\ref{thm:Poly} apply, and the stability of the periodic orbit is governed by 
\eqref{eq:polynomiallambda2kp1} or
\eqref{eq:polynomiallambda2kp2}.

Let $\tau=k+\theta$ where $k\in\mathbb{N}_0$. If $\theta\in(0,\tfrac12)$ then from \eqref{eq:hdashalph} and \eqref{eq:c0proof} we have
\be \label{eq:hdashalpcu}
h'(\alpha)=1+\sqrt{c^2-u^2(\theta)},
\ee
while for $\theta\in(\tfrac12,1)$ from \eqref{eq:hdashalph} and \eqref{eq:c1proof} we have
$$h'(\alpha)=-1+\sqrt{c^2-v^2(\theta)}.$$
In both cases $h'(\alpha)\to+\infty$ as $c\to+\infty$.
Hence, for $\theta\in(0,\tfrac12)$ we rewrite $\Delta_{2k+1}$ as
$$\Delta_{2k+1}(\lambda)=
(\lambda-1)\lambda^{2k}+4\epsilon\lambda^k-4\epsilon^2,$$
where $\epsilon=1/h'(\alpha)>0$ and $\epsilon\to0^+$ as $c\to\infty$. When $\epsilon=0$ this polynomial has a root of multiplicity $2k$ at zero and a simple root at $\lambda=1$. By continuity, for $c\gg0$ and so $0<\epsilon\ll1$, $\Delta_{2k+1}$  will have $2k$ roots close to zero. It remains only to consider how the root at $\lambda=1$ perturbs for $\epsilon>0$.  But 
$$\Delta_{2k+1}(1)=4\epsilon(1-\epsilon)>0,$$
when $\epsilon\in(0,1)$, while for all $\lambda\geq1$ and $k\geq0$ we have
$$\Delta^{'}_{2k+1}(\lambda)=\lambda^{k-1}[
(2k+1)\lambda^{k+1}-2k\lambda^k+4\epsilon k]>0.$$
Hence $\Delta(\lambda)>0$ for all $\lambda\geq1$ when $\epsilon\in(0,1)$. It follows that the simple root $\lambda=1$ when $\epsilon=0$ perturbs to a real root $\lambda\in(0,1)$ for $0<\epsilon\ll1$. Hence, for $\epsilon$ sufficiently small, or equivalently for $c$ sufficiently large, all of the non-trivial Floquet multipliers are smaller than $1$ in modulus and the periodic orbit is stable.

For $\theta\in(\tfrac12,1)$ we rewrite $\Delta_{2k+2}$ as
$$\Delta_{2k+2}(\lambda)=
(\lambda-1)\lambda^{2k+1}+4\epsilon\lambda^k+4\epsilon^2,$$
and the remaining steps are analogous to the case 
 $\theta\in(0,\tfrac12)$.
\end{proof}

In the following theorem, we use \eqref{eq:polynomiallambda2kp1},\eqref{eq:polynomiallambda2kp2} to determine the exact coordinates in the $(\tau,c)$-plane at which 
a 1:1-periodic orbit with phase $\alpha\in[-\tfrac14,\frac14]\pmod 1$ becomes unstable in a torus bifurcation or a saddle-node bifurcation of periodic orbits. In the case of the torus bifurcations, we find the associated closed-form expression for the rotation number $\rho$ as parametrized by $\tau=k+\theta$. We determine that the rotation number is constant over intervals where $\theta\in(0, \frac{1}{2})$ and $\theta\in(\frac{1}{2},1)$, and there are jump discontinuities between such plateaus. 

\begin{thm} \label{thm:w}
Consider symmetric 1:1-periodic orbits with no degenerate zeros and phase $ \alpha \in \left[-\tfrac{1}{4}, \tfrac{1}{4}\right] \pmod 1$ and delay $ \tau = k + \theta $, where $k\in\N_0$. For $ \tau > \tfrac12$, but $\tau$ not an integer or half integer (so $\tau\neq j/2$, $ j \in \mathbb{N} $), these orbits lose stability   along a torus bifurcation curve $\textup{\textbf{T}}$, where a complex conjugate pair of Floquet multipliers crosses the unit circle, i.e., $ \lambda = e^{\pm 2\pi i\rho} $, with $\rho \in (0,\tfrac12)$.

The torus bifurcation curve $\textup{\textbf{T}}$ is given by
\begin{align} \label{eq:T}
\begin{dcases*}
c^2=\left(\csc\left({\frac{\pi}{2(4k+1)}}\right)-1\right)^2+u^2(\theta), & \textrm{for}\; $\theta\in(0, \frac{1}{2})$\quad\text{and}\quad $k>0$\\
c^2=\left(\csc\left({\frac{\pi}{2(4k+3)}}\right)+1\right)^2+v^2(\theta), & \textrm{for}\;  $\theta\in(\frac{1}{2},1)$\quad\text{and}\quad $k\geq 0$.
\end{dcases*}\end{align}
Along the curve $\textup{\textbf{T}}$, the rotation number $\rho$ is given by 
\begin{align} \label{eq:w}
\rho&= \begin{dcases*}
\frac{1}{4k+1}, & \textrm{for}\; $\theta\in(0, \frac{1}{2})$\quad\text{and}\quad $k> 0$\\
\frac{1}{4k+3} & \textrm{for}\;  $\theta\in(\frac{1}{2},1)$\quad\text{and}\quad $k\geq 0$.
\end{dcases*}\end{align}

For $ \tau \in \left[\tfrac{1}{2} - \hat{\theta}, \hat{\theta} \right] $, where $ \hat{\theta} \approx 0.4695 $ satisfies \eqref{eq:thetahat}, the orbits lose stability via a saddle-node bifurcation of periodic orbits, denoted $\textup{\textbf{SN}}$. 
The saddle-node bifurcation curve 
$\textup{\textbf{SN}}$ is given by:
\begin{equation}
    \label{eq:SN}
    {u}^2(\theta)=c^2\,\,\textrm{for}\,\, \theta\in\left[\tfrac{1}{2}-\hat{\theta},\hat{\theta}\right].
\end{equation}
\end{thm}

The proof of Theorem~\ref{thm:w} is quite technical, and involves evaluating the real and imaginary parts of the polynomials defined in
\eqref{eq:polynomiallambda2kp1} 
and
\eqref{eq:polynomiallambda2kp2}
with $\lambda=e^{i\omega}$,
combined with some trigonometric identities to figure out where the bifurcation curves are in parameter space. It
is provided in the Appendix~\ref{sec:pf:w}.

\begin{figure}
    \centering
    \includegraphics[width=\textwidth]{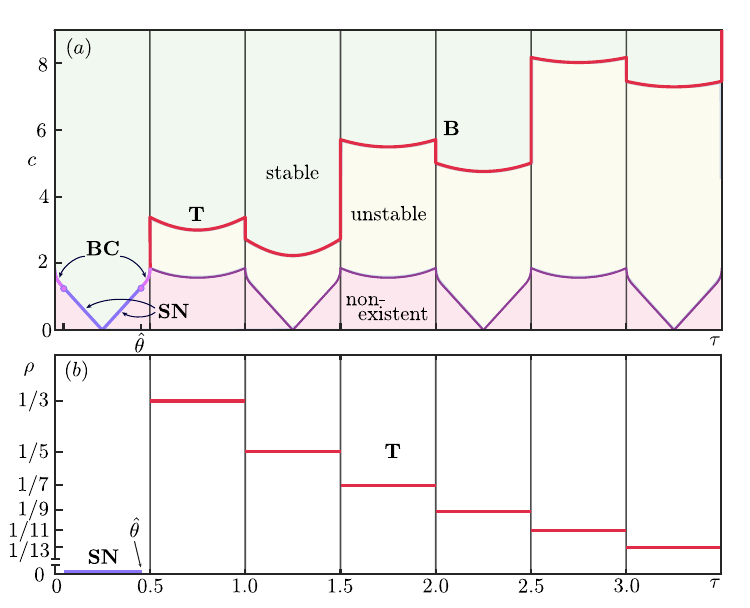}
    \caption{(a) Stability boundary $\mathbf{B}$ for symmetric 
    1:1-periodic orbits with phase $\alpha \in [-\frac{1}{4}, \frac{1}{4}]$. For $\tau\geq\tfrac12$ the red curve $\mathbf{T}$ separates the stability (green) and instability (yellow) regions in the $(\tau, c)$-plane.    
    Purple curves mark where symmetric 1:1 orbits cease to exist; no such orbits exist in the purple region. 
    For $\tau\in(\tfrac12-\hat\theta,\hat\theta)$ there is a saddle-node bifurcation of 1:1-periodic orbits along the part of the curve $c^2=u^2(\theta)$ labelled $\mathbf{SN}$. At the pink curves labelled $\mathbf{BC}$ the 1:1-periodic orbits cease to exist, 
    as already illustrated in Figure~\ref{fig:RegEx}(e).
    (b) Rotation number $\rho$, parameterized by $\tau$, with a logarithmic $y$-axis. 
    }
    
    \label{fig:T}
\end{figure}

Theorem~\ref{thm:w} provides an almost complete explanation for the loss of stability of the symmetric 1:1-periodic orbits.
The bifurcation curves derived in the theorem define a stability boundary, denoted as 
$\mathbf{B}$, in the $(\tau,c)$-plane,
which is depicted in Figure~\ref{fig:T}(a).
Above the stability curve  $\textbf{B}$, symmetric 1:1-periodic orbits are stable, and periodic (seasonal) forcing is dominant. The purple curves mark the 
boundary of the region where these periodic orbits exist;
below this curve, there are no symmetric 1:1-periodic orbits. 
The stability boundary \( \textbf{B} \) 
is composed of three main parts:
the torus bifurcation curve \( \textbf{T} \) for $\tau>\tfrac12$
and the saddle-node bifurcation curve \( \textbf{SN} \)
for $\tau\in(\tfrac12-\hat\theta,\hat\theta)$, 
as described in Theorem~\ref{thm:w}, along with a border collision  \( \textbf{BC} \) 
for $\tau\in(0,\tfrac12-\hat\theta)\cup(\hat\theta,\tfrac12)$, which is discussed below. 

We first describe the torus bifurcation curve \( \textbf{T} \).
From  Figure~\ref{fig:T}(a) and Theorem~\ref{thm:w} we see that  \( \textbf{T} \) is smooth on the open intervals $\tau\in(k,k+\tfrac12)$ 
and
$\tau\in(k+\tfrac12,k+1)$ 
for $k\in\N$. Figure~\ref{fig:T}(b) shows that the rotation number $\rho$ remains constant and rational over each of these intervals  
but decreases as $\tau$ increases, creating a descending staircase-like pattern. 

That the rotation number is a piecewise constant function of $\tau$ follows from \eqref{eq:elimzero} and \eqref{eq:elimone}
in the proof of Theorem~\ref{thm:w}.
The only variables in these equations
are $k$ and $\omega$; an immediate consequence of this
is that for
$\tau=k+\theta$ 
the value of any Floquet multiplier $\lambda=e^{i\omega}$ 
depends on $\tau$ only
through the value of its integer part $k$,
and whether $\theta\in(0,\tfrac12)$ or $\theta\in(\tfrac12,1)$.
Once a value of $\omega$ is found, the value of $h'(\alpha)$ is defined through \eqref{eq:hdashzero} or \eqref{eq:hdashone}.
Thus, $\omega$ being a piecewise constant function of $\tau$ implies that $h'(\alpha)$ is also constant over each subinterval 
 $\tau\in(k,k+\tfrac12)$ 
or
$\tau\in(k+\tfrac12,k+1)$.
Since, due to equations \eqref{eq:alpha0} and \eqref{eq:alpha1},
the phase $\alpha$ and hence also the value of $h'(\alpha)$ are non-constant functions of $c$ and $\tau$, the curvature of  \( \textbf{T} \)  seen
in Figure~\ref{fig:T}(a) is required 
to obtain a constant $h'(\alpha)$
and, hence, a constant rotation number over each subinterval 
 $\tau\in(k,k+\tfrac12)$ 
or
$\tau\in(k+\tfrac12,k+1)$; note that
the values of $c$ and $\theta$ only enter into the computation to ensure that $h'(\alpha)$ has the required value.

The stability of symmetric 1:1-periodic orbits is delicate when 
 \( \tau = j/2 \) for \( j \in \mathbb{N} \), because then $\alpha$ a nondegenerate zero, which implies that $\alpha-\tau$ is also a nondegenerate zero, and so $h'(t)$ is discontinuous at $t=\alpha$. The behaviour observed will then depend on one or both of $h'(\alpha^\pm)$ depending on how $\widetilde{h}$ is perturbed from $h$. We will not analyse this case in detail, but note that, since our previous arguments apply when \( \tau = j/2\pm\epsilon \), the vertical segments of \( \textbf{B} \) for 
  \( \tau = j/2 \) shown in Figure~\ref{fig:T}(a) 
 are certainly part of the stability boundary.

While the rotation number of the torus curve  \( \textbf{T} \)  in Figure~\ref{fig:T}(b) forms a descending staircase, the pattern of the
segments of \( \textbf{T} \) in Figure~\ref{fig:T}(a) is altogether more interesting.
On the torus curve  \( \textbf{T} \), the value of $c$ is generally increasing with the delay $\tau$,
which accords with the concept of delay-induced instability,  
but it does so in a two steps forward, one step backward pattern. This manifests as a large increase in $c$ as increasing $\tau$ crosses a half integer
$\tau=n+\tfrac12$, $n\in\N$, and a small decrease when $\tau$ crosses an integer value. 
While it may be surprising to see a periodic orbit become more stable by increasing the delay,
Theorem~\ref{thm:Poly} suggests an explanation for this behaviour.
For $\tau=k+\theta$ with $k\in\N$  the nontrivial Floquet multipliers are defined by the polynomial $\Delta_{2k+1}$ of odd degree for $\theta\in(0,\tfrac12)$, while for $\theta\in(\tfrac12,1)$ they are defined by the polynomial $\Delta_{2k+2}$ of even degree. It is well-known from recurrence relations (see for example the first Dahlquist barrier for zero-stable linear multistep methods \cite{HNWI}) that even-degree polynomials tend to be more stable than odd-degree polynomials, and we suspect that to play a role here. 

For $\tau<\tfrac12$,
the situation, as illustrated in Figure~\ref{fig:T}, is more complicated. 
For $\tau \in[\frac{1}{2} - \hat{\theta}, \hat{\theta}]$, 
the stability boundary for symmetric 1:1-periodic orbits with $\alpha \in[-\frac{1}{4},\frac{1}{4}]$ is the curve 
\textbf{SN},  or equivalently $c^2=u^2(\tau)$, 
which by Theorem~\ref{thm:po} is also the boundary of the parameter domain on which they exist.
On the curve 
\textbf{SN}, 
there is 
fold bifurcation of periodic orbits between the stable 1:1-periodic orbit
with phase $\alpha\in[-\tfrac14,\tfrac14]$ and the unstable 1:1-periodic orbit
(established in Theorem~\ref{thm:po2} and shown to be unstable in Theorem~\ref{thm:unstab11}). This fold bifurcation \textbf{SN} was already illustrated in Figure~\ref{fig:RegEx2} and discussed after Theorem~\ref{thm:po2}.

The \textbf{SN} and \textbf{T} curves define all of the stability boundary except for the intervals  \( \tau \in [0, \tfrac{1}{2} - \hat{\theta}) \) and \( \tau \in (\hat{\theta}, \tfrac{1}{2}] \),
labelled \textbf{BC} in Figure~\ref{fig:T}(a).
For $\tau$ in this interval as $c$ is decreased, as we already saw in Figure~\ref{fig:RegEx}(e) in Section~\ref{sec:existsym11pos14}, 
on the curve $c^2=w_-^2(\tau)$ where $w_-$ is defined by \eqref{eq:boundaryIc(theta)}, the symmetric 1:1 periodic solution gains two additional degenerate zeros.
Below this curve, in the region labelled I in Figure~\ref{fig:RegEx}(a),
symmetric 1:1-periodic orbits no longer exist. 

In discontinuous delayed systems, such as the psGZT model \eqref{eq:iGZT}, the additional complexities that arise due to the presence of discontinuities can lead to border-collision bifurcations, which fall outside the scope of standard Floquet theory. In such systems, the Floquet multipliers may have discontinuities and do not necessarily move continuously around the unit circle \cite{banerjee1999border,LEINE2002259,RKA20}. 

It is thus not clear what type of bifurcation occurs on the curve \textbf{BC} for $\tau<\tfrac12$, or even if there is a bifurcation at all. The non-trivial Floquet multipliers given by \eqref{eq:polynomiallambda2kp1} all remain inside the unit circle above the straight lines defined by 
\( c^2 = u^2(\tau)\), 
but Proposition~\ref{prop:cEformulation} does not apply 
for $u^2(\tau)\leq c^2\leq w_-^2(\tau)$, 
because of the non-existence of the 1:1-periodic orbits for 
$c^2< w_-^2(\tau)$, 
and the degenerate zeros of the solution for 
$c^2= w_-^2({\tau})$. 
For a smooth dynamical system, a periodic orbit acquiring additional zeros in its profile does not, in general, change its stability and therefore, this is not normally considered to be a bifurcation.
However, for the non-smooth system \eqref{eq:iGZT}, the transition along $c^2 = w_-^2(\tau)$, where additional degenerate zeros are created, may be defined as a border-collision bifurcation \cite{banerjee1999border}, potentially creating a 1:3 orbit. Our constructions and analysis of periodic orbits and their stability from Section~\ref{section:sym11} onwards all require that the periodic orbit have only one upward zero per period, and so they cannot be directly applied to study 1:3 periodic orbits. 
Theorem~\ref{thm:ivpsol} can be used to numerically explore the existence of 
1:3 periodic orbits, but we were not able to identify initial conditions in region I that lead to them. Moreover, simulations suggest the existence of seemingly non-periodic or extremely high-period orbits for $c^2<w^2_-(\tau)$.

\begin{figure}
    \centering
    \includegraphics[width=\textwidth]{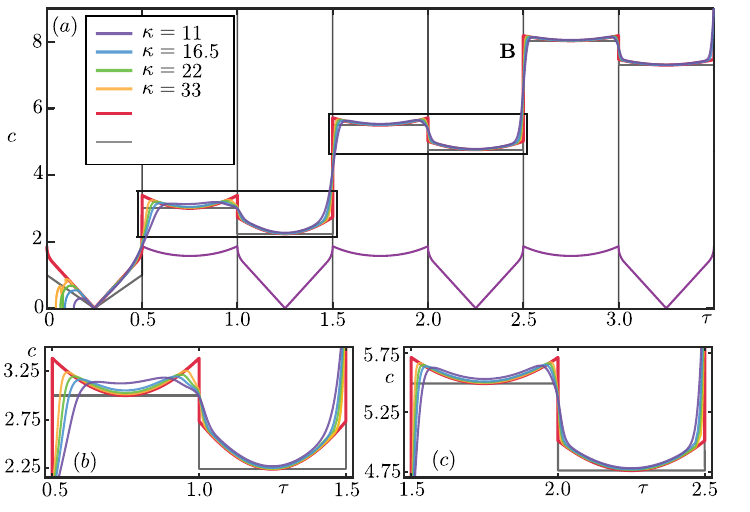}
    \put(-300,195){\small{\eqref{eq:iGZT}}}
    \put(-300,183){\small{\eqref{eq:RKA}}}
    \caption{(a) Comparison in the $(\tau,c)$-plane of the stability boundary $\mathbf{B}$ for 1:1-periodic solutions of the psGZT model \eqref{eq:iGZT} in red, the bifurcation curves of period-one orbits in the smooth GZT model \eqref{eq:GZT} shown for increasing coupling strengths $\kappa = \{11, 16.5, 22, 33\}$, and the bifurcation curve in gray of the signum forced psGZT model \eqref{eq:RKA}. (b) and (c) show enlargements of regions of the main figure. }
    
    \label{fig:Tddebiftool}
\end{figure}

\subsection{Comparison of Stability Domains for Different Variants of the GZT Model}
\label{sec:compare}

In this section, we show that, 
for increasing values of the slope $\kappa$,
the psGZT model \eqref{eq:iGZT} provides a significantly more accurate limiting description of the dynamics of the smooth GZT model~\eqref{eq:GZT} 
than the signum-forced system~\eqref{eq:RKA}
does.
Figure~\ref{fig:Tddebiftool} provides a global comparison in the $(\tau,c)$-plane between the stability boundary~\textbf{B} of the psGZT model~\eqref{eq:iGZT}, the corresponding bifurcation curves of the signum-forced system~\eqref{eq:RKA}, and those of the smooth GZT model~\eqref{eq:GZT} for increasing values of $\kappa \in\{ 11,16.5,22,33\}$.

Figure~\ref{fig:Tddebiftool} shows that 
the stability boundary of the psGZT model
agrees closely with the
bifurcation curves of the smooth GZT model~\eqref{eq:GZT}  
for $\tau>\tfrac14$; here we used \texttt{DDE-BifTool} \cite{sieber2014dde} within \texttt{Matlab} \cite{MATLAB2024b} to compute those bifurcation curves.
The torus bifurcation curve of the GZT model follows the convex segments of the stability boundary~$\mathbf{B}$ of the psGZT system well over intervals of~$\tau$ that do not contain half-integer values, and the agreement improves with growing~$\kappa$.  
Small deviations persist only near the endpoints of these intervals, where $\mathbf{B}$ is vertical.
As highlighted in Figures~\ref{fig:Tddebiftool}(b--c), this agreement becomes even more pronounced for larger delays~$\tau$ and, correspondingly, for larger forcing amplitudes~$c$.  
In these regimes, trajectories cross the zero-line with steeper slopes, which enhances the accuracy of the limiting approximation~\eqref{eq:limit}.

We already explained the convex shape of the segments of $\mathbf{T}$ between half integers for the psGZT model in Section~\ref{sec:stab11igzt} when discussing Figure~\ref{fig:T}. 
Now, the very close alignment of the bifurcation curves of the smooth GZT model \eqref{eq:GZT} and the psGZT model \eqref{eq:iGZT} on these segments suggests that the torus bifurcation curve for the original smooth GZT model arises through the same mechanism. 
In particular, the close agreement between the stability boundaries of both models demonstrates that the psGZT system serves as a powerful analytical proxy for the original GZT model.

The bifurcation curve of the signum-forced system~\eqref{eq:RKA},
in contrast to the GZT and psGZT models, consists of flat segments
with a discontinuities when $\tau$ is a half-integer. 
The stability curves of the 
signum-forced GZT model~\eqref{eq:RKA}
and the psGZT model \eqref{eq:iGZT}
coincide exactly at $\tau = k + \tfrac{1}{4}$ and $\tau = k + \tfrac{3}{4}$ for $k \in \mathbb{N}$.
This contrast arises because, in the signum-forced system~\eqref{eq:RKA}, the smooth forcing term $\cos(2\pi t)$ in~\eqref{eq:iGZT} is replaced by its sign, $\sign(\cos(2\pi t))$.  
Consequently, the phase of the solution remains constant over each open interval $\tau \in (k, k+\tfrac{1}{2})$ and $\tau \in (k+\tfrac{1}{2}, k+1)$ for $k \in \mathbb{N}_0$, producing the characteristic flat segments of the torus curve for the signum-forced GZT model~\eqref{eq:RKA} in Figure~\ref{fig:Tddebiftool}.  
In particular, in the special cases where $u(\theta)=0$ and $v(\theta)=0$, occurring when $\tau=k+\tfrac{1}{4}$ and $\tau=k+\tfrac{3}{4}$, equation~\eqref{eq:T} of the psGZT model \eqref{eq:iGZT} reduces to equation~(19) of~\cite{RKA20} for the signum forced model \eqref{eq:RKA}.  
Hence, when the forcing is replaced by its sign, the smooth dependence of the stability on the phase terms $u(\theta)$ and $v(\theta)$ is lost; the stability properties of the 1:1 periodic orbits then reduce to those of the piecewise-constant (signum-forced) system.

\begin{figure}
    \centering
    \includegraphics[width=\textwidth]{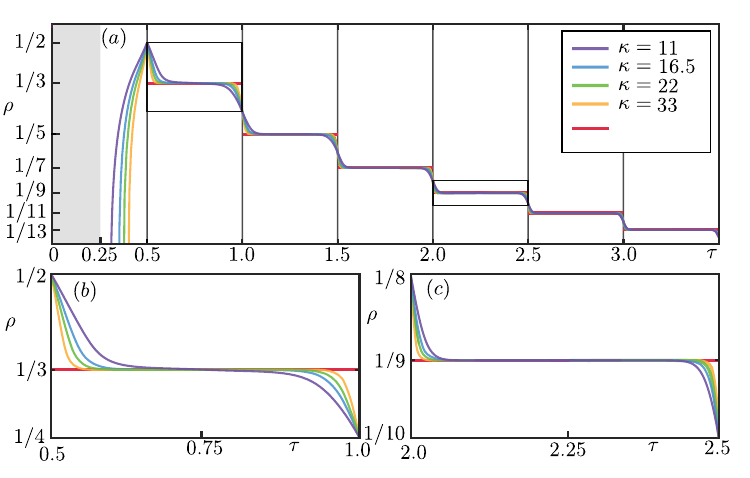}
    \put(-66,171){\small{\eqref{eq:iGZT}}}
    \caption{(a) Comparison for $\tau>\tfrac14$ between
    the rotation number $\rho$ 
    on the stability boundary $\mathbf{B}$ and {on} the bifurcation curves of period-one orbits in the smooth GZT model \eqref{eq:GZT}  shown for increasing coupling strengths $\kappa \in \{11, 16.5, 22, 33\}$. 
    (b) and (c) show enlargements of parts of the main figure (a). }
    \label{fig:Tddebiftoolrho}
\end{figure}

Figure~\ref{fig:Tddebiftoolrho} compares the rotation number~$\rho$ along the stability boundary curve~\textbf{B} for the psGZT model~\eqref{eq:iGZT} with the numerically computed rotation number along the torus bifurcation curve of the smooth GZT model~\eqref{eq:GZT}, as computed with \texttt{DDE-BifTool}.  
The two curves show excellent agreement over the displayed range of panel (a), often being indistinguishable in the figure.  
Both exhibit a staircase-like descent, with each step corresponding to an interval where~$\rho$ remains nearly constant.  
This correspondence becomes increasingly close as either~$\tau$ or~$\kappa$ grows, as illustrated by the enlargement panels~(b) and (c). Overall, Figure~\ref{fig:Tddebiftool}
confirms that the psGZT model also captures the limiting dynamics of the smooth system in its rotation number.

\begin{figure}
    \centering
    \includegraphics[width=\textwidth]{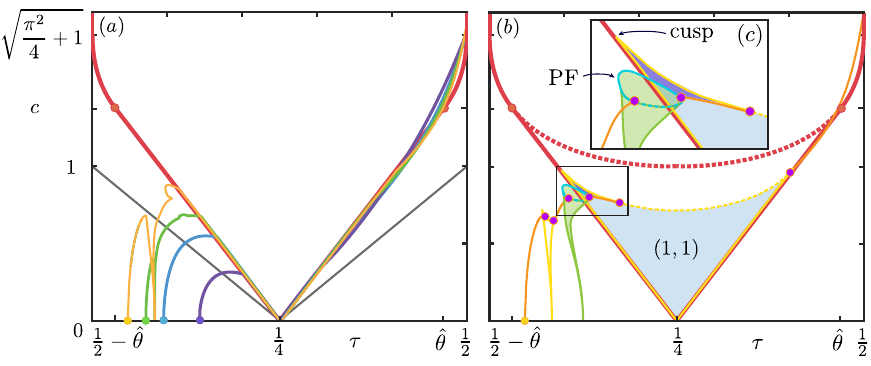}
    \caption{
     (a) Stability boundary \textup{\textbf{B}} for 1:1-periodic solutions of the psGZT model \eqref{eq:iGZT} in the \((\tau,c)\)-plane (red), compared with bifurcation curves of period-one orbits in the smooth GZT model  \eqref{eq:GZT} for increasing \(\kappa = \{11, 16.5, 22, 33\}\) (colors match Figure~\ref{fig:Tddebiftool}), restricted to \(\tau < \tfrac{1}{2}\). 
     The light grey curves denote 
     the stability boundary from \cite{RKA20} for the signum forced psGZT model \eqref{eq:RKA}.
     (b) Comparison between the bifurcations of the psGZT model \eqref{eq:iGZT} (red curves) and 
     the full bifurcation set of the smooth GZT model \eqref{eq:GZT} for \(\kappa=33\), with enlargement (c). Folds of periodic orbits (yellow),  torus bifurcations (orange), pitchfork of periodic orbits (light blue), folds of symmetrically related periodic orbits (green) are shown, along with the (1,1) resonance tongue (light grey region).
     Dashed curves indicate bifurcations involving saddle-type orbits; solid curves if one orbit is stable. 
     Purple dots indicate $(1,1)$ resonance points; codimension-2 bifurcations in which a pair of conjugate Floquet multipliers collide at $+1$.
     The long dashed yellow curve at the top of the (1,1) resonance tongue,
      between two $(1,1)$ resonance points of the GZT model, corresponds to
      the red dashed curve of the psGZT model 
     defined by $c^2=w_+^2(\theta)$ which denotes the upper bound on the region where 1:1-periodic orbits of the psGZT model coexist.
     }
    \label{fig:RT11}
\end{figure}

As discussed after Theorem~\ref{thm:w}, the curve~\textbf{B} for the psGZT model represents a degenerate case where $\rho$ is constant on the open intervals $\theta \in (0, \tfrac{1}{2})$ and $\theta \in (\tfrac{1}{2},1)$.  
In contrast, $\rho$ varies continuously with~$\tau$ in the smooth GZT model, taking the values $\tfrac{1}{4k+1}$ and $\tfrac{1}{4k+3}$ at $\tau = k + \tfrac{1}{4}$ and $\tau = k + \tfrac{3}{4}$, respectively.  
As~$\tau$ increases, $\rho(\tau)$ flattens near these points but transitions sharply near half-integer values of $\tau$, where the psGZT curve~\textbf{B} is discontinuous.  
In particular, for the smooth GZT model, $\rho = \tfrac{1}{2k}$ whenever $\tau = \tfrac{k}{2}$, illustrating how the continuous dependence of~$\rho$ on~$\tau$ smooths out the staircase pattern observed in the piecewise-smooth limit.

We now turn to the comparison of the three models for small delays $\tau < \tfrac{1}{2}$, where the dynamics are more complicated.  
Figure~\ref{fig:RT11}(a) shows an enlargement of the $(\tau,c)$-plane from Figure~\ref{fig:Tddebiftool}(a), revealing that the three models differ significantly in this $\tau$-range.  
The stability boundaries of the psGZT model~\eqref{eq:iGZT} and the signum-forced model~\eqref{eq:RKA} contrast sharply: for the latter, the gray curve representing its stability boundary has a smaller slope than that of the psGZT model.  
As shown in~\cite{RKA20}, this gray curve corresponds to a fold bifurcation of the 1:1-periodic orbit in the signum-forced GZT system~\eqref{eq:RKA}.  
By contrast, as discussed at the end of Section~\ref{sec:stab11igzt}, for the psGZT model the stability boundary is determined by a fold occurring for $\tau\in(\tfrac12-\hat{\theta},\hat{\theta})$, and by the appearance of new zeros along the curve $c^2=w_-^2(\tau)$ for other values of $\tau$.

For the smooth GZT model~\eqref{eq:GZT}, the bifurcation curves that define the stability boundary approach those of the psGZT model as $\kappa$ increases.
However, we observe that the smooth model \eqref{eq:GZT} exhibits considerably more structure than the psGZT model~\eqref{eq:iGZT} when $\tau<\tfrac14$.

To explain this intricate structure, Figure~\ref{fig:RT11}(b) displays the complete bifurcation diagram for $\kappa = 33$ and contrasts it with the corresponding bifurcation set of the psGZT model.  
Following the torus bifurcation curve from $\tau=\tfrac12$ towards smaller $\tau$-values, we find that near $(\tau,c)\approx(0.4,1)$ the curve reaches a $(1,1)$ resonance point.  
From this point, two fold bifurcation curves emerge: one between two unstable periodic orbits (the dashed curve bounding the top of the light gray region labeled $(1,1)$), and another between a stable and an unstable orbit (the solid curve forming the lower boundary of the same region).  
The latter curve traces the stability boundary~\textbf{B}; it touches the zero-forcing line at $(\tau,c)=(\tfrac14,0)$, and then bends upward along~\textbf{B}, reaching the cusp point marked in Figure~\ref{fig:RT11}(c).  
This fold curve ultimately reconnects with the second fold of unstable periodic orbits at another $(1,1)$ resonance point (the rightmost one in Figure~\ref{fig:RT11}(c)), from which a new torus bifurcation curve emerges.  
The fold and torus bifurcation curves thus enclose a region—shown in dark blue in the inset of Figure~\ref{fig:RT11}(c)— with bistability between two period-one orbits.  
This phenomenon contrasts sharply with the psGZT model~\eqref{eq:iGZT}, for which bistability of 1:1-periodic orbits is ruled out by Theorem~\ref{thm:unstab11}.

Returning to Figure~\ref{fig:RT11}(b), 
the two fold curves form a $(1,1)$ resonance tongue, depicted in light gray.  
Within this region, trajectories 
of the GZT model \eqref{eq:GZT} converge to periodic orbits which 
are $(1,1)$-locked: these periodic orbits lie on an invariant torus and complete one revolution in each torus direction before returning to their initial position.  
Importantly, these $(1,1)$-locked orbits are not necessarily 1:1-periodic.  
The upper boundary of this resonance tongue  
lies just below
the curve $c^2=w_+^2(\tau)$ (and its reflection about $\tau=\tfrac14$), which  defines the upper boundary of the set~$\mathcal{S}$,
for the psGZT model \eqref{eq:iGZT} from Section~\ref{section:existencesymmetric1434}.
Inside~$\mathcal{S}$, 1:1-periodic orbits with phase $\alpha\in[\tfrac14,\tfrac34]\ (\mathrm{mod}\ 1)$ and $\alpha\in[-\tfrac14,\tfrac14]\ (\mathrm{mod}\ 1)$ coexist in the psGZT model.
In the smooth GZT model~\eqref{eq:GZT}, the upper boundary of the $(1,1)$ resonance tongue corresponds to a fold of 1:1-periodic orbits.  
At the upper edge of~$\mathcal{S}$, the corresponding orbits of the psGZT model develop degenerate zeros.
The nature of this transition is not yet clear but, as discussed at the end of Section~\ref{sec:stab11igzt} for the curve $c^2=w_-^2(\tau)$, 
a border-collision bifurcation may also occur along the curve $c^2=w_+^2(\tau)$.
A companion paper~\cite{GZTn} will explore these resonance regions associated with higher-period orbits of the psGZT model~\eqref{eq:iGZT}, focusing on transitions between periodic orbits as the number of zeros changes.

Finally, the torus bifurcation curve emerging to the left of the main resonance tongue in Figure~\ref{fig:RT11}(b) follows a sequence of $(1,1)$ resonance points, from which two additional resonance tongues arise.  
The rightmost of these, shaded green, contains symmetrically related $(1,1)$ periodic orbits—that is, pairs of periodic orbits connected by the symmetry~\eqref{eq:sym}.  
Such symmetry-related orbits do not exist in the psGZT model~\eqref{eq:iGZT}.  
As seen in Figure~\ref{fig:RT11}(a), these intricate bifurcation structures in the smooth GZT model for $\tau<\tfrac14$ shift progressively to smaller $\tau$ values as $\kappa$ increases.  
This behaviour is consistent with the fact that the torus bifurcation curves of the GZT model are anchored to the zero-forcing line $c=0$ at $\tau=\tfrac{\pi}{2\kappa}$ (see Section~2.2 of~\cite{SamMSc}).  
Since the psGZT model corresponds to the $\kappa\to\infty$ limit of the GZT system, this likely explains why such additional resonance structures are absent in the psGZT model~\eqref{eq:iGZT}.

These observations demonstrate that the psGZT model \eqref{eq:iGZT} provides a far better limiting description of the dynamics of the smooth GZT system \eqref{eq:GZT} than the signum-forced model~\eqref{eq:RKA} does.  
The close alignment in terms of both the bifurcation and rotation-number structure across a wide range of parameters shows that the piecewise-smooth approximation~\eqref{eq:iGZT} accurately represents the 
 geometry of the stability boundary of the smooth GZT model \eqref{eq:GZT}.

\section{Summary}

We introduced and studied the psGZT model \eqref{eq:iGZT}, which is a DDE with smooth periodic forcing and piecewise-constant time-delayed feedback. The psGZT model is 
a simplification of the smooth GZT model \eqref{eq:GZT}, which is 
a conceptual 
model for
ENSO.
The dynamics and bifurcations of our psGZT model \eqref{eq:iGZT}
can be analyzed mathematically
rather than relying solely on
numerical solutions, as is the case for the GZT model \eqref{eq:GZT}.
For the psGZT model, we constructed explicit solutions for initial value problems and derived closed-form expressions for period-one orbits.

By exploiting the properties and symmetries of the model, 
we determined the regions of parameter space where period-one orbits exist
as well as
the precise conditions under which they lose stability. We showed that stability
is characterized by a linear mapping that tracks the zeros of perturbations, yielding the Floquet multipliers of periodic solutions. 
This analysis revealed that stable period-one orbits mainly lose stability through a torus bifurcation, for which we obtained a closed-form expression. For small delays, however, stability is lost through a fold bifurcation, giving rise to coexisting period-one orbits.

We compared the bifurcation curves of the psGZT model with those of the signum-forced approximation \eqref{eq:RKA} and those of the smooth GZT model \eqref{eq:GZT}.
The psGZT model \eqref{eq:iGZT}, by incorporating smooth periodic forcing, is
the large $\kappa$ limit of
the GZT model \eqref{eq:GZT}.
The 
signum forced psGZT model \eqref{eq:RKA}, in contrast only admits a discrete set of velocities, and cannot be obtained by taking a parameter limit in the GZT model \eqref{eq:GZT}.
We
showed that the psGZT model reproduces the bifurcation structures of the GZT model \eqref{eq:GZT} much more faithfully
for sufficiently large $\kappa$. 
Namely, we found close agreement between the bifurcation curves of the psGZT and GZT models over wide parameter ranges, which improved as the coupling strength $\kappa$ and the delay increased. 
This demonstrates that the psGZT model provides a powerful analytical proxy for the GZT model, preserving the essential ENSO-like dynamics while remaining tractable to rigorous study.
Our explicit characterization of the bifurcation curves for the psGZT model \eqref{eq:iGZT} provides insight into the bifurcations of the smooth GZT model \eqref{eq:GZT},
providing a better understanding of the delay-driven oscillations relevant to ENSO dynamics, as we demonstrated by our comparison.

The psGZT model \eqref{eq:iGZT} reveals a rich landscape of dynamics, of which we only considered period-one orbits in this work. In a companion paper
\cite{GZTn}, we extend this analysis to higher odd-period orbits and study frequency-locking mechanisms emerging from the interaction of delayed feedback and seasonal forcing. 

For future work, it would be interesting to incorporate additional effects, such as asymmetry of the coupling function, positive feedback, and even state-dependent delays. Studies of the smooth GZT model suggest that such extensions lead to significantly more complex and chaotic dynamics, which are closer to ENSO observations \cite{KKD19,KKP16}.

\section*{Acknowledgments}
The authors would like to thank Bernd Krauskopf for introducing us to the GZT model, and for his generous feedback
throughout this project. We also grateful to Stefan Ruschel for his comments on an early draft of this paper.


\appendix
\section*{Appendices}
\renewcommand{\theequation}{\thesection.\arabic{equation}}
\renewcommand\thefigure{\thesection.\arabic{figure}} 


\section{Proofs of main results}\label{sec-proofs}

\subsection{Proof of Theorem~\ref{thm:po}}
\label{sec:pf:po}
If $\alpha\in[-\frac{1}{4},\frac{1}{4}] \pmod 1$ then translating the solution by an integer number of periods we obtain $\alpha\in[-\frac{1}{4},\frac{1}{4}]$; hence without loss of generality we always assume that $\alpha$ is chosen so that 
 $\alpha\in[-\frac{1}{4},\frac{1}{4}]$ in this proof.

In Theorem~\ref{thm:po_closedform} we already derived the profiles of candidate 1:1-periodic orbits. The profiles in \eqref{eq:h11}-\eqref{eq:phalf1} were constructed so that $h(\alpha)=h(\alpha+\tfrac12)=0$, but this used  the assumption that \eqref{eq:hsigns} holds. Now, we need to determine for which parameter values 
the profiles defined in \eqref{eq:h11}-\eqref{eq:phalf1} satisfy \eqref{eq:hsigns}  to ensure that the constructed profiles are actually solutions of the psGZT model \eqref{eq:iGZT}. We will proceed by first showing that the profiles defined in 
\eqref{eq:h11}-\eqref{eq:phalf1} all satisfy the symmetry property
\eqref{eq:sym}. After this we will show that 
\be \label{eq:hleftsign}
h(t)>0, \textit{ a.e.~on } t\in(\alpha,\alpha+\tfrac12).
\ee
From \eqref{eq:hleftsign} and the symmetry \eqref{eq:sym} it then follows that 
$h(t)<0$ for almost all $t\in(\alpha+\frac12,\alpha+1)$, and the profile does indeed define a symmetric 1:1-periodic orbit.

First, we show that the profile $h(t)$ 
defined by \eqref{eq:h11}-\eqref{eq:phalf1}
satisfies \eqref{eq:sym} and hence is symmetric. It suffices to show that 
\be \label{eq:hsym}
h(t) + h(t + \tfrac{1}{2}) = 0, \quad\text{for all }t \in [\alpha, \alpha + 1).
\ee
From \eqref{eq:h11} we compute
$$
h(t) + h(t + \tfrac{1}{2}) = p(t) + p(t + \tfrac{1}{2}) - \frac{c}{\pi} \sin(2\pi \alpha).
$$
From \eqref{eq:p0}-\eqref{eq:phalf1}
it is easy verify that for all 
$t\in[\alpha,\alpha+1)$
$$
p(t) + p(t + \tfrac{1}{2}) =
\begin{cases}
2\theta - \tfrac{1}{2}, & \text{for } \theta \in [0, \tfrac{1}{2}), \\
\tfrac{3}{2} - 2\theta, & \text{for } \theta \in [\tfrac{1}{2}, 1).
\end{cases}
$$
Then the phase condition \eqref{eq:alpha0}
ensures that \eqref{eq:hsym} is satisfied for all $\theta\in[0,\tfrac12)$, while the 
phase condition \eqref{eq:alpha1}
ensures that \eqref{eq:hsym} is also satisfied for all $\theta\in[\tfrac12,1)$.
Thus, the profile is symmetric whenever such a value of $\alpha$ exists, which it does, if and only if,
the constraint \eqref{eq:necc} holds.
Therefore, in the remainder of the analysis, we assume $\theta$ and $c$ satisfy \eqref{eq:necc}.

We now can use the symmetry of $h(t)$ to confirm that it satisfies the system \eqref{eq:iGZTth}. Specifically, it suffices to check that there exists an $\alpha$ satisfying \eqref{eq:alpha0} or \eqref{eq:alpha1}, such that $h(\alpha) = 0$, and that \eqref{eq:hleftsign} is satisfied.
If these conditions hold, then by symmetry, $h(t)$ changes sign exactly at $t = \alpha$ and $t = \alpha + \tfrac{1}{2}$. Then, by Theorem~\ref{thm:po_closedform}, $h(t)$ satisfies \eqref{eq:iGZT} and defines a 1:1-periodic orbit. Because we assume \eqref{eq:necc}, a suitable phase $\alpha$ always exists, and the condition $h(\alpha) = 0$ is automatically satisfied. Thus, to complete the proof, it remains only to determine necessary and sufficient 
conditions on the parameters $\theta$ and $c$ such that \eqref{eq:hleftsign} holds.

The proof distinguishes seven cases, depending on the value of the parameter $\theta$, 
namely
$\theta = 0$, $\theta = \tfrac{1}{2}$, $\theta \in (\tfrac{1}{4}, \tfrac{1}{2})$, $\theta \in (0, \tfrac{1}{4})$, $\theta = \tfrac{1}{4}$, $\theta \in (\tfrac{1}{2}, \tfrac34]$, and $\theta \in (\tfrac{3}{4}, 1)$. 
Fortunately, due to the symmetry relation developed in Corollary~\ref{cor:sym2}, 
we only need to consider four of these cases.

The result for $\theta\in(0,\tfrac14)$ and for $\theta=\tfrac12$
follows from the results for $\theta\in(\tfrac14,\tfrac12)$ and for $\theta=0$, respectively, by Corollary~\ref{cor:sym2}(1),
while the result for $\theta \in [\tfrac{3}{4}, 1)$  follows from
$\theta \in (\tfrac{1}{2}, \tfrac34]$ by Corollary~\ref{cor:sym2}(2). We now consider the remaining four cases.

\textbf{Case \boldmath$\theta=0$:} 
We show that $h(t)>0$ for all $t\in(\alpha,\alpha+\tfrac{1}{2})$. 

The orbit $h(t)$ has discontinuous derivative at $t=\alpha$ and $t=\alpha+\tfrac{1}{2}$, so we consider a one-sided derivative, $h'(\alpha^+)$, where $\alpha^+$ defines the right hand limit as $t$ approaches $\alpha$ from above. 
For $t\in(\alpha,\alpha+\tfrac12)$,
from \eqref{eq:h11} and \eqref{eq:p0}, it follows that
$h'(t)=-1+c\cos(2\pi t)$. 
Then
$$
h'(\alpha^+)=
\lim_{t\to\alpha^+}h'(t)=
-1+c\cos(2\pi\alpha).$$
But $h'(\alpha^+)\geq0$ is a necessary condition for \eqref{eq:hleftsign} to hold, and 
$h'(\alpha^+)\geq0$
is equivalent to
\be \label{eq:th0cgt1}
c\cos(2\pi\alpha)\geq1.
\ee
But, by \eqref{eq:alpha0}, for $\theta=0$ we have
$$
c^2=c^2\cos(2\pi\alpha)+c^2\sin(2\pi\alpha)=c^2\cos(2\pi\alpha) + {u}^2(0)
=c^2\cos(2\pi\alpha)+\frac{\pi}{4}^2.$$
Therefore, in the case $\theta=0$, equation \eqref{eq:c2pi2}
is equivalent to $h'(\alpha^+)\geq0$.

To satisfy \eqref{eq:alpha0} with $\theta=0$ we require $\sin(2\pi\alpha)<0$ and, along with \eqref{eq:th0cgt1},
this implies that
$\alpha\in(-\tfrac14,0)$, 
for which \eqref{eq:alpha0} then defines a unique value of $\alpha$. 
But then, for all $t\in(\alpha,-\alpha)$,
we have 
$c\cos(2\pi t)>c\cos(2\pi\alpha)\geq1$,
which implies that $h'(t)>0$ for all 
$t\in(\alpha,-\alpha)$.

Because of the properties of the 
cosine function there exists a
$\beta$ with $0<-\alpha<\beta<\tfrac12$ such that $c\cos(2\pi\beta)=1$
and $c\cos(2\pi t)<1$ for $t\in(\beta,1-\beta)$. It follows that $h'(t)>0$ for $t\in(\alpha,\beta)$ and $h'(t)<0$ for $t\in(\beta,1-\beta)$.
But $1-\beta>1/2>\alpha+1/2$, so $h(t)$ is a unimodal function on the interval $(\alpha,\alpha+\tfrac12)$, with $h(\alpha)=h(\alpha+\tfrac12)=0$, and the result for this case follows.

\textbf{Case \boldmath$\theta\in(\frac{1}{4},\tfrac12)$:}  
The solution $h(t)$ defined by \eqref{eq:h11} and \eqref{eq:p0half}
has discontinuous derivative at
$t=\alpha+\theta$.

First, we establish that $h(t)>0$ for all $t\in[\alpha+\theta,\alpha+\tfrac{1}{2})$. 
We observe that for $\theta\in(\tfrac14,\tfrac12)$
the right-hand side of the phase condition \eqref{eq:alpha0}
is positive, and it follows that $\alpha\in(0,\tfrac{1}{4}]$.
Then, because $t\in(\alpha+\theta,\alpha+\tfrac{1}{2})\subset(\tfrac{1}{4},\tfrac{3}{4})$ and $\cos{(2\pi t)}\leq 0$ for all $t\in(\tfrac{1}{4},\tfrac{3}{4})$, we obtain
$h'(t)=-1+c\cos(2\pi t)\leq-1<0$ for all
 $t\in(\alpha+\theta,\alpha+\tfrac{1}{2})$. Together with
$h(\alpha+\tfrac{1}{2})=0$, this establishes that
$h(t)>0$ for all $t\in[\alpha+\theta,\alpha+\tfrac{1}{2})$.

It remains to consider the behaviour of $h(t)$ for 
$t\in(\alpha,\alpha+\theta)$. 
As noted above, 
$\alpha\in(0,\tfrac{1}{4}]$ and, hence, $\cos(2\pi\alpha)>0$.
The orbit $h(t)$ is smooth around $t=\alpha$, and
from \eqref{eq:hdashalph}, 
$h'(\alpha)=1+c\cos(2\pi\alpha)>1$. 
It also follows from \eqref{eq:h11} and
\eqref{eq:p0half} that $h''(t)=-2\pi c\sin(2\pi t)$. Thus $h''(t)$ has at most one zero for $t\in(\alpha,\alpha+\theta)\subset(\alpha,\alpha+\tfrac12)$ and, by the generalised Rolle's theorem, $h'(t)$ has at most two zeros on this interval. 

If $h'(t)$ has no zeros in the interval  $(\alpha,\alpha+\theta)$ then $h(t)$ is strictly monotonically increasing for $t\in(\alpha,\alpha+\theta)$ and $h(t)>h(\alpha)=0$ for all  $t\in(\alpha,\alpha+\theta)$.
If $h'(t)$ has one zero, $t_1$, in the interval  $(\alpha,\alpha+\theta)$
then $h(t)$ is unimodal on $(\alpha,\alpha+\theta)$ with a local maxima at $t=t_1\in(\alpha,\alpha+\theta)$. Since we already established that $h(\alpha)=0<h(\alpha+\theta)$ this again establishes that
$h(t)>0$ for all  $t\in(\alpha,\alpha+\theta)$.

The only remaining and only interesting case is when there exist $t_1$, $t_2$ with 
$\alpha<t_1<t_2<\alpha+\theta$ such that $h(t)$ has a local maxima at $t=t_1$, followed by a local minima at $t=t_2$.
If this occurs we require $h(t_2)\geq0$ to ensure that
$h(t)>0$ almost everywhere on $(\alpha,\alpha+\theta)$ and, hence, on
$(\alpha,\alpha+\tfrac12)$. Since $c\leq1$ ensures that $h'(t)\geq0$ for $t\in(\alpha,\alpha+\tfrac12)$, we only need to consider the case where $c>1$.

When $c>1$, there exists $t_1\in(\frac{1}{4},\frac{1}{2})$
and $t_2\in(\tfrac{1}{2},\tfrac{3}{4})$ such that $\cos(2\pi t_1)=\cos(2\pi t_2)=-1/c$. 
Now, $h'(t)>0$ for
$t\in(0,t_1)$ and $h'(t)<0$ for $t\in(t_1,t_2)$. 
We already showed that $\alpha\in(0,\tfrac14]$ which
implies $t_1>\alpha$.
The only way for the profile not to be a valid solution is if
\be\label{eq:nonvalidt2}
 t_2 < \alpha+\theta\quad\textnormal{and}\quad h(t_2)< 0.
\ee
To complete the proof we first identify the region in $(\theta,c)$-parameter space where $t_2<\alpha+\theta$, and then identify the subset of this region on which $h(t_2)<0$.

We first find two conditions on the parameters which are both equivalent to $t_2<\alpha+\theta$. The region in $(\theta,c)$-space where this holds is labelled II
in Figure~\ref{fig:RegEx}(a). 
Since we are considering $\theta\in(\tfrac14,\tfrac12)$ for which $\alpha\in(0,\tfrac14]$, the condition $t_2<\alpha+\theta$
is equivalent to $\tfrac12<t_2<\alpha+\theta<\tfrac34$.
Because $\sin{(2\pi t)}$ is strictly decreasing for
$t\in(\tfrac{1}{2},\tfrac{3}{4})$, 
we have that $t_2<\alpha+\theta$ if and only if
\be \label{eq:unimod0sin}
\sin(2\pi(\alpha+\theta))< \sin(2\pi t_2)=-\sqrt{1-1/c^2}.
\ee
The phase condition \eqref{eq:alpha0} and $\alpha\in(0,\tfrac14]$ imply that \eqref{eq:c0proof} holds. Then, 
\eqref{eq:unimod0sin} is equivalent to
\begin{align}\label{eq:unimod0sin2}-\sqrt{c^2-1}  > c\sin(2\pi(\alpha+\theta)) &=c\cos(2\pi\alpha)\sin(2\pi\theta)+c\sin(2\pi\alpha)\cos(2\pi\theta)\nonumber\\&=\sqrt{c^2-u^2(\theta)}\sin(2\pi\theta)+u(\theta)\cos(2\pi\theta).
\end{align}

Alternatively, because
$\cos{(2\pi t)}$ is increasing for $t\in(\tfrac{1}{2},\tfrac{3}{4})$, the condition $t_2<\alpha+\theta$ is also 
equivalent to
\begin{equation}
    \label{eq:unimod0cos}
    \cos(2\pi(\alpha+\theta))>\cos(2\pi t_2)=-\frac1c,
\end{equation}
or equivalently
\begin{align}\label{eq:unimod0cos2}
-1  < c\cos(2\pi(\alpha+\theta)) &=c\cos(2\pi\alpha)\cos(2\pi\theta)-c\sin(2\pi\alpha)\sin(2\pi\theta) \notag\\
&=\sqrt{c^2-u^2(\theta)}\cos(2\pi\theta)-u(\theta)\sin(2\pi\theta).
\end{align}
Rearranging we see that
$$0<-\sqrt{c^2-u^2(\theta)}\cos(2\pi\theta)<
1-u(\theta)\sin(2\pi\theta),$$
and, hence,
\be \label{eq:boundaryII}
c^2<u^2(\theta)+\left(\frac{1-u(\theta)\sin(2\pi\theta)}{\cos(2\pi\theta)}\right)^2
= 1 + \left(\frac{\sin(2\pi\theta)-u(\theta)}{\cos(2\pi\theta)}\right)^2.
\ee

Equation~\eqref{eq:boundaryII} along with the constraints $c>1$ and $c^2\geq u^2(\theta)$ defines the region II of the $(\theta,c)$-plane 
where the periodic orbit has a local minimum $t_2$ with $t_2\in(\alpha, \alpha + \theta)$,
which is illustrated in Figure~\ref{fig:RegEx}(a).
At the left end of this region, the right-hand side of \eqref{eq:boundaryII} is equal to $1$, so $\sin(2\pi\theta)=u(\theta)$. Since $\sin(2\pi\theta)$ is decreasing and $u(\theta)$ is increasing for $t\in[\tfrac14,\tfrac12]$ this point is uniquely defined, with $\theta\approx0.3676$.

Next, let us investigate the sign of $h(t)$ when $t_2<\alpha+\theta$.
The value of 
$\sin(2\pi\alpha)$ is defined by the phase condition \eqref{eq:alpha0}, while
$\cos(2\pi t_2)=-1/c$ and $t_2\in(\tfrac12,\tfrac34)$ implies
$\sin(2\pi t_2)=-\sqrt{1-1/c^2}$.
When $t_2<\alpha+\theta$, the
solution $h(t_2)$ defined by \eqref{eq:p0half} can be written as
\be
\label{eq:h(t_2)}
h(t_2)=t_2-\alpha+\frac{c}{2\pi}(\sin(2\pi t_2)-\sin(2\pi\alpha))
=(t_2-\alpha-\theta)+\frac14-\frac{1}{2\pi}\sqrt{c^2-1}.
\ee
Then
\be
\label{eq:h(t_2)<0 iff sin}
h(t_2)<0\quad\text{if and only if} \quad 2\pi(t_2-\theta-\alpha) < \sqrt{c^2-1}-\frac{\pi}{2}
\ee
and $\tfrac12<t_2<\alpha+\theta<\tfrac34$ 
implies that $-\tfrac\pi2<2\pi(t_2-\theta-\alpha)$.
With $\sin(2\pi\theta)>0>\cos(2\pi\theta)>-1$
and $u(\theta)<\tfrac\pi2$ for $\theta\in(\tfrac14,\tfrac12)$, it follows from 
\eqref{eq:unimod0sin2} that
\begin{align*}
\sqrt{c^2 - 1} - \frac{\pi}{2} & < -\sqrt{c^2 - u^2(\theta)} \sin(2\pi\theta) - u(\theta)\cos(2\pi\theta) - \frac{\pi}{2}
\leq
 - u(\theta)\cos(2\pi\theta) - \frac{\pi}{2}\\
& <u(\theta) - \frac{\pi}{2}
<0.
\end{align*}

Since $\cos(2\pi t)$ is increasing for $t\in(-\tfrac14,0)$, 
applying the cosine 
function to both sides of the second inequality in \eqref{eq:h(t_2)<0 iff sin} 
preserves the inequality;  so $h(t_2)<0$, if and only if, 
\begin{align*}
    c^2\cos\big(\sqrt{c^2-1}-\tfrac\pi2\big)
    &>c^2\cos{(2\pi\left(t_2-\theta-\alpha\right))}\\
    &=c^2\cos{(2\pi t_2)}\cos{(2\pi(\alpha+\theta))+c^2\sin{(2\pi t_2)}\sin{(2\pi(\theta+\alpha))}}\\
    &=-c\cos{(2\pi(\alpha+\theta))-c\sqrt{c^2-1}\sin{(2\pi(\theta+\alpha))}}\\
    &=-\sqrt{c^2-{{u}^2(\theta)}}\cos(2\pi\theta)+{{u}(\theta)}\sin(2\pi\theta)\\
    &\hspace*{5em}-\sqrt{c^2-1}\left(\sqrt{c^2-{u}^2(\theta)}\sin{(2\pi\theta)}+{{u}(\theta)}\cos{(2\pi\theta)}\right).
\end{align*}
Since $\sin{(\sqrt{c^2-1})}=\cos{\left(\sqrt{c^2-1}-\tfrac\pi2\right)}$, we obtain that $h(t_2)<0$, if and only if,
\begin{align}
\label{eq:boundaryI}
c^2\sin{(\sqrt{c^2-1})}&>{u}(\theta) \left( \sin(2\pi \theta) - \cos(2\pi \theta) \sqrt{c^2 - 1} \right) \\
&\hspace*{5em}-\sqrt{c^2 - {u}^2(\theta)} \left( \sin(2\pi \theta) \sqrt{c^2 - 1} + \cos(2\pi \theta) \right).\notag
\end{align}

To analyze this condition, 
for $(\theta,c)$ in region I such that $t_2<\alpha+\theta$ we define 
\begin{align}
\label{eq:W-}
W_-(\theta,c)&={u}(\theta) \left( \sin(2\pi \theta) - \cos(2\pi \theta) \sqrt{c^2 - 1} \right) \\
&\hspace*{3em}-\sqrt{c^2 - {u}^2(\theta)} \left( \sin(2\pi \theta) \sqrt{c^2 - 1} + \cos(2\pi \theta) \right)-c^2\sin{(\sqrt{c^2-1})}.\notag
\end{align}
Then, from \eqref{eq:boundaryI}, it follows that $W_-(\theta,c)$ has the same sign as $h(t_2)$, and the zero level set $W_-(\theta,c) = 0$ 
defines the boundary curve of the region I where $h(t_2)<0$
within the region II where $t_2<\alpha+\theta$. Both of these regions are illustrated in Figure~\ref{fig:RegEx}(a).
(Note that, while the curve $c^2=u^2(\theta)+1$ formally satisfies $W_-(\theta,c)=0$, this curve lies completely outside region II \eqref{eq:boundaryII} and is not relevant to our analysis.)

We define the curve $c^2=w_-^2(\theta)$ implicitly by $W_-(\theta,w_-^2(\theta)) = 0$. 
The function $w_-^2(\theta)$ is then given by \eqref{eq:boundaryIc(theta)} for $\theta \in [\hat{\theta},\tfrac12]$. The threshold value $\hat\theta$ is defined in \eqref{eq:thetahat} as the value of $\theta$ such that $w_-^2(\hat\theta)=u^2(\hat\theta)$. The function $w_-^2(\theta)$ is easily seen to satisfy the conditions \eqref{eq:w^-2boundary} for the boundary. 

To determine which side of the boundary curve $w_-^2(\theta)$ corresponds to $h(t_2) < 0$, it is sufficient to evaluate $W_-$ at two points on either side of this curve. Doing this at two points on the lower boundary of region II where $c^2=u^2(\theta)$ we find 
$W_-(\tfrac12,\tfrac\pi2)<0<W_-(\tfrac{1}{2\pi}+\tfrac14,1)$.
Thus, by continuity we conclude that within region II (where $t_2<\alpha+\theta$),
the solution satisfies $h(t_2)<0$ for $c^2<w_-^2(\theta)$ and $h(t_2)\geq0$ for $c^2 \geq w_-^2(\theta)$. Hence, the candidate 
periodic solution defined by \eqref{eq:h11} and \eqref{eq:p0half}
is invalid for $u^2(\theta)\leq c^2<w_-^2(\theta)$ when 
$\theta\in(\hat\theta,\tfrac12)$,
but valid for all $c^2\geq u^2(\theta)$ for $\theta\in(\tfrac14,\hat\theta]$, and for all $c^2\geq w_-^2(\theta)$ for 
$\theta\in(\hat\theta,\tfrac12)$. This completes the proof for $\theta\in(\tfrac14,\tfrac12)$.

\textbf{Case \boldmath$\theta = \tfrac{1}{4}$:}
Since $u(\tfrac14)=0$, for all $c>0$ we obtain $\alpha=0$ in this case. 
Then, from
\eqref{eq:h11} and \eqref{eq:p0half} we have
$$
h(t)= \begin{dcases*}
t+\frac{c}{2\pi}\sin(2\pi t), & \textrm{for}\; $ t\in[\alpha,\alpha+\theta]=[0,\frac14],$\\
\frac12-t+\frac{c}{2\pi}\sin(2\pi t), & \textrm{for}\; $ t\in[\alpha+\theta,\alpha+\tfrac{1}{2}]=[\frac14,\frac12]$,
\end{dcases*} 
$$
from which it follows immediately that $h(t)>0$ for all $t\in(0,\tfrac12)=(\alpha,\alpha+\tfrac12)$.

\textbf{Case \boldmath$\theta\in(\frac{1}{2},\tfrac34]$:} From \eqref{eq:necc},
we require $c^2\geq v^2(\theta)$ for the phase $\alpha$ to be defined.
We observe then that $v(\theta)\geq0$ for $\theta\in(\tfrac12,\tfrac34]$, and it follows from the phase condition \eqref{eq:alpha1}
that $\alpha\in[0,\tfrac{1}{4})$, and
$\sin(2\pi\alpha)$ and $\cos(2\pi\alpha)$
satisfy
\eqref{eq:c1proof}.

The candidate solution $h(t)$, defined by \eqref{eq:h11} and \eqref{eq:phalf1}
has discontinuous derivative at $t = \alpha + \theta - \tfrac{1}{2}$.
At $t=\alpha$ the derivative is continuous, and from 
\eqref{eq:necc1} a necessary condition for $h'(\alpha)\geq0$ is that $c^2\geq 1+v^2(\theta)$. 

We next derive conditions equivalent to $h(\alpha+\theta-\tfrac12)\geq0$. From the definition of the solution in \eqref{eq:h11} and \eqref{eq:phalf}, 
and using \eqref{eq:c1proof} we obtain
\begin{align*}
    h(\alpha+\theta-\tfrac{1}{2})
    &= 1-2\theta+\big(\alpha+\theta-\frac12\big)-\alpha+\frac{c}{2\pi}\big(\sin{(2\pi(\alpha+\theta-\tfrac{1}{2}))}-\sin{(2\pi\alpha)}\big)\\
   & = \frac12-\theta
    -\frac{1}{2\pi}\big(c\sin{(2\pi(\alpha+\theta))}+v(\theta)\big)\\
    & = \frac12-\theta-\frac12\Big(\frac32-\theta\Big)
   -\frac{1}{2\pi}\big(c\cos{(2\pi\alpha)}\sin{(2\pi\theta)}+c\sin{(2\pi\alpha)}\cos{(2\pi\theta)}\big)\\
    &=-\frac14
    -\frac{1}{2\pi}\big(v(\theta)\cos(2\pi\theta)+
    \sqrt{c^2- v^2}(\theta)\sin(2\pi\theta)\big).  
\end{align*}
Thus, $h(\alpha+\theta-\tfrac{1}{2})\geq0$ if and only if 
\be \label{eq:halpthetahalfpos}
-\sqrt{c^2-v^2(\theta)}\,\sin(2\pi\theta)\geq \frac\pi2+v(\theta)\cos(2\pi\theta).
\ee
The right-hand side is positive since
$\tfrac\pi2+v(\theta)\cos(2\pi\theta)\geq
\tfrac\pi2-v(\theta)=\pi\big(2\theta-\tfrac12\big)>0$ for $\theta\in(\tfrac12,\tfrac34]$.
Then, squaring both sides of the inequality allows us to isolate $c^2$, yielding \eqref{eq:cQsin}.

The region where \eqref{eq:cQsin} holds is labelled $\cR_2$ in Figure~\ref{fig:RegEx}(a). 
In this region equations
\eqref{eq:cQsin}
and \eqref{eq:necc1} must both hold
for the candidate solution $h(t)$ defined by
\eqref{eq:h11} and \eqref{eq:phalf1}
to be valid.
We will show below that 
\eqref{eq:cQsin} implies \eqref{eq:necc1},
so a 1:1-periodic orbit for $\theta\in(\tfrac12,\tfrac34]$ can only exist
if \eqref{eq:cQsin} holds.
Region III in Figure~\ref{fig:RegEx}(a)
denotes the parameter set on which
\eqref{eq:necc1} holds (so $h'(\alpha)\geq0$) but \eqref{eq:cQsin} 
does not hold (so $h(\alpha+\theta-\tfrac12)<0$) and $h(t)$ does not define a valid solution.

We will show that \eqref{eq:cQsin} is a stronger condition than \eqref{eq:necc1}
by showing that
\be
\label{eq:vQsinproof}
1 + v^2(\theta) < v^2(\theta) + \left( \frac{\tfrac{\pi}{2} + v(\theta)\cos(2\pi\theta)}{\sin(2\pi\theta)} \right)^2
\ee
for all $\theta\in(\tfrac12,\tfrac34]$. 
Equations \eqref{eq:vQsinproof} and
\eqref{eq:cQsin} 
then imply that 
$c^2>1+v^2(\theta)$ for $\theta\in(\tfrac12,\tfrac34]$.

To show that \eqref{eq:vQsinproof} holds, define  
\begin{gather*}
\mu(\theta) := \frac{\tfrac{\pi}{2} + v(\theta)\cos(2\pi\theta)}{\sin(2\pi\theta)},\\
\nu(\theta) := \left(\mu(\theta) + 1\right)\sin(2\pi\theta) = \frac{\pi}{2} + v(\theta)\cos(2\pi\theta) + \sin(2\pi\theta).
\end{gather*}
Now $\nu(\tfrac12)=0$, while
$\sin(2\pi\theta) < 0 < v(\theta)$ for all $t\in(\tfrac{1}{2}, \tfrac{3}{4})$ implies
\[
\nu'(\theta) = -2\pi v(\theta)\sin(2\pi\theta)>0
\]  
for all $t\in(\tfrac{1}{2}, \tfrac{3}{4})$ and, hence,
$\nu(\theta)>0$ for all  $t\in(\tfrac{1}{2}, \tfrac{3}{4}]$.
Then
$$\mu(\theta) + 1 = \frac{\nu(\theta)}{\sin(2\pi\theta)}<0$$
for all  $t\in(\tfrac{1}{2}, \tfrac{3}{4}]$. Thus
$\mu^2(\theta) - 1 = \left(\mu(\theta) + 1\right)\left(\mu(\theta) - 1\right)>0$,
since both factors in the product are negative. But $\mu^2(\theta)>1$ is equivalent to \eqref{eq:vQsinproof}, as required.

To conclude the proof, we now establish that $h(t) > 0$ for almost all $t \in (\alpha, \alpha + \tfrac{1}{2})$, when \eqref{eq:cQsin} is satisfied.
By construction we have $h(\alpha)=h(\alpha+\tfrac12)=0$, and  \eqref{eq:cQsin} implies that $h(\alpha+\theta-\tfrac12)\geq0$; that is, $h$ is non-negative at the point where $h'$ is discontinuous.

Since \eqref{eq:cQsin} implies that $c^2>1+v^2(\theta)$, we then have from
\eqref{eq:hdashalph} and \eqref{eq:c1proof} that $h'(\alpha)>0$.
But it also follows from \eqref{eq:h11}, \eqref{eq:phalf1} 
and \eqref{eq:hdashalph}
that
$$
h'(\alpha+\tfrac12) =1+c\cos(2\pi(\alpha+\tfrac12)) 
 =  1 - c\cos(2\pi\alpha) = -h'(\alpha) <0$$
for all $\theta\in(\tfrac12,\tfrac34]$.

Next, we show that $h(t)>0$ for all
$t \in (\alpha, \alpha + \theta - \tfrac12)$. 
We have $h(\alpha)=0\leq h(\alpha + \theta - \tfrac12)$, and $h'(\alpha)>0$. The function $h$ is smooth on the interval 
$t\in(\alpha,\alpha + \theta - \tfrac12)$, with
$h'(t)=-1+c\cos(2\pi t)$. Since
$(\alpha,\alpha + \theta - \tfrac12)\subset(0,\tfrac12)$,
it follows that there exists a unique point $t_1\in(0,\tfrac12)$
such that $\cos(2\pi t_1)=1/c$ and $h'(t_1)=0$. Then 
$h'(t)>0$ for $t<t_1$ and $h'(t)<0$ for $t>t_1$.
Since $h'(\alpha)>0$ it follows that $\alpha<t_1$. There are then two possibilities. Either $t_1\geq  \alpha + \theta - \tfrac12$ and $h(t)$ is
monotonically increasing for $t\in(\alpha,\alpha + \theta - \tfrac12)$,
or $t_1\in(\alpha,\alpha + \theta - \tfrac12)$ and $h(t)$ is unimodal
for $t\in(\alpha,\alpha + \theta - \tfrac12)$ with a local maximum at $t=t_1$. In both cases, $h(\alpha)=0\leq h(\alpha + \theta - \tfrac12)$
implies that $h(t)>0$ for all $t\in(\alpha,\alpha + \theta - \tfrac12)$.

It remains to consider $t \in (\alpha + \theta - \tfrac{1}{2},\alpha+\tfrac12)$, for which we have shown that  $h(\alpha + \theta - \tfrac{1}{2})\geq0=h(\alpha+\tfrac12)$ with $h'(\alpha+\tfrac12)<0$.
We have $h'(t)=1+c\cos(2\pi t)$, and
an argument analogous to the case $\theta \in (\tfrac{1}{4}, \tfrac{1}{2})$ shows that $h'(t)$ has at most two zeros on this interval.
These zeros will occur at $t_1<\tfrac12<t_2$ which satisfy
$\cos(2\pi t_1)=\cos(2\pi t_2)=-\tfrac1c$, with $h'(t)<0$ for $t\in(t_1,t_2)$ and $h'(t)>0$ for $t<t_1$ and for $t>t_2$. But since $h'(\alpha+\tfrac12)<0$, it follows that $t_1<\alpha+\tfrac12<t_2$. 
Thus, $h'(t)$ has at most one zero on the interval
$(\alpha + \theta - \tfrac{1}{2},\alpha+\tfrac12)$,
and it follows that $h(t)$ is either monotonically decreasing on this interval, or unimodal with a local maximum.
In either case $h(t)>0$ for all
$t \in (\alpha + \theta - \tfrac{1}{2},\alpha+\tfrac12)$ as required,
which completes the proof for this 
last case.~\hfill$\qed$

\subsection{Proof of Theorem~\ref{thm:po2}}
\label{sec:pf:po2}

This proof follows the same structure as that of Theorem~\ref{thm:po}, and we highlight the modifications required 
for the
phase $\alpha\in[\tfrac 14,\tfrac 3 4]$. 
Since we only need to consider $\theta\in(0,\tfrac12)$ reduces the number of subcases that we need to consider.

We consider the profile \eqref{eq:h11},\eqref{eq:p0half}, which is our candidate 1:1-periodic orbit. We already established the necessary condition \eqref{eq:necc2} for such a solution to exist. This condition ensures that there exists a unique phase \( \alpha\in[\tfrac14,\tfrac34] \) satisfying the phase condition \eqref{eq:alpha0} with \( h(\alpha) = h(\alpha + \tfrac{1}{2}) = 0 \).
In the 
proof of Theorem~\ref{thm:po}, we also showed that this profile \( h(t) \) satisfies the symmetry condition \eqref{eq:sym}, and so, as in the proof of 
Theorem~\ref{thm:po}, it suffices to verify the positivity condition \eqref{eq:hleftsign} to complete the proof.
By Corollary~\ref{cor:sym2}, we restrict our analysis to the interval \( \theta \in [\tfrac14,\tfrac12) \).

We begin by showing that $h'(t)<0$ for all $t \in (\alpha + \theta, \alpha + \tfrac12)$. Since $h(\alpha+\tfrac12)=0$, this establishes that
$h(t)>0$ for all $t \in [\alpha + \theta, \alpha + \tfrac12)$.

For $\theta \in [\tfrac14, \tfrac{1}{2})$ the right-hand side of 
\eqref{eq:alpha0} is non-negative and, hence, 
$\alpha \in [\tfrac14,\tfrac12]$. This implies that 
$c\cos(2\pi\alpha)\leq0$. From \eqref{eq:hdashalph} we have 
$h'(\alpha)=1+c\cos(2\pi\alpha)$, and so
a necessary condition for $h'(\alpha)\geq0$ is that 
$-c\cos(2\pi\alpha)\leq1$ (which is equivalent to the right-hand inequality in \eqref{eq:necc2}).

Because $\cos(2\pi t)$ is strictly increasing on $[\tfrac12,1]$, 
for all $t \in (\alpha + \theta, \alpha + \tfrac12)\subset(\tfrac12,1)$ 
we have 
$c\cos(2\pi t) < c\cos\left(2\pi (\alpha + \tfrac{1}{2})\right) = -c\cos(2\pi \alpha)\leq 1$. Then from 
\eqref{eq:h11}, \eqref{eq:p0half} for $t \in (\alpha + \theta, \alpha + \tfrac12)$ we conclude
$$
h'(t)=-1+c\cos(2\pi t)<0,
$$
as required.

It remains to show that $h(t)>0$ for almost every $t\in(\alpha,\alpha+\theta)$. This is trivially true if $u^2(\theta)\leq c <1$, since then $h'(t)=1+c\cos(2\pi t)>0$ for all $t\in(\alpha,\alpha+\theta)$. 

If $c=1$ and $u^2(\theta)\leq c$ 
then $h'(t)\geq0$ for all $t\in(\alpha,\alpha+\theta)$.
For $\theta\in(\tfrac14,\tfrac12)$ it follows that $c^2=1<1+u^2(\theta)$
and from \eqref{eq:hdashalph} and \eqref{eq:c0proof3} that 
$$h'(\alpha)=1+c\cos(2\pi\alpha)=1-\sqrt{c^2-u^2(\theta)}>0.$$
On the other hand, if $c=1$ and $\theta=\tfrac14$ then $\alpha=\tfrac12$ and
$h'(\alpha)=0$. Differentiating twice reveals that $h''(\alpha)=0$ and 
$h'''(\alpha)>0$. In both cases $\theta=\tfrac14$ and $\theta\in(\tfrac14,\tfrac12)$ it follows that
$h(t)>0$ for all $t\in(\alpha,\alpha+\theta)$
when $u^2(\theta)\leq c\leq1$. 

It remains only to consider $h(t)$ for $t\in(\alpha,\alpha+\theta)$ when $c>1$.
In this case, there exist $t_1\in(\frac{1}{4},\frac{1}{2})$
and $t_2\in(\tfrac{1}{2},\tfrac{3}{4})$ such that $\cos(2\pi t_1)=\cos(2\pi t_2)=-1/c$, with $h'(t)>0$ for
$t\in(0,t_1)$ and $t\in(t_2,1)$, and
$h'(t)<0$ for $t\in(t_1,t_2)$. The condition \eqref{eq:necc1.5}
ensures that $h'(\alpha)\geq0$, which implies that $\alpha\leq t_1$.
We will show below that $t_2<\alpha+\theta$. Then, the function
$h(t)$ has a unique local minimum at $t=t_2$ in the interval
$(\alpha,\alpha+\theta)$, and showing that $h(t)>0$ for almost every $t\in(\alpha,\alpha+\theta)$ is equivalent to showing that $h(t_2)\geq0$.

Next, we show that $t_2<\alpha+\theta$. The critical point $t_2$ satisfies \( \cos(2\pi t_2) = -1/c \) and $\sin{(2\pi t_2)}=-\sqrt{1-1/c^2}$, while
$\alpha\in[\tfrac14,\tfrac12]$ satisfies \eqref{eq:c0proof3}.
Since we consider $\theta\in[\tfrac14,\tfrac12)$, we have $\cos(2\pi\theta) \leq 0 < \sin(2\pi\theta)$. 
Note also that, since $u(\theta)$ and $\sin(2\pi\theta)$ are both monotonic for $\theta\in[\tfrac14,\tfrac12]$, 
it follows that 
$$u(\theta)\sin(2\pi\theta)\leq
\begin{cases}
    u(\tfrac38)\sin(\frac\pi2)=\frac{\pi}{4}, & \theta\in[\tfrac14,\tfrac38],\\
       u(\tfrac{5}{12})\sin(\frac{3\pi}{4})=\frac{\pi}{3\sqrt{2}}, & \theta\in[\tfrac38,\tfrac{5}{12}],\\
    u(\tfrac12)\sin(\frac{5\pi}{6})=\frac{\pi}{4}, & \theta\in[\tfrac{5}{12},\tfrac12],
\end{cases}$$
and hence $u(\theta)\sin(2\pi\theta)<1$ for all
$\theta\in[\tfrac14,\tfrac12]$. Thus for 
$\theta\in[\tfrac14,\tfrac12)$ we get
\begin{align*}
    -1 & < -u(\theta)\sin(2\pi\theta)
    \leq  -\sqrt{c^2-u^2(\theta)}\cos(2\pi\theta)
    -u(\theta)\sin(2\pi\theta)\\
     & = c\cos(2\pi\alpha)\cos(2\pi\theta)-c\sin(2\pi\alpha)\sin(2\pi\theta)
     = c\cos(2\pi(\alpha+\theta)).
\end{align*}
Hence,
$$\cos(2\pi(\alpha+\theta))>-\frac1c=\cos(2\pi t_2).$$
Since $t_2\in(\tfrac12,\tfrac34)$,
$\alpha+\theta\in[\tfrac12,1)$ and 
$\cos{(2\pi t)}$ is strictly increasing for $t\in[\tfrac12,1]$, it follows that $t_2<\alpha+\theta$.

It remains to show that $h(t_2)\geq0$. Since $t_2<\alpha+\theta$, the value of 
$h(t_2)$ defined by \eqref{eq:h11}, \eqref{eq:p0half}  is given by \eqref{eq:h(t_2)}
from which  \eqref{eq:h(t_2)<0 iff sin} follows and defines the condition for $h(t_2)<0$. Now $\tfrac12<t_2<\alpha+\theta<1$ implies that 
$2\pi(t_2-\theta-\alpha)\in(-\pi,0)$. From \eqref{eq:necc1.5}
we obtain $\sqrt{c^2-1}\leq u(\theta)<\frac\pi2$ for $\theta\in[\tfrac14,\tfrac12)$ and, hence, $\sqrt{c^2-1}-\frac\pi2\in[-\frac\pi2,0)$.
Then, because $\cos{(2\pi t)}$ is  increasing for $t\in(-\tfrac12,0)$, applying the cosine function to both sides of \eqref{eq:h(t_2)<0 iff sin} preserves the inequality; so $h(t_2)<0$, if and only if,
\[c^2\cos{\left(\sqrt{c^2-1}-\tfrac\pi2\right)}>c^2\cos{(2\pi\left(t_2-\theta-\alpha\right))}.\]
This expression can be simplified, similarly to how \eqref{eq:boundaryI} is obtained in the proof of Theorem~\ref{thm:po}, but this time using
equation \eqref{eq:c0proof3} in place of \eqref{eq:c0proof}, to arrive at the condition 
\begin{align}
\label{eq:boundaryIb}
    c^2\sin{(\sqrt{c^2-1})}&>{u}(\theta) \left( \sin(2\pi \theta) - \cos(2\pi \theta) \sqrt{c^2 - 1} \right) \\&\hspace*{5em}+\sqrt{c^2 - {u}^2(\theta)} \left( \sin(2\pi \theta) \sqrt{c^2 - 1} + \cos(2\pi \theta) \right).\nonumber
\end{align}
We define 
\begin{align}
\label{eq:W+}
W_+(\theta,c)&={u}(\theta) \left( \sin(2\pi \theta) - \cos(2\pi \theta) \sqrt{c^2 - 1} \right) \\
&\hspace*{3em}+\sqrt{c^2 - {u}^2(\theta)} \left( \sin(2\pi \theta) \sqrt{c^2 - 1} + \cos(2\pi \theta) \right)-c^2\sin{(\sqrt{c^2-1})}.\notag
\end{align}

Then, from \eqref{eq:boundaryIb}, it follows that $W_+(\theta,c)$ has the same sign as $h(t_2)$. We define the curve $c^2=w_+^2(\theta)$, on which $h(t_2)=0$, implicitly by $W_+(\theta,w_+^2(\theta)) = 0$. 
The function $w_+^2(\theta)$ is then given by \eqref{eq:boundaryIc(theta)2} for $\theta \in [\tfrac14,\hat{\theta}]$. The threshold value $\hat\theta$ is defined in \eqref{eq:thetahat} as the value of $\theta$ such that $w_+^2(\hat\theta)=u^2(\hat\theta)=w_-^2(\hat{\theta})$. The function $w_+^2(\theta)$ satisfies the boundary conditions in \eqref{eq:w+2boundary}. 

To determine which side of the boundary curve $w_+^2(\theta)$ corresponds to $h(t_2) < 0$, it is sufficient to evaluate $W_+$ at two points on either side of this curve.
Doing this at two points on the line $\theta=\tfrac 13$, with one on the lower boundary $c^2=u^2(\theta)$ and the other one on the upper boundary $c^2=u^2(\theta)+1$, we find 
$W_+\big(\tfrac13,(u^2(\frac13)+1)^\frac12\big) < 0 < W_+(\tfrac 13,u(\tfrac13))$.
Thus $h(t_2)\geq0$ for $c^2 \leq w_+^2(\theta)$.

Thus for $\theta\in(\tfrac14,\hat\theta]$ we require $c^2\leq w_+^2(\theta)$ and
from \eqref{eq:necc1.5} that $c^2\leq 1+u^2(\theta)$. But $h(t_2)\geq0$ requires $h(t_1)>0$ and, hence, $h'(\alpha)>0$, while $c^2=1+u^2(\theta)$ implies
$h'(\alpha)\leq0$. Thus, $w_+^2(\theta)< 1+u^2(\theta)$ for $\theta\in(\tfrac14,\hat\theta]$, and so the curve $c^2= w_+^2(\theta)$ 
defines the upper boundary curve of the region $\mathcal{S}$, 
which is illustrated in Figure~\ref{fig:RegEx2}(a).
Thus the candidate periodic solution defined by \eqref{eq:h11} and \eqref{eq:p0half} is 
invalid for any $c$ when  $\theta\in(\hat\theta,\tfrac12)$, but valid for all $u^2(\theta)\leq c^2\leq w^2_+(\theta)$ for $\theta\in[\tfrac14,\hat\theta]$.~\hfill$\qed$


\subsection{Proof of Theorem~\ref{thm:w}}
\label{sec:pf:w}

In Theorem~\ref{thm:po} it was shown that symmetric 1:1-periodic orbits with phase $\alpha\in[-\tfrac14,\tfrac14]$ exist in the region $(\theta,c)\in\cR$.
For $\tau$ not an integer or half integer, so $\tau\ne j/2$ for $j\in\N$, the characteristic polynomial $\Delta(\lambda)$ defining the Floquet multipliers of these periodic orbits was derived in Theorem~\ref{thm:Poly}; this was used in Theorem~\ref{thm:stabc} to show that these periodic orbits are stable for all sufficiently large $c$. 

The standard approach  to find where these orbits become unstable in a smooth system is to look for the first bifurcation where a Floquet multiplier crosses the unit circle. We will adopt the same approach here, but need to take some care, as our system is only piecewise smooth and the Floquet multipliers do not always vary continuously. 
In particular, it follows from Theorem~\ref{thm:Poly} that when  $\tau\ne j/2$ for $j\in\N$ the number of non-trivial Floquet multipliers is $\lceil2\tau\rceil$ with their locations governed  by $\Delta_{\lceil2\tau\rceil}$. 
Consequently, if $\tau$ crosses $j/2$ for $j\in\N$ the Floquet multipliers must vary discontinuously as the characteristic polynomial switches between \eqref{eq:polynomiallambda2kp1} and \eqref{eq:polynomiallambda2kp2}, and not even the number of multipliers is conserved.

To ensure that the Floquet multipliers vary continuously, we will consider delays $\tau=k+\theta$ for fixed $k\in\N_0$ and varying $\theta$ with either $\theta\in(0,\tfrac12)$ or $\theta\in(\tfrac12,1)$, so that we do not cross any of the lines $\tau=j/2$ for $j\in\N$. Furthermore, we will only consider variation of the two parameters $c$ and $\theta$ with $\theta\in(0,\tfrac12)$ and $(\theta,c)\in\cR_1$ or  $\theta\in(\tfrac12,1)$ and $(\theta,c)\in\cR_2$.

We first consider the case where 
$\theta\in(0,\tfrac12)$ and $(\theta,c)\in\cR_1$, and so the Floquet multipliers are defined by $\Delta(\lambda)=(\lambda-1)\Delta_{2k+1}(\lambda)$ where $\Delta_{2k+1}(\lambda)$ is defined by \eqref{eq:polynomiallambda2kp1}.
For $(\theta,c)\in\cR_1$,
the constraint \eqref{eq:necc} ensures that the phase $\alpha\in[-\tfrac14,\tfrac14]$, defined by \eqref{eq:alpha0}, is a smooth function of $\theta$ and $c$.
It then follows from \eqref{eq:hdashalph} that $h'(\alpha)$ is a smooth function of $\theta$ and $c$, with $h'(\alpha)\geq1$ since $\alpha\in[-\tfrac14,\tfrac14]$ implies $c\cos(2\pi\alpha)\geq0$.
Thus for $\theta\in(0,\tfrac12)$ and $(\theta,c)\in\cR_1$, the coefficients of
$\Delta_{2k+1}(\lambda)$ and, hence, the Floquet multipliers vary smoothly with $c$ and $\theta$. Thus to find the stability boundary $\textup{\textbf{B}}$
of symmetric 1:1-periodic orbits with phase $\alpha\in[-\tfrac14,\tfrac14]$
for any fixed value of $\theta\in(0,\tfrac12)$ it is sufficient to find the largest value of $c$ for which a bifurcation occurs. We will first identify the bifurcation points, and then consider which one causes the loss of stability of the 1:1-periodic orbit.

Bifurcations occur when $|\lambda|=1$, that is when a Floquet multiplier satisfies $\lambda=\pm1$, 
or a complex conjugate pair of Floquet multipliers satisfy
$\lambda=e^{\pm i\omega}$ with $\omega\in(0,\pi)$ and we will consider each of these cases separately for $\theta\in(0,\tfrac12)$.

Consider first the case of $\lambda=1$.
  see that $\Delta_{2k+1}(1)=0$, 
 if and only if, $h'(\alpha)=1$, in which case $\lambda=1$ is a simple root of
$\Delta_{2k+1}$.
But in this case \eqref{eq:hdashalph} for $\theta\in(0, \frac{1}{2})$ implies that
$\sin(2\pi\alpha)=\pm1$. It then follows from \eqref{eq:alpha0} that $c^2=u^2(\theta)$. For $\theta\in[\tfrac12-\hat\theta,\hat\theta]$, from Theorem~\ref{thm:po}, this corresponds to the lower boundary of $\cR_1$ where these 1:1-periodic orbits cease to exist. From Theorem~\ref{thm:po2} this is also the boundary of the existence region $\cS$ of unstable 1:1-periodic orbits with phase $\alpha\in[\tfrac14,\tfrac34]$. As already illustrated in Figure~\ref{fig:RegEx2} and discussed after Theorem~\ref{thm:po2},
these two periodic orbits undergo a fold bifurcation of periodic orbits along this curve. 

Although this bifurcation occurs for $\tau=k+\theta$ with $\theta\in[\tfrac12-\hat\theta,\hat\theta]$ for all $k\in\N_0$, we will see below that it only defines the stability-boundary for 1:1-periodic orbits when $k=0$. Note that there are no fold bifurcations of 1:1-periodic orbits for $\theta\in(0,\tfrac12-\hat\theta)$ or for $\theta\in(\hat\theta,\tfrac12)$, because for these values of $\theta$ the 1:1-periodic orbits do not exist for $u^2(\theta)\leq c^2<w_-(\theta)$.

When $k$ is even, $h'(\alpha)\geq1$ and
$$\Delta_{2k+1}(-1)=-2(-1)^{2k}+\frac{4}{h'(\alpha)}(-1)^k-\frac{4}{h'(\alpha)^2}=-2\bigg[\Big(1-\frac{1}{h'(\alpha)}\Big)^2+\frac{1}{h'(\alpha)^2}\bigg],$$
implies $\Delta_{2k+1}(-1)<0$, while when $k$ is odd,
$\Delta_{2k+1}(-1)\leq-2$ follows trivially.
So, it is not possible for $\lambda=-1$ to be a Floquet multiplier.

Thus, we consider the possibility of complex conjugate roots with modulus one, so that $\lambda=e^{\pm i \omega}$ with $\omega\in(0,\pi)$.
Substituting
$\lambda=e^{i\omega}$ into  
$\Delta_{2k+1}(\lambda)=0$
and collecting
real and imaginary parts results in
\begin{gather} \label{eq:realpartzero}
\cos((2k+1)\omega)-\cos(2 k\omega)
+\frac{4}{h'(\alpha)}\cos(k\omega)-\frac{4}{[h'(\alpha)]^2}=0,\\ \label{eq:imagpartzero}
\sin((2k+1)\omega)-\sin(2 k\omega)+\frac{4}{h'(\alpha)}\sin( k\omega)=0.
\end{gather}

Isolating $h'(\alpha)$ in \eqref{eq:imagpartzero} gives
\be \label{eq:hdashzero}
h'(\alpha)=\frac{4\sin(k\omega)}{\sin(2k\omega)-\sin((2k+1)\omega)}.
\ee
Then eliminating $h'(\alpha)$ from \eqref{eq:realpartzero} results in
\begin{align} \label{eq:elimzero}
0&=
\sin^2(k\omega)
\Bigl(\cos((2k+1)\omega)-\cos(2k\omega)\Bigr)\\ \notag
&\hspace*{1em} 
+\sin(k\omega)\cos(k\omega)\Bigl(\sin(2k\omega)-\sin((2k+1)\omega)\Bigr)
-\frac14\Bigl(\sin(2k\omega)-\sin((2k+1)\omega)\Bigr)^2.
\end{align}

Using standard trigonometric identities, we can further
simplify equation \eqref{eq:elimzero} to
\begin{align*}
0&=
\sin^2(k\omega)
\Bigl(\cos((2k+1)\omega)-\cos(2k\omega)\Bigr)\\ \notag
&\hspace*{2em} 
+\frac12\sin(2k\omega)\Bigl(\sin(2k\omega)-\sin((2k+1)\omega)\Bigr)
-\frac14\Bigl(\sin(2k\omega)-\sin((2k+1)\omega)\Bigr)^2\\
&=\sin^2(k\omega)
\Bigl(\cos((2k+1)\omega)-\cos(2k\omega)\Bigr)
+\frac14\sin^2(2k\omega)-\frac14\sin^2((2k+1)\omega)\\
&=\sin^2(k\omega)
\Bigl(\cos((2k+1)\omega)-\cos(2k\omega)\Bigr)
-\frac18\cos(4k\omega)+\frac18\cos((4k+2)\omega)\\
&=-2\sin^2(k\omega)\sin((4k+1)\omega/2)\sin(\omega/2)
-\frac14\sin((4k+1)\omega)\sin(\omega)\\
&=-\sin(\omega/2)\sin((4k+1)\omega/2)\Bigl(
2\sin^2(k\omega)+\cos((4k+1)\omega/2)\cos(\omega/2)\Bigr)\\
&=-\sin(\omega/2)\sin((4k+1)\omega/2)\Bigl(
\frac32-\cos^2(k\omega)+\frac12\cos((2k+1)\omega)
\Bigr).
\end{align*}

The term in brackets and the $\sin(\omega/2)$ term 
are both strictly positive for $\omega\in(0,\pi)$; thus,
all the complex Floquet multipliers with modulus one are
given by the roots of
$$\sin((4k+1)\omega/2)=0, \qquad \omega\in(0,\pi).$$
Hence, a necessary condition for a pair of complex conjugate Floquet multipliers 
of modulus one
is that
\be \label{eq:omegazero}
\omega=\frac{2j\pi}{4k+1}, \qquad j=\{1,2,\ldots,2k\}.
\ee

Equation~\eqref{eq:omegazero} shows that there are no Floquet multipliers $\lambda=e^{\pm i\omega}$ when $k=0$; that is when $\tau=\theta\in(0,\tfrac12)$. 

For $k>0$ and $j$ satisfying \eqref{eq:omegazero}, we
substitute \eqref{eq:omegazero} into \eqref{eq:hdashzero} to obtain
\begin{align} \notag
h'(\alpha) &=
\frac{-4\sin\big(\frac{2jk\pi}{4k+1}\big)}{\sin\big(\frac{2j(2k+1)\pi}{4k+1}\big)-\sin\big(\frac{4kj\pi}{4k+1}\big)} 
= \frac{-2\sin\big(\frac{2jk\pi}{4k+1}\big)}{
\sin\big(\frac{j\pi}{4k+1}\big)\cos\big(j\pi\big)} \\ \notag
&
= 2(-1)^{j+1}\sin\Big(\frac{2jk\pi}{4k+1}\Big)
\csc\Big(\frac{j\pi}{4k+1}\Big) 
\\ \notag &
= 2(-1)^{j+1}\csc\Big(\frac{j\pi}{4k+1}\Big) 
\sin\Big(\frac{j\pi}{2}-\frac{j\pi}{2(4k+1)}\Big).
\end{align}
Thus, when $j$ is odd 
\be \label{eq:hdashjodd}
h'(\alpha) 
= 2(-1)^{(j-1)/2}\csc\Big(\frac{j\pi}{4k+1}\Big) 
\cos\Big(\frac{j\pi}{2(4k+1)}\Big) 
= (-1)^{(j-1)/2}\csc\Big(\frac{j\pi}{2(4k+1)}\Big),
\ee
while for $j$ even
\be \label{eq:hdashjeven}
h'(\alpha) 
= 2(-1)^{j/2}\csc\Big(\frac{j\pi}{4k+1}\Big) 
\sin\Big(\frac{j\pi}{2(4k+1)}\Big)
= (-1)^{j/2}\sec\Big(\frac{j\pi}{2(4k+1)}\Big).
\ee

Notice that for $j\in\{2,3,\ldots,2k\}$,
\be \label{eq:cscsec}
\csc\Big(\frac{\pi}{2(4k+1)}\Big)>
\csc\Big(\frac{j\pi}{2(4k+1)}\Big)>\sqrt{2}>
\sec\Big(\frac{j\pi}{2(4k+1)}\Big)>1>0.
\ee
Thus, since we require $h'(\alpha)\geq0$ to satisfy the phase condition \eqref{eq:halpha}, 
only values of $j$ that result in even powers of $-1$ in
\eqref{eq:hdashjodd} and \eqref{eq:hdashjeven} are valid. 
Hence $j=1,4,5,8,9,\ldots \leq 2k$. From
\eqref{eq:hdashjodd}, $j=1$ implies that $h'(\alpha)>0$ for all $k\in\N$, and then it follows from
\eqref{eq:cscsec} that out of all the eligible values, $j=1$ results in the largest value of $h'(\alpha)$.  
Furthermore, from \eqref{eq:hdashalpcu} we know that
the largest value of $h'(\alpha)$ at a torus bifurcation
corresponds to the largest value of $c$ at such a bifurcation, 
and equation \eqref{eq:T} for $\theta\in(0,\tfrac12)$ follows from \eqref{eq:hdashalpcu} and \eqref{eq:hdashjodd} with $j=1$.
From \eqref{eq:T} it follows that 
for all $k\geq1$,
\[
c^2 = {u}^2(\theta) + \left( \csc{\left( \frac{\pi}{2(4k+1)} \right)} - 1 \right)^2 \geq 
0^2+ (3 - 1)^2 = 4 > 1+\frac{\pi^2}{4\phantom{^2}}.
\]
Thus, for each $(\tau,c)\in\textbf{T}$, where $\tau=k+\theta$, we have $(\theta,c)\in\mathcal{R}_1$; that is, the torus bifurcation curve
is contained in the interior of the region where the 1:1-periodic orbits exist, and so defines their stability boundary.
Segments of the curves for $j=4$ and $j=5$ are shown in Figure~\ref{fig:Tunst}, while the segments with $j=1$ are all contained in \textbf{B}.

Equation~\eqref{eq:w} for $\theta\in(0,\tfrac12)$ also follows immediately from \eqref{eq:omegazero} with $j=1$,
since the rotation number is $\rho=\omega/2\pi$.

Now consider the case of $\tau=k+\theta$ for fixed $k\in\N_0$ 
with $\theta\in(\tfrac12,1)$
and $(\theta,c)\in\cR_2$. The 
Floquet multipliers are then defined by $\Delta(\lambda)=(\lambda-1)\Delta_{2k+2}(\lambda)$, where $\Delta_{2k+2}(\lambda)$ is defined by \eqref{eq:polynomiallambda2kp2}. The remaining steps in the proof are similar to the previous case, but 
$\Delta_{2k+2}(\lambda)$ is a polynomial of degree $2k+2$, whereas $\Delta_{2k+1}(\lambda)$ was an odd degree polynomial, so the algebra will play out differently, as we will see below.

The constraint \eqref{eq:necc} ensures that the phase $\alpha\in[-\tfrac14,\tfrac14]$, defined by \eqref{eq:alpha1}, is a smooth function of $\theta$ and $c$, and from  \eqref{eq:hdashalph}
that $h'(\alpha)$ and the coefficients of
$\Delta_{2k+2}(\lambda)$ 
are smooth function of $\theta$ and $c$;
 hence the Floquet multipliers vary smoothly with $c$ and $\theta$.

It is easy show that $\Delta_{2k+2}(\lambda)$ 
does not have any roots $\lambda=\pm1$ with $h'(\alpha)>0$. 
So we again look for roots $\lambda=e^{\pm i\omega}$ with 
$\omega\in(0,\pi)$.
Substituting
$\lambda=e^{i\omega}$ into  $\Delta_{2k+2}(\lambda)=0$
and separating real and imaginary parts
results in
\begin{gather} \label{eq:realpartone}
\cos((2k+2)\omega)-\cos((2k+1)\omega)
+\frac{4}{h'(\alpha)}\cos((k+1)\omega)+\frac{4}{[h'(\alpha)]^2}=0,\\ \label{eq:imagpartone}
\sin((2k+2)\omega)-\sin((2k+1)\omega)+\frac{4}{h'(\alpha)}\sin((k+1)\omega)=0.
\end{gather}
Isolating $h'(\alpha)$ in \eqref{eq:imagpartone} gives
\be \label{eq:hdashone}
h'(\alpha)=\frac{-4\sin((k+1)\omega)}{\sin((2k+2)\omega)-\sin((2k+1)\omega)}.
\ee
Then eliminating $h'(\alpha)$ from \eqref{eq:realpartone} and applying trigonometric identities yields
\begin{align} \label{eq:elimone}
&0 =4\sin((k+1)\omega)\Bigl(\sin(k\omega)-\sin((k+1)\omega)\Bigr)
+\Bigl(\sin((2k+2)\omega)-\sin((2k+1)\omega)\Bigr)^2 \\ \notag
& =-8\sin((k+1)\omega)\sin(\omega/2)\cos((2k+1)\omega/2)
+4\sin^2(\omega/2)\cos^2((4k+3)\omega/2)\\ \notag
&=-4\sin(\omega/2)\bigg[2\sin((k+1)\omega)\cos((2k+1)\omega/2)
-\sin(\omega/2)\Big(1-\sin^2((4k+3)\omega/2)\!\Big)\!\bigg]\\ \notag
&=-4\sin(\omega/2)\bigg[\sin((4k+3)\omega/2)+\sin(\omega/2)
-\sin(\omega/2)\Big(1-\sin^2((4k+3)\omega/2)\bigg]\\
&=-4\sin(\omega/2)\sin((4k+3)\omega/2)\bigg[1+\sin(\omega/2)\sin((4k+3)\omega/2)\bigg]. \notag
\end{align}

The other terms are again strictly positive for $\omega\in(0,\pi)$, thus
all the complex conjugate Floquet multipliers 
with modulus one
are
given by the roots of
$$\sin((4k+3)\omega/2)=0, \qquad \omega\in(0,\pi).$$
Hence a necessary condition for modulus one complex Floquet multipliers is that
\be \label{eq:omegaone}
\omega=\frac{2j\pi}{4k+3}, \qquad j=\{1,2,\ldots,2k+1\}.
\ee

The rest of the proof is similar to the $\theta\in(0,\tfrac12)$ case. 
In particular, substituting
\eqref{eq:omegaone} into \eqref{eq:hdashone} results in
$$
h'(\alpha) = 
\begin{cases}
(-1)^{(j-1)/2}\csc\Big(\frac{j\pi}{2(4k+3)}\Big), &
\text{for $j$ odd,}\\
(-1)^{(j-2)/2}\sec\Big(\frac{j\pi}{2(4k+3)}\Big), &
\text{for $j$ even.}
\end{cases}
$$
\sloppy{The trigonometric terms are again positive for $j$ satisfying
\eqref{eq:omegaone}, so we require $j=1,2,5,6,9,10,\ldots \leq 2k+1$ for $h'(\alpha)>0$. The case $j=1$, this time for all $k\in\N_0$,
again results in the largest value of $h'(\alpha)$ and $c$ and so defines the stability boundary.} Equations~\eqref{eq:T} and \eqref{eq:w} follow for $\theta\in(\tfrac12,1)$.
Segments of the curves for $j=2$ and $j=5$ are shown in Figure~\ref{fig:Tunst}, while the $j=1$ segments are all contained 
in \textbf{B}.~\hfill$\qed$
\vspace*{1.15ex}

\begin{figure}
    \centering
    \includegraphics[width=\textwidth]{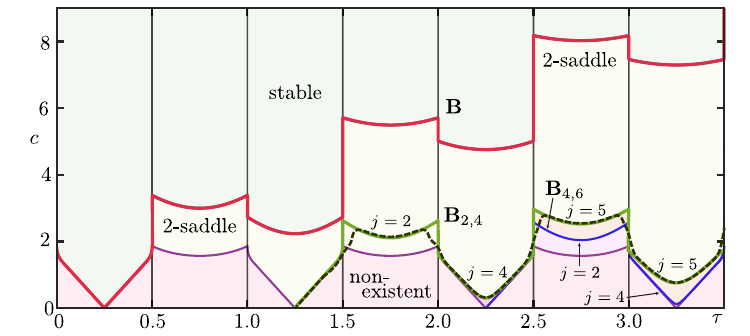}
    \caption{Stability boundary \textup{\textbf{B}} for symmetric $1{:}1$-periodic orbits with phase $\alpha\in[-\tfrac{1}{4},\tfrac{1}{4}]$, as in Figure~\ref{fig:T}(a). The green curve 
    $\textup{\textbf{B}}_{2,4}$ 
    marks the transition from $2$-saddle to $4$-saddle orbits, and the blue curve 
    $\textup{\textbf{B}}_{4,6}$ the transition from $4$-saddle to $6$-saddle orbits. The dashed black curve shows the corresponding bifurcation of period-one orbits in the smooth GZT model~\eqref{eq:GZT} for $\kappa=33$, where the orbit changes from $2$-saddle to $4$-saddle. The numbers $j$ label the unitary complex Floquet multipliers, indexed as in~\eqref{eq:omegazero} and~\eqref{eq:omegaone}.}
    \label{fig:Tunst}
\end{figure}

The stability boundary \textbf{B} for the psGZT model \eqref{eq:iGZT} is rooted at $(\tau,c)=(\tfrac14,0)$,
and follows the fold bifurcation curve 
$c=u(\tau)$ for $\tau\in[\tfrac12,\hat\theta]$. 
But the proof of Theorem~\ref{thm:w} shows that there
are fold bifurcation curves for 
$c=u(\theta)$ for $\tau=k+\theta\in[k+\tfrac12,k+\hat\theta]$ for each $k\in\N_0$, which give rise to additional bifurcation curves for the 1:1-periodic orbits. 

In the proof of Theorem~\ref{thm:w} we needed to identify which of these bifurcations resulted in the loss of stability. Here, we will briefly investigate the additional bifurcations from the unstable 1:1-periodic orbit for the psGZT model, as well as their homologues in the GZT model.

Two of these bifurcation curves, labelled 
$\textup{\textbf{B}}_{2,4}$ and 
$\textup{\textbf{B}}_{4,6}$
are shown in Figure~\ref{fig:Tunst}. For 
$\tau\in[\tfrac54,\tfrac32]$ the $\textup{\textbf{B}}_{2,4}$ is just a translation of the \textbf{B} curve
from $\tau\in[\tfrac14,\tfrac12]$. 
For $\tau\geq\tfrac32$ the curve $\textup{\textbf{B}}_{2,4}$ is composed of vertical segments 
at
when $\tau=\tfrac{k}{2}$ for $k\in\N$, and curved segments in between corresponding to the torus bifurcations with different values of $j$ as identified in the proof of Theorem~\ref{thm:w}. Each time a torus bifurcation curve is crossed as $c$ is decreased the number of unstable Floquet multipliers increases by two, and the labelling $\textup{\textbf{B}}_{2,4}$, $\textup{\textbf{B}}_{4,6}$ denotes this. When there are multiple valid values of $j$, we showed in the proof of Theorem~\ref{thm:w} that $j=1$ always results in the largest $c$ value
and hence the segments with $j=1$ form the smooth parts of \textbf{B}. The curve $\textup{\textbf{B}}_{2,4}$ corresponds to the next torus bifurcation where the 1:1-period orbit transitions form a 2-saddle to a 4-saddle orbit. The curve follows different values of $j$ on different segments; $j=2$ for $\tau\in(\tfrac32,2)$ and $j=4$ for $\tau\in(2,\tfrac52)$ because these are the only valid $j$ values on these intervals, and $j=5$ for  
$\tau\in(\tfrac52,3)$ because it has higher $c$ values
than $j=2$ on this interval.

Also shown in Figure~\ref{fig:Tunst} is the numerically computed corresponding bifurcation curve for the smooth GZT model \eqref{eq:GZT}, which shows strong agreement with the curve $\textup{\textbf{B}}_{24}$, except for the vertical segments, similar
to what is seen in Figure~\ref{fig:Tddebiftool} 
for the curve \textbf{B}.


\begin{thebibliography}{39}
\ifx \bisbn   \undefined \def \bisbn  #1{ISBN #1}\fi
\ifx \binits  \undefined \def \binits#1{#1}\fi
\ifx \bauthor  \undefined \def \bauthor#1{#1}\fi
\ifx \batitle  \undefined \def \batitle#1{#1}\fi
\ifx \bjtitle  \undefined \def \bjtitle#1{#1}\fi
\ifx \bvolume  \undefined \def \bvolume#1{\textbf{#1}}\fi
\ifx \byear  \undefined \def \byear#1{#1}\fi
\ifx \bissue  \undefined \def \bissue#1{#1}\fi
\ifx \bfpage  \undefined \def \bfpage#1{#1}\fi
\ifx \blpage  \undefined \def \blpage #1{#1}\fi
\ifx \burl  \undefined \def \burl#1{\textsf{#1}}\fi
\ifx \doiurl  \undefined \def \doiurl#1{\url{https://doi.org/#1}}\fi
\ifx \betal  \undefined \def \betal{\textit{et al.}}\fi
\ifx \binstitute  \undefined \def \binstitute#1{#1}\fi
\ifx \binstitutionaled  \undefined \def \binstitutionaled#1{#1}\fi
\ifx \bctitle  \undefined \def \bctitle#1{#1}\fi
\ifx \beditor  \undefined \def \beditor#1{#1}\fi
\ifx \bpublisher  \undefined \def \bpublisher#1{#1}\fi
\ifx \bbtitle  \undefined \def \bbtitle#1{#1}\fi
\ifx \bedition  \undefined \def \bedition#1{#1}\fi
\ifx \bseriesno  \undefined \def \bseriesno#1{#1}\fi
\ifx \blocation  \undefined \def \blocation#1{#1}\fi
\ifx \bsertitle  \undefined \def \bsertitle#1{#1}\fi
\ifx \bsnm \undefined \def \bsnm#1{#1}\fi
\ifx \bsuffix \undefined \def \bsuffix#1{#1}\fi
\ifx \bparticle \undefined \def \bparticle#1{#1}\fi
\ifx \barticle \undefined \def \barticle#1{#1}\fi
\bibcommenthead
\ifx \bconfdate \undefined \def \bconfdate #1{#1}\fi
\ifx \botherref \undefined \def \botherref #1{#1}\fi
\ifx \url \undefined \def \url#1{\textsf{#1}}\fi
\ifx \bchapter \undefined \def \bchapter#1{#1}\fi
\ifx \bbook \undefined \def \bbook#1{#1}\fi
\ifx \bcomment \undefined \def \bcomment#1{#1}\fi
\ifx \oauthor \undefined \def \oauthor#1{#1}\fi
\ifx \citeauthoryear \undefined \def \citeauthoryear#1{#1}\fi
\ifx \endbibitem  \undefined \def \endbibitem {}\fi
\ifx \bconflocation  \undefined \def \bconflocation#1{#1}\fi
\ifx \arxivurl  \undefined \def \arxivurl#1{\textsf{#1}}\fi
\csname PreBibitemsHook\endcsname

\bibitem[\protect\citeauthoryear{Ghil and Zaliapin}{2010}]{GZT2}
\begin{barticle}
\bauthor{\bsnm{Ghil}, \binits{M.}},
\bauthor{\bsnm{Zaliapin}, \binits{I.}}:
\batitle{A delay differential model of {ENSO} variability – part 2: Phase
  locking, multiple solutions and dynamics of extrema}.
\bjtitle{Nonlinear Process. Geophys.}
\bvolume{17},
\bfpage{123}--\blpage{125}
(\byear{2010})
\doiurl{10.5194/npg-17-123-2010}
\end{barticle}
\endbibitem

\bibitem[\protect\citeauthoryear{Ghil et~al.}{2008}]{GZT}
\begin{barticle}
\bauthor{\bsnm{Ghil}, \binits{M.}},
\bauthor{\bsnm{Zaliapin}, \binits{I.}},
\bauthor{\bsnm{Thompson}, \binits{S.}}:
\batitle{A delay differential model of {ENSO} variability: parametric
  instability and the distribution of extremes}.
\bjtitle{Nonlinear Process. Geophys.}
\bvolume{15},
\bfpage{417}--\blpage{433}
(\byear{2008})
\doiurl{10.5194/npg-15-417-2008}
\end{barticle}
\endbibitem

\bibitem[\protect\citeauthoryear{Keane et~al.}{2015}]{KKP15}
\begin{barticle}
\bauthor{\bsnm{Keane}, \binits{A.}},
\bauthor{\bsnm{Krauskopf}, \binits{B.}},
\bauthor{\bsnm{Postlethwaite}, \binits{C.}}:
\batitle{Delayed feedback versus seasonal forcing: Resonance phenomena in an
  {El Ni\~no Southern Oscillation} model}.
\bjtitle{SIAM J. Appl. Dyn. Syst}
\bvolume{14},
\bfpage{1229}--\blpage{1257}
(\byear{2015})
\doiurl{10.1137/140998676}
\end{barticle}
\endbibitem

\bibitem[\protect\citeauthoryear{Keane and Krauskopf}{2018}]{KK18}
\begin{barticle}
\bauthor{\bsnm{Keane}, \binits{A.}},
\bauthor{\bsnm{Krauskopf}, \binits{B.}}:
\batitle{Chenciner bubbles and torus break-up in a periodically forced delay
  differential equation}.
\bjtitle{Nonlinearity}
\bvolume{31}(\bissue{6}),
\bfpage{165}--\blpage{187}
(\byear{2018})
\doiurl{10.1088/1361-6544/aab8a2}
\end{barticle}
\endbibitem

\bibitem[\protect\citeauthoryear{Sieber et~al.}{2014}]{sieber2014dde}
\begin{botherref}
\oauthor{\bsnm{Sieber}, \binits{J.}},
\oauthor{\bsnm{Engelborghs}, \binits{K.}},
\oauthor{\bsnm{Luzyanina}, \binits{T.}},
\oauthor{\bsnm{Samaey}, \binits{G.}},
\oauthor{\bsnm{Roose}, \binits{D.}}:
{DDE-BIFTOOL} Manual-Bifurcation analysis of delay differential equations.
arXiv preprint arXiv:1406.7144
(2014).
\doiurl{10.48550/arXiv.1406.7144}
\end{botherref}
\endbibitem

\bibitem[\protect\citeauthoryear{Keane et~al.}{2016}]{KKP16}
\begin{barticle}
\bauthor{\bsnm{Keane}, \binits{A.}},
\bauthor{\bsnm{Krauskopf}, \binits{B.}},
\bauthor{\bsnm{Postlethwaite}, \binits{C.M.}}:
\batitle{Investigating irregular behavior in a model for the {El Ni\~no
  Southern Oscillation} with positive and negative delayed feedback}.
\bjtitle{SIAM J. Appl. Dyn. Syst}
\bvolume{15}(\bissue{3}),
\bfpage{1656}--\blpage{1689}
(\byear{2016})
\doiurl{10.1137/16M1063605}
\end{barticle}
\endbibitem

\bibitem[\protect\citeauthoryear{Ryan et~al.}{2020}]{RKA20}
\begin{barticle}
\bauthor{\bsnm{Ryan}, \binits{P.}},
\bauthor{\bsnm{Keane}, \binits{A.}},
\bauthor{\bsnm{Amann}, \binits{A.}}:
\batitle{Border-collision bifurcations in a driven time-delay system}.
\bjtitle{Chaos}
\bvolume{30}(\bissue{2}),
\bfpage{023121}
(\byear{2020})
\doiurl{10.1063/1.5119982}
\end{barticle}
\endbibitem

\bibitem[\protect\citeauthoryear{Suarez and Schopf.}{1988}]{SuarezSchopf1988}
\begin{barticle}
\bauthor{\bsnm{Suarez}, \binits{M.J.}},
\bauthor{\bsnm{Schopf.}, \binits{P.S.}}:
\batitle{A delayed action oscillator for {ENSO}}.
\bjtitle{J. Atmos. Sci.}
\bvolume{45},
\bfpage{3283}--\blpage{3287}
(\byear{1988})
\doiurl{10.1175/1520-0469(1988)045<3283:ADAOFE>2.0.CO;2}
\end{barticle}
\endbibitem

\bibitem[\protect\citeauthoryear{Dijkstra}{2024}]{DIJKSTRA2024133984}
\begin{barticle}
\bauthor{\bsnm{Dijkstra}, \binits{H.A.}}:
\batitle{The role of conceptual models in climate research}.
\bjtitle{Physica D}
\bvolume{457},
\bfpage{133984}
(\byear{2024})
\doiurl{10.1016/j.physd.2023.133984}
\end{barticle}
\endbibitem

\bibitem[\protect\citeauthoryear{Keane et~al.}{2017}]{keane2017climate}
\begin{botherref}
\oauthor{\bsnm{Keane}, \binits{A.}},
\oauthor{\bsnm{Krauskopf}, \binits{B.}},
\oauthor{\bsnm{Postlethwaite}, \binits{C.M.}}:
Climate models with delay differential equations.
Chaos
\textbf{27}(11)
(2017)
\doiurl{10.1063/1.5006923}
\end{botherref}
\endbibitem

\bibitem[\protect\citeauthoryear{Tziperman et~al.}{1998}]{TSCJ2}
\begin{barticle}
\bauthor{\bsnm{Tziperman}, \binits{E.}},
\bauthor{\bsnm{Cane}, \binits{M.A.}},
\bauthor{\bsnm{Zebiak}, \binits{S.E.}},
\bauthor{\bsnm{Xue}, \binits{Y.}},
\bauthor{\bsnm{Blumenthal}, \binits{B.}}:
\batitle{Locking of {El Nino}’s peak time to the end of the calendar year in
  the delayed oscillator picture of {ENSO}}.
\bjtitle{J. Clim.}
\bvolume{11},
\bfpage{2191}--\blpage{2199}
(\byear{1998})
\doiurl{10.1175/1520-0442(1998)011<2191:LOENOS>2.0.CO;2}
\end{barticle}
\endbibitem

\bibitem[\protect\citeauthoryear{Tziperman et~al.}{1994}]{TSCJ94}
\begin{barticle}
\bauthor{\bsnm{Tziperman}, \binits{E.}},
\bauthor{\bsnm{Stone}, \binits{L.}},
\bauthor{\bsnm{Cane}, \binits{M.A.}},
\bauthor{\bsnm{Jarosh}, \binits{H.}}:
\batitle{{El Ni\~no} chaos: Overlapping of resonances between the seasonal
  cycle and the pacific ocean-atmosphere oscillator}.
\bjtitle{Science}
\bvolume{264}(\bissue{5155}),
\bfpage{72}--\blpage{74}
(\byear{1994})
\doiurl{10.1126/science.264.5155.72}
\end{barticle}
\endbibitem

\bibitem[\protect\citeauthoryear{Keane et~al.}{2019}]{KKD19}
\begin{barticle}
\bauthor{\bsnm{Keane}, \binits{A.}},
\bauthor{\bsnm{Krauskopf}, \binits{B.}},
\bauthor{\bsnm{Dijkstra}, \binits{H.A.}}:
\batitle{The effect of state dependence in a delay differential equation model
  for the {El Niño Southern Oscillation}}.
\bjtitle{Phil. Trans. R. Soc. A.}
\bvolume{377}(\bissue{2153}),
\bfpage{20180121}
(\byear{2019})
\doiurl{10.1098/rsta.2018.0121}
\end{barticle}
\endbibitem

\bibitem[\protect\citeauthoryear{Falkena et~al.}{2019}]{Courtney2019}
\begin{barticle}
\bauthor{\bsnm{Falkena}, \binits{S.K.J.}},
\bauthor{\bsnm{Quinn}, \binits{C.}},
\bauthor{\bsnm{Sieber}, \binits{J.}},
\bauthor{\bsnm{Frank}, \binits{J.}},
\bauthor{\bsnm{Dijkstra}, \binits{H.A.}}:
\batitle{Derivation of delay equation climate models using the {Mori-Zwanzig}
  formalism}.
\bjtitle{Proc. R. Soc. A.}
\bvolume{475},
\bfpage{20190075}
(\byear{2019})
\doiurl{10.1098/rspa.2019.0075}
\end{barticle}
\endbibitem

\bibitem[\protect\citeauthoryear{Boutle et~al.}{2007}]{boutle2007nino}
\begin{barticle}
\bauthor{\bsnm{Boutle}, \binits{I.}},
\bauthor{\bsnm{Taylor}, \binits{R.H.}},
\bauthor{\bsnm{R{\"o}mer}, \binits{R.A.}}:
\batitle{El ni{\~n}o and the delayed action oscillator}.
\bjtitle{Am. J. Phys.}
\bvolume{75}(\bissue{1}),
\bfpage{15}--\blpage{24}
(\byear{2007})
\doiurl{10.1119/1.2358155}
\end{barticle}
\endbibitem

\bibitem[\protect\citeauthoryear{Wei and Zhang}{2022}]{wei2022simple}
\begin{barticle}
\bauthor{\bsnm{Wei}, \binits{X.}},
\bauthor{\bsnm{Zhang}, \binits{R.}}:
\batitle{A simple conceptual model for the self-sustained multidecadal amoc
  variability}.
\bjtitle{Geophys. Res. Lett.}
\bvolume{49}(\bissue{14}),
\bfpage{2022}--\blpage{099800}
(\byear{2022})
\doiurl{10.1029/2022GL099800}
\end{barticle}
\endbibitem

\bibitem[\protect\citeauthoryear{Bellman and Cooke}{1963}]{BellmanCooke63}
\begin{bbook}
\bauthor{\bsnm{Bellman}, \binits{R.E.}},
\bauthor{\bsnm{Cooke}, \binits{K.L.}}:
\bbtitle{Differential-Difference Equations}.
\bpublisher{Academic Press},
\blocation{New York, London}
(\byear{1963})
\end{bbook}
\endbibitem

\bibitem[\protect\citeauthoryear{Hale}{1988}]{Hale88}
\begin{bbook}
\bauthor{\bsnm{Hale}, \binits{J.K.}}:
\bbtitle{Asymptotic Behavior of Dissipative Systems}.
\bpublisher{American Mathematical Society},
\blocation{Providence, Rhode Island}
(\byear{1988})
\end{bbook}
\endbibitem

\bibitem[\protect\citeauthoryear{Hale and Verduyn~Lunel}{1993}]{HaleLunel93}
\begin{bbook}
\bauthor{\bsnm{Hale}, \binits{J.K.}},
\bauthor{\bsnm{Verduyn~Lunel}, \binits{S.M.}}:
\bbtitle{Introduction to Functional Differential Equations}.
\bpublisher{Springer},
\blocation{Heidelberg}
(\byear{1993})
\end{bbook}
\endbibitem

\bibitem[\protect\citeauthoryear{Smith}{2011}]{Smith11}
\begin{bbook}
\bauthor{\bsnm{Smith}, \binits{H.}}:
\bbtitle{An Introduction to Delay Differential Equations with Applications to
  the Life Sciences}.
\bpublisher{Springer},
\blocation{New York}
(\byear{2011})
\end{bbook}
\endbibitem

\bibitem[\protect\citeauthoryear{Mallet-Paret}{1988}]{JMP88}
\begin{barticle}
\bauthor{\bsnm{Mallet-Paret}, \binits{J.}}:
\batitle{Morse decompositions for delay-differential equations}.
\bjtitle{J. Differ. Equ.}
\bvolume{72},
\bfpage{270}--\blpage{315}
(\byear{1988})
\doiurl{10.1016/0022-0396(88)90157-X}
\end{barticle}
\endbibitem

\bibitem[\protect\citeauthoryear{Shustin}{1995}]{Shustin1995}
\begin{barticle}
\bauthor{\bsnm{Shustin}, \binits{E.}}:
\batitle{Super-high-frequency oscillations in a discontinuous dynamic system
  with time delay}.
\bjtitle{Israel J. Math.}
\bvolume{90},
\bfpage{199}--\blpage{219}
(\byear{1995})
\doiurl{10.1007/BF02783213}
\end{barticle}
\endbibitem

\bibitem[\protect\citeauthoryear{Akian and Bliman}{2000}]{akian:1998}
\begin{barticle}
\bauthor{\bsnm{Akian}, \binits{M.}},
\bauthor{\bsnm{Bliman}, \binits{P.-A.}}:
\batitle{On super-high frequencies in discontinuous 1st-order
  delay–differential equations}.
\bjtitle{J. Differ. Equ.}
\bvolume{162},
\bfpage{326}--\blpage{358}
(\byear{2000})
\doiurl{10.1006/jdeq.1999.3706}
\end{barticle}
\endbibitem

\bibitem[\protect\citeauthoryear{Fridman et~al.}{1996}]{fridman2002}
\begin{bchapter}
\bauthor{\bsnm{Fridman}, \binits{L.M.}},
\bauthor{\bsnm{Shustin}, \binits{E.I.}},
\bauthor{\bsnm{Fridman}, \binits{E.M.}}:
\bctitle{Steady modes and sliding modes in the relay control systems with time
  delay}.
In: \bbtitle{Proceedings of 35th IEEE Conference on Decision and Control},
vol. \bseriesno{4},
pp. \bfpage{4601}--\blpage{4606}
(\byear{1996}).
\bcomment{IEEE}
\end{bchapter}
\endbibitem

\bibitem[\protect\citeauthoryear{Nussbaum and
  Shustin}{2001}]{nussbaum2001nonexpansive}
\begin{barticle}
\bauthor{\bsnm{Nussbaum}, \binits{R.D.}},
\bauthor{\bsnm{Shustin}, \binits{E.}}:
\batitle{Nonexpansive periodic operators in l1 with application to
  superhigh-frequency oscillations in a discontinuous dynamical system with
  time delay}.
\bjtitle{J. Dyn. Differ. Equ.}
\bvolume{13},
\bfpage{381}--\blpage{424}
(\byear{2001})
\end{barticle}
\endbibitem

\bibitem[\protect\citeauthoryear{{The MathWorks, Inc.}}{2024}]{MATLAB2024b}
\begin{bbook}
\bauthor{\bsnm{{The MathWorks, Inc.}}}:
\bbtitle{{MATLAB} Version R2024b}.
\bpublisher{The MathWorks, Inc.},
\blocation{Natick, Massachusetts}
(\byear{2024}).
\bcomment{The MathWorks, Inc.. Available at
  \url{https://www.mathworks.com/products/matlab.html}}
\end{bbook}
\endbibitem

\bibitem[\protect\citeauthoryear{Kaplan and Yorke}{1974}]{kaplan1974ordinary}
\begin{barticle}
\bauthor{\bsnm{Kaplan}, \binits{J.L.}},
\bauthor{\bsnm{Yorke}, \binits{J.A.}}:
\batitle{Ordinary differential equations which yield periodic solutions of
  differential delay equations}.
\bjtitle{J. Math. Anal}
\bvolume{48}(\bissue{2}),
\bfpage{317}--\blpage{324}
(\byear{1974})
\doiurl{10.1016/0022-247X(74)90162-0}
\end{barticle}
\endbibitem

\bibitem[\protect\citeauthoryear{Cao}{1996}]{cao1996uniqueness}
\begin{barticle}
\bauthor{\bsnm{Cao}, \binits{Y.}}:
\batitle{Uniqueness of periodic solution for differential delay equations}.
\bjtitle{J. Differ. Equ.}
\bvolume{128}(\bissue{1}),
\bfpage{46}--\blpage{57}
(\byear{1996})
\doiurl{10.1006/jdeq.1996.0088}
\end{barticle}
\endbibitem

\bibitem[\protect\citeauthoryear{Chow and
  Walther}{1988}]{chow1988characteristic}
\begin{barticle}
\bauthor{\bsnm{Chow}, \binits{S.-N.}},
\bauthor{\bsnm{Walther}, \binits{H.-O.}}:
\batitle{Characteristic multipliers and stability of symmetric periodic
  solutions of $\dot{x}(t)=g(x(t-1))$}.
\bjtitle{Trans. Am. Math. Soc.}
\bvolume{307}(\bissue{1}),
\bfpage{127}--\blpage{142}
(\byear{1988})
\doiurl{10.2307/2000754}
\end{barticle}
\endbibitem

\bibitem[\protect\citeauthoryear{Nussbaum}{1979}]{nussbaum1979uniqueness}
\begin{barticle}
\bauthor{\bsnm{Nussbaum}, \binits{R.D.}}:
\batitle{Uniqueness and nonuniqueness for periodic solutions of $x'(t)=- g (x
  (t- 1))$}.
\bjtitle{J. Differ. Equ.}
\bvolume{34}(\bissue{1}),
\bfpage{25}--\blpage{54}
(\byear{1979})
\doiurl{10.1016/0022-0396(79)90016-0}
\end{barticle}
\endbibitem

\bibitem[\protect\citeauthoryear{Kennedy and
  Stumpf}{2016}]{kennedy2015multiple}
\begin{barticle}
\bauthor{\bsnm{Kennedy}, \binits{B.}},
\bauthor{\bsnm{Stumpf}, \binits{E.}}:
\batitle{Multiple slowly oscillating periodic solutions for $x'(t)= f (x
  (t-1))$ with negative feedback}.
\bjtitle{Ann. Polon. Math.}
\bvolume{118},
\bfpage{113}--\blpage{140}
(\byear{2016})
\end{barticle}
\endbibitem

\bibitem[\protect\citeauthoryear{Bolduc-St-Aubin and Humphries}{2025}]{GZTn}
\begin{botherref}
\oauthor{\bsnm{Bolduc-St-Aubin}, \binits{S.}},
\oauthor{\bsnm{Humphries}, \binits{A.R.}}:
Resonant Dynamics of a piecewise smooth {Ghil-Zaliapin-Thompson ENSO} model.
In preparation
(2025)
\end{botherref}
\endbibitem

\bibitem[\protect\citeauthoryear{Guo and Wu}{2013}]{guo2013bifurcation}
\begin{bbook}
\bauthor{\bsnm{Guo}, \binits{S.}},
\bauthor{\bsnm{Wu}, \binits{J.}}:
\bbtitle{Bifurcation Theory of Functional Differential Equations}.
\bpublisher{Springer},
\blocation{New York}
(\byear{2013})
\end{bbook}
\endbibitem

\bibitem[\protect\citeauthoryear{Isaacson and Keller}{1994}]{IK66}
\begin{bbook}
\bauthor{\bsnm{Isaacson}, \binits{E.}},
\bauthor{\bsnm{Keller}, \binits{H.B.}}:
\bbtitle{Analysis of Numerical Methods}.
\bpublisher{Dover},
\blocation{New York}
(\byear{1994})
\end{bbook}
\endbibitem

\bibitem[\protect\citeauthoryear{Kuznetsov}{2023}]{kuznetsov1998elements}
\begin{bbook}
\bauthor{\bsnm{Kuznetsov}, \binits{Y.A.}}:
\bbtitle{Elements of Applied Bifurcation Theory},
\bedition{4}th edn.
\bpublisher{Springer},
\blocation{New York}
(\byear{2023})
\end{bbook}
\endbibitem

\bibitem[\protect\citeauthoryear{Leine and {van Campen}}{2002}]{LEINE2002259}
\begin{barticle}
\bauthor{\bsnm{Leine}, \binits{R.I.}},
\bauthor{\bsnm{{van Campen}}, \binits{D.H.}}:
\batitle{Discontinuous bifurcations of periodic solutions}.
\bjtitle{Math. Comput. Model.}
\bvolume{36}(\bissue{3}),
\bfpage{259}--\blpage{273}
(\byear{2002})
\doiurl{10.1016/S0895-7177(02)00124-3}
\end{barticle}
\endbibitem

\bibitem[\protect\citeauthoryear{Hairer et~al.}{1994}]{HNWI}
\begin{bbook}
\bauthor{\bsnm{Hairer}, \binits{E.}},
\bauthor{\bsnm{Wanner}, \binits{G.}},
\bauthor{\bsnm{Nørsett}, \binits{S.P.}}:
\bbtitle{Solving Ordinary Differential Equations I}.
\bpublisher{Springer},
\blocation{Berlin}
(\byear{1994})
\end{bbook}
\endbibitem

\bibitem[\protect\citeauthoryear{Banerjee and
  Grebogi}{1999}]{banerjee1999border}
\begin{barticle}
\bauthor{\bsnm{Banerjee}, \binits{S.}},
\bauthor{\bsnm{Grebogi}, \binits{C.}}:
\batitle{Border collision bifurcations in two-dimensional piecewise smooth
  maps}.
\bjtitle{Phys. Rev. E.}
\bvolume{59}(\bissue{4}),
\bfpage{4052}
(\byear{1999})
\doiurl{10.1103/PhysRevE.59.4052}
\end{barticle}
\endbibitem

\bibitem[\protect\citeauthoryear{Bolduc-St-Aubin}{2023}]{SamMSc}
\begin{botherref}
\oauthor{\bsnm{Bolduc-St-Aubin}, \binits{S.}}:
Analysis and numerical analysis of periodically forced delay differential
  equation conceptual climate models.
Master's thesis,
McGill University
(2023).
\url{https://escholarship.mcgill.ca/concern/theses/8p58pk24t}
\end{botherref}
\endbibitem

\end{thebibliography}
\end{document}